\documentclass{amsart}[12pt]
\AtBeginDocument{\addtocontents{toc}{\protect\setlength{\parskip}{0pt}}}

\usepackage{amssymb, enumitem}
\usepackage{comment}
\usepackage{xcolor}
\usepackage{tikz}
\usepackage{hyperref}
\hypersetup{
   colorlinks,
    linkcolor={red!60!black},
    citecolor={blue!60!black},
    urlcolor={blue!90!black}
}

\DeclareMathOperator{\dist}{dist}
\DeclareMathOperator{\spt}{spt}
\DeclareMathOperator{\hdim}{\dim_H}
\DeclareMathOperator{\pdim}{\dim_P}
\DeclareMathOperator{\diam}{diam}
\newcommand{\G}{\mathbb{G}}
\newcommand{\X}{\mathbb{X}}
\newcommand{\A}{\mathbb{A}}

\newcommand{\Z}{\mathbb{Z}}
\newcommand{\R}{\mathbb{R}}
\newcommand{\N}{\mathbb{N}}
\newcommand{\m}{\mathfrak{m}}

\newtheorem{theorem}{Theorem}[section]
\newtheorem*{theorem*}{Theorem}
\newtheorem{corollary}[theorem]{Corollary}
\newtheorem{proposition}[theorem]{Proposition}
\newtheorem{lemma}[theorem]{Lemma}
\newtheorem{terminology}[theorem]{Terminology}

\theoremstyle{definition}
\newtheorem{definition}[theorem]{Definition}
\newtheorem{remark}[theorem]{Remark}

\newcommand{\e}{\varepsilon}
\numberwithin{equation}{section}

\title{Spread-out Furstenberg sets I: the sticky case}
\author{Tuomas Orponen}
\author{Pablo Shmerkin}

\begin{document}

\begin{abstract}
   We introduce spread-out measures, a notion of local non-concentration near pencils for measures on Grassmannians. We establish an optimal $\delta$-discretized dual Furstenberg estimate for ``sticky spread-out'' shadings. A simple-to-state corollary is the following. Assume that $\nu$ is an Ahlfors $s$-regular measure on $\mathbb{G}(2,1)$ with $s > 0$. Let $E \subset \R^{2}$ be a Borel set with equal Hausdorff and packing dimension.  Then,
   \begin{displaymath} \dim_{\mathrm{H}} \Pi_{V}(E) = \min\{\dim_{\mathrm{H}} E,1\} \quad \text{for $\nu$-a.e. } V \in \mathbb{G}(2,1). \end{displaymath}
  The general $\delta$-discretized estimate for ``sticky spread out'' configurations here is the starting point for a more general theorem in the companion paper of Cohen, Shmerkin, Wang and Zakharov, where the ``stickiness'' hypothesis is removed. For instance, the assumption of equal Hausdorff and packing dimension in the result stated above will be dispensed with in the companion paper.
    \end{abstract}

\maketitle

\tableofcontents

\section{Introduction}

This paper and its companion \cite{SpreadOutFurstGeneral} concern optimal Furstenberg estimates and restricted projection theorems for spread-out families of planes. Roughly speaking, a family of planes is spread-out if it is locally unconcentrated near pencils at most points and scales. The general estimates in \cite{SpreadOutFurstGeneral} are the main motivation for this work. Here we introduce spread-outness and prove the sticky $\delta$-discretized Furstenberg estimate, Theorem \ref{thm:AhlforsFurstenberg}, in which the point set is roughly Ahlfors regular (or \emph{sticky}). The terminology was introduced in the sticky Kakeya problem \cite{WangZahlStickyKakeya}, where stickiness roughly corresponds to tubes in nearby directions ``sticking together''. Cohen, the second author, Wang and Zakharov \cite{SpreadOutFurstGeneral} use this estimate to treat general Frostman and Katz--Tao configurations -- or, in short, remove the "stickiness" hypothesis. The two papers form a single project, presented separately for clarity of exposition and to distinguish the authorship of the sticky estimate from that of the passage to general configurations.

Before stating the $\delta$-discretized result, we record some of its easier-to-state corollaries and applications. We state the corollaries only in the planar case, although their natural higher-dimensional counterparts would also follow from Theorem \ref{thm:AhlforsFurstenberg}. The main reason for stating only the planar cases here is to avoid duplicate work: the higher-dimensional results in full generality will anyway be stated and proven in \cite{SpreadOutFurstGeneral}. The $\delta$-discretized result in Theorem \ref{thm:AhlforsFurstenberg} will be stated in all dimensions, since it is required as a black box in \cite{SpreadOutFurstGeneral}.

\subsection{Restricted projections}

For $V\in\G(d,k)$, let $\Pi_V\colon\R^d\to V$ denote the orthogonal projection onto $V$. The Marstrand-Mattila projection theorem \cite{Mar,MR409774}  is a cornerstone of geometric measure theory. It states that if $E \subset \R^{d}$ is a Borel set then, denoting Lebesgue measure on the Grassmannian $\G(d,k)$ of $k$-planes in $\R^d$ by $\gamma_{d,k}$, one has 
\begin{equation} \label{eq:marstrand-mattila}
    \hdim \Pi_{V}(E) = \min\{\hdim E,k\} \quad\text{ for $\gamma_{d,k}$-a.e. } V \in \G(d,k).
\end{equation}
Here $\hdim$ is Hausdorff dimension. Projection theory seeks to understand the behavior of $\hdim\Pi_{V}(E)$ for $\nu$-typical $V$, when $\nu$ is a highly singular measure on $\G(d,k)$,   perhaps under additional assumptions on $E$. See \cite{Wang26RestrictedProjections} for a recent survey on the topic, including its connection with the Kakeya problem.

A natural question is for which measures $\nu$, in place of $\gamma_{d,k}$, the Marstrand-Mattila estimate \eqref{eq:marstrand-mattila} continues to hold. This question has been extensively studied in the case of length measure on smooth curves in $\G(d,k)$; see, for example, \cite{FasslerOrponen14,KaenmakiOrponenVenieri25,PYZ22,GanGuthMaldague24} for $\G(3,1)$, \cite{FasslerOrponen14,GGGHMW22} for $\G(3,2)$, and \cite{GGW22, Du25RestrictedMarstrand, Wang26RestrictedProjections} for higher-dimensional Grassmannians and further references. By now, the picture in this case is fairly complete in dimension and co-dimension $1$ ($k=1$ or $d-1$), but much less is known for $1<k<d-1$. All the proofs in this line of work rely on the smoothness of the curve, and it is not clear how to extend them to more general measures. In a different direction, Chen \cite{Chen18RestrictedProjections} proved a restricted projection theorem for a class of random measures on $\G(3,1)$; the argument does not extend to $\G(2,1)$.

In particular, up to now there were no genuine restricted Marstrand-type results for measures on $\G(2,1)$. The following Theorem \ref{thm:projection-ADR} is a consequence of our main result. We write $\pdim$ for packing dimension. Recall that a locally finite Borel measure $\mu$ on a metric space $X$ is called \emph{uniformly perfect} if $\diam(\spt \mu) > 0$, and there exist $\beta \in [0,1)$ and $D > 1$ such that
\begin{displaymath} \mu(B(x,r)) \leq \beta \cdot \mu(B(x,Dr)) \end{displaymath}
for all closed balls $B(x,r) \subset X$ such that $\spt \mu \not\subset B(x,Dr)$. In particular, it is easy to see that every non-trivial Ahlfors $s$-regular measure with $s > 0$ is uniformly perfect. (A measure $\mu$ on a metric space $X$ is \emph{Ahlfors $s$-regular} if there is $C\ge 1$ such that $C^{-1}r^s \le \mu(B(x,r)) \le Cr^s$ for all $x\in\spt\mu,\, r\in (0,\diam(\spt\mu)].$)

\begin{theorem} \label{thm:projection-ADR}
    Let $\nu$ be a uniformly perfect measure on $\mathbb{G}(2,1)$. For any Borel set $E\subset\R^2$ with $\hdim(E)=\pdim(E)$,
    \begin{displaymath}
        \hdim \Pi_V(E) = \min\{\hdim E,1\} \quad \text{ for $\nu$-a.e. } V \in \G(2,1).
    \end{displaymath}
\end{theorem}

\begin{remark} The assumption of equal Hausdorff and packing dimension is not necessary, and it will be removed in the companion paper \cite{SpreadOutFurstGeneral}, where more general higher-dimensional results will be presented.
\end{remark}

As shown by Kaufman and Mattila \cite{MR407248}, there are singular measures $\nu$ on $\G(2,1)$, of dimension arbitrarily close to $1$, for which the Marstrand-Mattila estimate fails for some compact set $E$. Both the set $E$ and the measure $\nu$ in these examples are rather special. In particular, $\nu$ is supported on a``well-spaced'' set, that is, a set $\Theta$ such that for many pairs of scales $0<\delta\ll r$, $\Theta\cap B_r$ is contained in a $\delta$-ball for all $r$-balls $B_r$.  Uniform perfectness rules out this type of behavior.

As is usually the case, these results are derived from more general discretized statements, which are far more useful in applications. The discretized sticky Furstenberg estimate is Theorem \ref{thm:AhlforsFurstenberg}; its projection counterpart is Theorem \ref{thm:AhlforsProjection}.

\subsection{Optimal Furstenberg estimates}

Projection theorems are heuristically (and, at the discretized level, formally) special cases of Furstenberg-type estimates. In the classical Furstenberg problem, one is given a set $E\subset\R^2$ and a family of lines $\mathcal{L}$, and one assumes $\hdim(E\cap\ell)\ge s$ for all $\ell\in\mathcal{L}$. The goal is to estimate $\hdim E$ in terms of $s$ and $\hdim\mathcal{L}$. This problem was solved in \cite{23arXiv230808819R}: with $\hdim\mathcal{L}=t$, one has
\begin{displaymath}
 \hdim(E)\ge \min\left\{s+t, \frac{3s+t}{2},s+1\right\},
\end{displaymath}
and this is sharp. In particular, the $\frac{3s+t}{2}$ term is witnessed by a lattice construction, see \cite[p. 10]{Wolff99} or \cite[Section A.1]{MR4745881}. It is a natural problem to ask whether, under some additional conditions on $E\cap\ell$, this term can be removed. (One might in principle also ask if the term $s + 1$ can be removed, but this question seems less interesting: the example witnessing the $s + 1$ term is $E = C \times [0,1]$, where $C \subset [0,1]$ is an arbitrary $s$-dimensional set.)

After point-line duality, the planar Furstenberg theorem can be equivalently stated in the following form: given a set $E\subset\R^2$ with $\hdim(E)=t$ and, for each $x\in E$, a family of lines $\mathcal{L}_x$ through $x$ of dimension $\ge s$, we have
\begin{displaymath}
\hdim\left(\bigcup_{x\in E}\mathcal{L}_x\right)\ge \min\left\{s+t, \frac{3s+t}{2},s+1\right\}.
\end{displaymath}
The set on the left is a subset of the affine Grassmannian $\A(2,1)$ of lines in $\R^2$. Very recently, Cs\"ornyei and Stull \cite[Theorem~1.1]{CsornyeiStull26Hyperplanes} partially generalized the planar discretized incidence estimate to higher dimensions, replacing bundles of lines by bundles of hyperplanes under Frostman non-concentration assumptions, with exponent greater than $d-2$ for the local families of hyperplanes.

It is this dual form that we consider in this paper, and extend to higher dimensions: we replace bundles of lines by bundles of $k$-planes in $\R^d$. The following is the ``Furstenberg version'' of Theorem \ref{thm:projection-ADR}. In the statement, a locally finite Borel measure $\nu$ on a metric space (now $\mathbb{G}(2,1)$) is \emph{$s$-Frostman} if there exists a constant $C \geq 1$ such that $\nu(B(x,r)) \leq Cr^{s}$ for all $r > 0$.

\begin{theorem} \label{thm:Furstenberg-continuous}
    Let $t \in [0,2]$, and $s > 0$. Let $E\subset\R^{2}$ be a Borel set with $\hdim(E)=\pdim(E)=t$. For each $x\in E$, let $\nu_x$ be a uniformly perfect $s$-Frostman measure on $\mathbb{G}(2,1)$, and define $\Theta_x:=\spt(\nu_x)$. Then,
    \begin{displaymath}
        \hdim\left(\bigcup_{x\in E}\{x+\ell:\ell\in \Theta_x\}\right)\ge \min\{t,1\}+s.
    \end{displaymath}
\end{theorem}

\begin{remark} 
    Morally, the projection cases discussed earlier correspond to the special case of Theorem \ref{thm:Furstenberg-continuous} where the family of subspaces $\Theta_x$ is independent of $x$. To make the moral precise, one needs to consider $\delta$-discretized statements. Theorem~4.6 of \cite{SpreadOutFurstGeneral} allows us to remove the assumption $\hdim(E)=\pdim(E)$ in this theorem.
\end{remark}

Theorems \ref{thm:projection-ADR} and \ref{thm:Furstenberg-continuous} follow from a $\delta$-discretized version, Theorem \ref{thm:AhlforsFurstenberg}, by standard discretization and covering arguments similar to \cite[Proposition 6.4]{OrponenShmerkin26}. We omit the details here rather completely, but include some of them in \cite{SpreadOutFurstGeneral}, where more general versions are established. The only point we mention here is that the uniform perfectness of $\nu_{x}$ assumed in Theorem \ref{thm:Furstenberg-continuous} implies that $\nu_{x}$ is a ``spread-out measure'' (Definition \ref{def:perfectSet}) thanks to Proposition \ref{prop:uniform-is-spread-out}. After the spread-outness of the measures $\nu_{x}$ has been established, one may build a "spread-out configuration" (Definition \ref{def:spread-out-cube-family}) inside $\bigcup_{x\in E}\{x+\ell:\ell\in \Theta_x\}$, and then apply Theorem \ref{thm:AhlforsFurstenberg} (following the argument in \cite[Proposition 6.4]{OrponenShmerkin26}) to bound the dimension of $\bigcup_{x\in E}\{x+\ell:\ell\in \Theta_x\}$ from below.

\subsection{Discretized formulations}
\label{s:discretized-formulations}

We now introduce the discretized sticky Furstenberg estimate. For this, we need to introduce a number of definitions.

A \emph{dyadic scale} is an element of $2^{-\N_0}$. For $r\in2^{-\N_0}$ and $m\in\N$, let
\begin{displaymath}
\mathcal D_r(\R^m):=\{r(j+[0,1)^m):j\in\Z^m\}.
\end{displaymath}
For $E\subset\R^m$, let $\mathcal D_r(E)$ be the cubes in $\mathcal D_r(\R^m)$ that meet $E$, write $|E|_r:=|\mathcal D_r(E)|$, and denote the lower-left endpoint of $Q\in\mathcal D_r(\R^m)$ by $z(Q)$. For a point $x\in\R^m$, $\mathcal D_r(x)$ denotes the unique cube in $\mathcal D_r(\R^m)$ containing $x$. For a family $\mathcal J$ of dyadic cubes, write 
\begin{equation}\label{leftCorners} z\mathcal J:=\{z(Q):Q\in\mathcal J\}. \end{equation}

We use the standard notion of \emph{Frostman} sets. For $\delta\in2^{-\N}$, $t\ge0$, and $C\ge1$, a nonempty bounded set $P\subset\R^d$ is \emph{$(\delta,t,C)$-Frostman} if
\begin{displaymath}
|P\cap Q|_\delta\le Cr^t|P|_\delta,
\qquad Q\in\mathcal D_r(\R^d),\quad r\in2^{-\N}\cap[\delta,1].
\end{displaymath}

Identify $\R^d=\R^k\times\R^{d-k}$. For $\theta\in\mathbf G_{d,k}:=\R^{k\times(d-k)}$, let
\begin{displaymath}
V_\theta:=\operatorname{Graph}(\theta^*)
=\{(x,\theta^*x):x\in\R^k\}.
\end{displaymath}
Here $\theta^*$ denotes the transpose of $\theta$. The map $\theta\mapsto V_\theta$ identifies $\mathbf G_{d,k}$ with a Zariski-dense open subset of $\G(d,k)$ (this is the higher dimensional analog of the parametrization of lines by $(a,b)\leftrightarrow y=ax+b$). We similarly identify the corresponding big cell $\mathbf A_{d,k}$ in the affine Grassmannian with $\R^{k\times(d-k)}\times\R^{d-k}$ via
\begin{displaymath}
(\theta,b)\longmapsto V_{\theta,b}:=\{(x,\theta^*x+b):x\in\R^k\},
\end{displaymath}
and use the dyadic covering numbers above in these coordinates.

\begin{definition}[Plane configurations]\label{def:planeConfigurations}
Let $M\in\N$. A \emph{$(\delta,M)$-plane configuration} is a pair $(P,\mathbf\Theta)$, where $P\subset\delta\Z^d\cap[0,1]^d$ is nonempty and, for each $p\in P$,
\begin{displaymath}
\mathbf\Theta(p)\subset\mathcal D_\delta([-1,1]^{k\times(d-k)}),
\qquad M\le|\mathbf\Theta(p)|<2M.
\end{displaymath}
We call $\mathbf\Theta(p)$ the \emph{shading} of $p$ and omit $M$ when it is not relevant. We silently identified $\mathbf{\Theta}(p)$ (a family of dyadic cubes) with its union, so that $\theta\in\mathbf{\Theta}(p)$ means $\theta$ is an element of some cube in $\mathbf{\Theta}(p)$.

The \emph{Furstenberg set associated to $(P,\mathbf\Theta)$} is defined using the lower-left corners $z\mathbf\Theta(p)$ familiar from \eqref{leftCorners}:
\begin{displaymath}
\mathcal F(P,\mathbf\Theta)
:=\{V_{\theta}+p:p\in P, \theta\in z\mathbf\Theta(p)\}\subset\mathbf A_{d,k}.
\end{displaymath}
Allowing all $\theta\in\mathbf\Theta(p)$ instead changes its $\delta$-covering number by at most a factor depending on $d,k$.
\end{definition}

The \emph{model hyperplanes} in $\mathbf G_{d,k}$ are
\begin{displaymath}
H_{[x,y,u]}:=\{\theta\in\mathbf G_{d,k}:x\cdot\theta\cdot y=u\},
\qquad x\in\R^k\setminus\{0\},\quad y\in\R^{d-k}\setminus\{0\},\quad u\in\R.
\end{displaymath}
Each model hyperplane is the intersection with $\mathbf G_{d,k}$ of a Schubert divisor in $\G(d,k)$, and every affine tangent hyperplane to a Schubert divisor at a smooth point 
is a model hyperplane in big cell coordinates. See \S\ref{s:bigcell}.

\begin{definition}[Singular scales and spread-out families]\label{def:intro-spreadout}\label{def:spread-out-cube-family}
Fix $\beta\in(0,1]$ and $C>0$, and let $\mathcal J\subset\mathcal D_\delta(\mathbf G_{d,k})$ be a nonempty finite family. All balls and distances in $\mathbf G_{d,k}$ are Euclidean. For $\mathcal J'\subset\mathcal J$ and $\theta\in\mathcal J'$, a scale $r\in2^{-\N}\cap[\delta,1]$ is \emph{$\beta$-singular for $(\mathcal J',\theta)$} if there exists a $H\in\mathfrak{H}_{d,k}$ such that
\begin{displaymath}
\mathcal J'\cap B(\theta,r)
\subset\{\omega\in\mathbf G_{d,k}:\dist(\omega,H)<\beta r\}.
\end{displaymath}

The family $\mathcal J$ is \emph{$(\delta,\beta,C)_{d,k}$-spread-out} if
\begin{displaymath}
|\mathcal J'|\le C2^{-\beta\mathfrak s  }|\mathcal J|
\end{displaymath}
whenever $\mathcal J'\subset\mathcal J$, $\mathfrak s\in\N_0$, and every $\theta\in\bigcup\mathcal J'$ has at least $\mathfrak s$ scales that are $\beta$-singular for $(\mathcal J',\theta)$. A plane configuration $(P,\mathbf\Theta)$ is \emph{$(\delta,\beta,C,M)$-spread-out} if $\mathbf\Theta(p)$ is $(\delta,\beta,C)_{d,k}$-spread-out for every $p\in P$ and $|\mathbf{\Theta}(p)| \in [M,2M]$ for all $p \in P$. 
\end{definition}

See Section \ref{sec:spread-out} for further background on model hyperplanes, their connection with Schubert cycles, and spread-out sets and measures, and Section \ref{sec:examples-applications} for examples of spread-out families.

For $\delta>0$, $\alpha\ge0$, and $C>0$, a finite nonzero Borel measure $\mu$ on $\R^d$ is \emph{$(\delta,\alpha,C)_d$-Frostman} if
    \begin{displaymath}
        \mu(B(x,r))\le Cr^\alpha
        \qquad(x\in\R^d,\ r\ge\delta).
    \end{displaymath}
    We call $\mu$ \emph{$(\alpha,C)_d$-Frostman} if this holds for every $r>0$. These bounds use the mass of $\mu$ without normalizing it; the subscript $d$ is omitted when clear from context.

    \begin{definition} \label{def:Ahlfors-regular}
        We say that a non-trivial finite measure $\mu$ on $\R^d$ is \emph{$(\delta,\alpha,C)_d$-regular} if it is $(\delta,\alpha,C)_d$-Frostman and
        \begin{displaymath}
            |\spt\mu\cap B(x,R)|_r\le C(R/r)^\alpha
            \qquad(x\in\R^d,\ \delta\le r\le R).
        \end{displaymath}
        A non-empty bounded set $P\subset\R^d$ is \emph{$(\delta,\alpha,C)_d$-regular} if the uniform probability measure $m_{P^{(\delta)}}$ on $P^{(\delta)}:=z\mathcal D_\delta(P)$ is $(\delta,\alpha,C)_d$-regular. The subscript $d$ is omitted when clear from context.
    \end{definition}

    \begin{remark}
        In \cite[Definition~6.1]{OS23}, regularity consists of a Frostman condition and a global covering bound at scale $\delta^{1/2}$ only. In Definition \ref{def:Ahlfors-regular}, we require the covering bound for all $\delta\le r\le R$.
    \end{remark}

We can now state the main $\delta$-discretized result of the paper.
\begin{theorem} \label{thm:AhlforsFurstenberg}
    Fix dimensions $1\le k<d$ and parameters $s \in (0,1]$ and $t \in (0,d]$.

    For any $\eta>0$ there is $\e' = \e'(s,t,\eta) > 0$ such that the following holds for every dyadic $\delta\ll_{s,t,\eta}1$.

    Let $(P,\mathbf{\Theta})$ be a $(\delta,s,\delta^{-\e'},M)_{d,k}$-spread-out configuration where $P \subset [-1,1]^{d}$ is $(\delta,t,\delta^{-\e'})_{d}$-regular. Then the associated Furstenberg set $\mathcal{F}=\mathcal{F}(P,\mathbf{\Theta})$ satisfies
    \begin{equation}\label{form17}
        |\mathcal{F}|_{\delta} \geq \delta^{\eta-\min\{t,d-k\}}\cdot M. \end{equation}
\end{theorem}
In the $(\theta,b)$-coordinates on $\mathbf{A}_{d,k}$, it is immaterial whether we measure the size of $\mathcal{F}$ using the metric on $\A(d,k)$ or the Euclidean metric on $\R^{(k+1)(d-k)}$.


In \cite[Theorem~1.6]{SpreadOutFurstGeneral}, the regularity assumption on $P$ will be replaced by the weaker Frostman or Katz--Tao conditions.

\subsection{Applications}
\label{s:applications}

Theorems \ref{thm:projection-ADR}, \ref{thm:Furstenberg-continuous}  and \ref{thm:AhlforsFurstenberg} have applications to additive-combinatorial problems for Ahlfors regular sets. It follows from Theorem \ref{thm:AhlforsFurstenberg} and a standard argument of Elekes \cite{MR1472816} that if $A\subset[1,2]$ is Ahlfors $s$-regular for some $s\in(0,1/2]$ (that is, $A$ equals the support of an Ahlfors $s$-regular measure), then, for every $\eta>0$ and all sufficiently small $\delta$,
\begin{displaymath}
    |A+A|_\delta\,|A\cdot A|_\delta\ge\delta^\eta|A|_\delta^3.
\end{displaymath}
In particular, $\max\{|A+A|_\delta,|A\cdot A|_\delta\}\ge\delta^\eta|A|_\delta^{3/2}$. See \S\ref{s:discretized-sum-product} for the proof and further discussion.

We also show in \S\ref{s:triple-sums-parabola} that if $A\subset\{(x,x^2):x\in[-1,1]\}$ is Ahlfors $s$-regular for some $s\in(0,1/2]$, then
\begin{displaymath}
    \hdim(A+A+A)=3\hdim(A).
\end{displaymath}
This is conjectured to hold for arbitrary Borel subsets of the parabola, and this is the first non-trivial class of sets for which it has been verified. The companion paper \cite{SpreadOutFurstGeneral} establishes broader sum-product statements, as well as applications to surfaces in the Grassmannian, pinned distances and self-affine sets.

\subsection{Proof strategy and wider connections}

The proof of Theorem \ref{thm:AhlforsFurstenberg} has several points in common with previous work by the authors on Furstenberg sets \cite{OS23,OrponenShmerkin26}. 
\begin{enumerate}[label=(\roman*)]
    \item We first establish the sticky projection estimate, Theorem \ref{thm:AhlforsProjection}, and then use an argument of Hong Wang which first appeared in \cite{OrponenShmerkin26} to deduce the Furstenberg estimate by multiscale analysis and pigeonholing.
    \item The projection estimate is proved by contradiction. Assuming that the estimate fails by a definite amount, we show that, after rescaling, an extremal configuration has a product structure inside typical slabs.
    \item This product structure contradicts an asymmetric sum-product estimate, Proposition \ref{prop:ABC}. Its proof uses Balog--Szemer\'edi--Gowers, Pl\"unnecke--Ruzsa, and an inverse theorem of the second author \cite{Shmerkin19,Shmerkin25Orbita}.
\end{enumerate}

Although not used in a direct way, we were also inspired over the years by Hochman's work on self-similar sets and measures based on additive combinatorics \cite{Hochman14, Hochman20}. At a heuristic level, we see the results of this work as an extension of Hochman's ideas to a setting in which the dynamical structure of self-similar or self-affine measures is replaced by the geometric structure of sticky and spread-out sets. An important conceptual analogy is that both in Hochman's work and in the present paper, an inverse theorem for convolutions is used to obtain a small improvement in the dimension of the object under study, which is then amplified by a multiscale analysis. The mechanisms by which this is achieved are completely different.

The use of spread-out sets, introduced in this paper, accounts for several differences from these earlier arguments. The main properties of spread out objects are:
\begin{enumerate}[label=(\alph*)]
    \item Invariance under scaling and refinement (which are immediate from the definitions).
    \item Invariance under projections; see Lemma \ref{lem:proj-spread-out}.
    \item Invariance under localization to small cubes; see Lemma \ref{lem:spread-out-sub-cube}.
    \item An $L^2$-flattening property; see Theorem \ref{thm:spread-out-implies-flattening}.
\end{enumerate}
See \S\ref{sec:spread-out} for a detailed discussion of spread-outness and its properties, and \S\ref{s:uniformly-perfect-are-spread-out} for a proof that uniformly perfect measures on $\R$ (or $\mathbb{G}(2,1)$) are spread-out.
\subsection{Notation}

We introduce some notation that will be used throughout the paper. 

Given a finite set $P$, we write $m_P$ for the uniform probability measure on $P$.

We write $|P|$ for the cardinality of a finite set $P$. We use the cube and covering notation of Section \ref{s:discretized-formulations}, extending the same definitions to arbitrary $\delta>0$. For $E\subset\R^d$, its \emph{$\delta$-discretization} is $E^{(\delta)}:=z\mathcal D_\delta(E)\subset\delta\Z^d$, so $|E|_\delta=|E^{(\delta)}|$. Up to a factor depending only on the ambient dimension, $|E|_\delta$ is the minimal number of $\delta$-balls needed to cover $E$.

Likewise, the \emph{$\delta$-discretization} of a finite Borel measure $\mu$ on $\R^d$ is defined by $\mu^{(\delta)}(\{z(Q)\})=\mu(Q)$ for $Q\in\mathcal D_\delta(\R^d)$.

The open $r$-neighborhood of a set $E\subset\R^d$ is denoted by $E(r)$. We write $B(0,r)$ for the ball of radius $r$ centered at the origin. For $\lambda>0$, $\lambda Q$ and $\lambda B$ denote the concentric dilations of a cube $Q$ and a ball $B$ by the factor $\lambda$.

If $(X,\mathcal{F},\mu)$ is a measure space, and $f \colon X \to Y$ is a map, we write $f\mu$ for the push-forward measure $(f\mu)(B) := \mu(f^{-1}(B))$ on $Y$.

We write $A \lessapprox_{\delta} B$ to mean $A \leq B(\log(1/\delta))^{C}$, where the implicit exponent $C$ is independent of $\delta$, but may depend on the other ambient parameters, which will be specified as needed. We define $\gtrapprox_{\delta}$ similarly.

By $A\gg_{\text{par}} 1$ we mean that $A$ is sufficiently large, depending on the parameters $\text{par}$. We define $A\ll_{\text{par}} 1$ similarly.

We will often work with dyadic scales. Occasionally, we will assume a scale $\delta$ to be dyadic even if it technically does not need to be; this means that $\delta$ is replaced by the closest dyadic scale. This will sometimes cause a constant factor loss, but this is harmless in all cases where we use this slight abuse of notation.

We will work with dimensions $1\le k\le d$. These will be fixed in each section, and we take them to be constants; in particular, all implicit constants are allowed to depend on $d,k$.

Logarithms are always base $2$ unless otherwise specified.

\subsection{AI declaration}

All the main ideas of the paper were conceived by the authors without AI involvement. For clarity, these ideas include the notion of spread-out measures, and the strategy to prove the main result, Theorem \ref{thm:AhlforsFurstenberg}. At the writing stage, P.S. used several LLMs to draft proofs of technical results in some parts of the paper, and to search for errors and potential simplifications in the arguments. The authors carefully checked all AI-generated text and revised it substantially to ensure correctness and clarity. Any remaining issues are the sole responsibility of the authors.

\subsection{Acknowledgments}

Part of this project was carried out during a visit of PS to the University of Jyväskylä in 2025, supported by a JYU Visiting Fellowship. PS acknowledges support from a Simons Fellowship and an NSERC Discovery Grant. TO is supported by the European Research Council (ERC) under the European Union's Horizon Europe research and innovation programme (grant agreement No 101087499). This material is partially based upon work supported by the Swedish Research Council under grant no. 2021-06594 while the authors were in residence at Institut Mittag-Leffler in Djursholm, Sweden, during the semester \emph{Interactions between fractal geometry, harmonic analysis, and dynamics in 2026}.

\section{Spread-out measures and their basic properties}
\label{sec:spread-out}

\subsection{Big cells in the Grassmannian}

\label{s:bigcell}

For $0\le k\le d$, let $\G(d,k)$ and $\A(d,k)$ denote the Grassmannian of $k$-dimensional linear subspaces of $\R^d$ and the affine Grassmannian of $k$-dimensional affine subspaces of $\R^d$, respectively.

The matrix space $\mathbf{G}_{d,k}=\R^{k\times(d-k)}$ can be identified with a Zariski-dense and open subset of $\G(d,k)$ via the embedding
\begin{equation}\label{def:graph}
    \theta \to V_{\theta} := \operatorname{Graph}(\theta^*) =  \{ (x, \theta^* x) : x \in \R^{k}\},
\end{equation}
where $\theta^*$ is the transpose of $\theta$. Note that $V_{\theta}^\perp$ equals $\ker(\pi_\theta)$, where $\pi_{\theta} \colon \R^{k} \times \R^{d - k} \to \R^{k}$ is the map $\pi_\theta(x,y)=x+\theta\cdot y$.  We think of $\mathbf{G}_{d,k}$ as a ``big cell'' in the Grassmannian. 

For the remainder of this section, assume that $1\le k<d$.

Let $\X_{k,d-k}=\bigl((\R^{k}\setminus \{0\})\times (\R^{d-k}\setminus \{0\})\times \R\bigr)/\sim$, where $\sim$ is the equivalence relation given by $(x,y,t)\sim (x',y',t')$ if there exist $\lambda,\mu\in\R\setminus\{0\}$ such that $x'=\lambda x$, $y'=\mu y$, and $t'=\lambda\mu t$. We write $[x,y,t]$ for the equivalence class of $(x,y,t)$. Given $[x,y,t]\in \X_{k,d-k}$, let 
\begin{displaymath}
    H_{[x,y,t]} := \bigl\{ \theta\in \mathbf{G}_{d,k}: x\cdot \theta\cdot y = t \bigr\}.    
\end{displaymath}
This is well-defined on equivalence classes. We denote the family of all such hyperplanes by $\mathfrak{H}_{d,k} := \{ H_{[x,y,t]} : [x,y,t] \in \X_{k,d-k}\}$.

Note that if we identify $\mathbf{G}_{d,k}$ with $\R^{k(d-k)}$, then $\mathfrak{H}_{d,k}\subset \A(k(d-k),k(d-k)-1)$ consists of affine hyperplanes in $\R^{k (d-k)}$. On the other hand, as a subset of $\G(d,k)$, the set $H_{[x,y,t]}$ corresponds to the Schubert divisor
\begin{equation} \label{eq:schubert-variety}
    \{\theta\in \mathbf{G}_{d,k}: \dim(V_{\theta}\cap W_{[x,y,t]})\ge 1\},
\end{equation}
where 
\begin{displaymath}
    W_{[x,y,t]}:=\bigl\{(sx,v)\in \R^k\times\R^{d-k}: y\cdot v=ts\bigr\}\in \G(d,d-k).
\end{displaymath}
At every smooth point of a Schubert divisor, its affine tangent hyperplane in big-cell coordinates is in $\mathfrak{H}_{d,k}$. The model hyperplanes are affine local models for the Schubert divisors appearing in the global non-concentration hypothesis of He's projection theorem \cite{He20}.

The set $\mathbf{G}_{d,k}$ has two natural metrics: the one induced from the natural metric on $\G(d,k)$ invariant under the action of the orthogonal group, and the one induced from the Euclidean metric on $\R^{k\times (d-k)}$. The two metrics induce the same topology and are bi-Lipschitz equivalent on compact subsets of $\mathbf{G}_{d,k}$. Unless otherwise specified, we will always use the Euclidean metric on $\mathbf{G}_{d,k}$.

These considerations extend to the affine Grassmannian $\A(d,k)$ as well. The big cell $\mathbf{A}_{d,k}$ of $\A(d,k)$ is the set of affine $k$-planes $V_{\theta,b} := \{(x,\theta^* x+b) : x \in \R^{k}\}$, with $\theta\in \mathbf{G}_{d,k}$ and $b\in \R^{d-k}$. The big cell is once again Zariski-dense and open. We identify $\mathbf{A}_{d,k}\cong \R^{(d-k)(k+1)}$ by identifying the plane $V_{\theta,b}$ with the pair $(\theta,b) \in \R^{k(d - k) + (d - k)}$. We will use a dyadic system $\{\mathcal{D}_{\delta}(\mathbf{A}_{d,k}) : \delta \in 2^{-\Z}\}$ on $\mathbf{A}_{d,k}$ which is the standard dyadic system on $\R^{(d - k)(k + 1)}$ under the identification above.  The metrics on $\mathbf{A}_{d,k}$ inherited from $\A(d,k)$ and $\R^{(k+1)(d-k)}$ are bi-Lipschitz equivalent on compact subsets of $\mathbf{A}_{d,k}$.

For $\theta\in\mathbf G_{d,k}$ and $(x,y)\in\R^k\times\R^{d-k}$, write $\rho_\theta(x,y):=y-\theta^*x=\pi_{-\theta^*}(y,x)$. The affine coordinates satisfy
\begin{equation}\label{eq:affine-incidence-coordinates}
    V_\theta+(x,y)\longleftrightarrow(\theta,\rho_\theta(x,y)).
\end{equation}
On fixed bounded sets, this coordinate map is Lipschitz, so changing point and direction representatives by $O(r)$ changes the affine coordinates by $O(r)$. Since the dyadic system on $\mathbf A_{d,k}$ is the product system, for $r>0$, any finite $\Theta\subset r\Z^{k\times(d-k)}$, and sets $P_\theta\subset\R^d$,
\begin{equation}\label{eq:affine-incidence-counting}
    \left|\bigcup_{\theta\in\Theta}\{V_\theta+p:p\in P_\theta\}\right|_r
    =\sum_{\theta\in\Theta}|\rho_\theta(P_\theta)|_r.
\end{equation}

Both $\G(d,k)$ and bounded subsets of $\A(d,k)$ can be covered by finitely many compact subsets of big cells (or more properly by finitely many translates of a compact subset of the big cell by elements of the orthogonal group). This allows us to reduce many questions about $\G(d,k)$ and $\A(d,k)$ to questions about $\mathbf{G}_{d,k}$, which is a vector space and hence more convenient to work with. In particular, as mentioned, $\mathfrak{H}_{d,k}$ is a subset of the affine hyperplanes on $\mathbf{G}_{d,k}$, which facilitates the definition and use of the notion of spread-outness introduced in the next subsection.

\subsection{Spread-out sets and measures}

We will consider subsets of $\mathbf{G}_{d,k}$ which, at almost all points and almost all scales, are not too concentrated near planes $H\in \mathfrak{H}_{d,k}$. Related local non-concentration conditions have appeared in Diophantine approximation, notably the absolute decay condition of Kleinbock, Lindenstrauss and Weiss \cite[\S 2]{KLW04}. Here we test only $H\in\mathfrak{H}_{d,k}$ and control the mass of subsets according to their number of singular scales.

\begin{definition}[$\beta$-singular scale]
    Let $\beta\in (0,1]$, let $r \in (0,1]$, let $\Theta \subset \mathbf{G}_{d,k}$, and let $\theta \in \Theta$. We say that $r$ is a \emph{$\beta$-singular scale} for $(\Theta,\theta)$ if there exists $H\in\mathfrak{H}_{d,k}$ such that $\Theta\cap B(\theta,r)\subset H(\beta r)$. We think of this as saying that $\Theta\cap B(\theta,r)$ has generalized diameter at most $\beta r$ in the directions of the hyperplanes $H\in\mathfrak{H}_{d,k}$.

    For $\Theta \subset \R^{k(d-k)}$ and $\delta,\Delta \in 2^{-\N}$ with $\delta \le \Delta$, we denote
    \begin{displaymath}
        \mathfrak{s}_{d,k}(\theta \mid \Theta,\beta,\delta \to \Delta)
        =
        \bigl|\{r \in 2^{-\N} \cap [\delta,\Delta] : r \text{ is a $\beta$-singular scale for } (\Theta,\theta)\}\bigr|.
    \end{displaymath}
    We abbreviate $\mathfrak{s}_{d,k}(\theta \mid \Theta,\beta,\delta) := \mathfrak{s}_{d,k}(\theta \mid \Theta,\beta,\delta \to 1)$ and, when $\delta$ is clear, $\mathfrak{s}_{d,k}(\theta \mid \Theta,\beta) := \mathfrak{s}_{d,k}(\theta \mid \Theta,\beta,\delta)$. We also omit the subscripts $d,k$ when they are clear from context.

    If $\mathcal{J} \subset \mathcal{D}_{\rho}(\mathbf{G}_{d,k})$ is a family of dyadic cubes for some dyadic scale $\rho$, and $\theta \in \bigcup \mathcal{J}$, we write $\mathfrak{s}_{d,k}(\theta \mid \mathcal{J},\beta,\delta \to \Delta):=\mathfrak{s}_{d,k}(\theta \mid \cup \mathcal{J},\beta,\delta \to \Delta)$.
    If $\mathbf{J} \in \mathcal{J}$ is itself a cube, we write
    \begin{displaymath}
        \mathfrak{s}_{d,k}(\mathbf{J} \mid \mathcal{J},\beta,\delta \to \Delta)
        :=
        \inf_{\theta \in \mathbf{J}}\mathfrak{s}_{d,k}(\theta \mid \mathcal{J},\beta,\delta \to \Delta).
    \end{displaymath}
    Thus, when no confusion is likely, we silently identify a family of dyadic cubes with its union in singular-scale counts.
\end{definition}

\begin{definition}[$(\delta,\beta,C)_k$-spread-out measure]\label{def:spread-out-measure-k}
    Let $\delta \in 2^{-\N}$ and $C>0$, $\beta \in (0,1]$. A finite nonzero Borel measure $\nu$ on $\R^k$, $k \geq 1$, is called \emph{$(\delta,\beta,C)_k$-spread-out} if the following holds. For a Borel set $E \subset \R^k$, a point $x \in E$, and dyadic scales $\delta \leq \Delta \leq 1$, write
    \begin{displaymath}
        \mathfrak{s}_k(x \mid E,\beta,\delta \to \Delta)
        :=
        \bigl|\{r \in 2^{-\N} \cap [\delta,\Delta] : \exists H \in \A(k,k-1) \text{ such that } E \cap B(x,r) \subset H(\beta r)\}\bigr|.
    \end{displaymath}
    Let $\mathfrak{s} \in \N_0$, and let $\Theta \subset \spt \nu$ be a Borel set such that
    \begin{displaymath}
        \mathfrak{s}_k(x \mid \Theta,\beta,\delta \to 1) \geq \mathfrak{s} \text{ for all } x \in \Theta.
    \end{displaymath}
    Then,
    \begin{displaymath}
        \nu(\Theta) \leq C \cdot 2^{-\beta \mathfrak{s}}.
    \end{displaymath}
    We say that a nonempty bounded set $P \subset \R^k$ is \emph{$(\delta,\beta,C)_k$-spread-out} if $\mathcal D_\delta(P)$ satisfies Definition \ref{def:spread-out-cube-family}, with all affine hyperplanes in $\R^k$ in place of the model hyperplanes.
\end{definition}

Next, we extend this definition to measures on $\mathbf{G}_{d,k}$.
\begin{definition}[$(\delta,\beta,C)_{d,k}$-spread-out measure]\label{def:perfectSet}
    Let $\delta \in 2^{-\N}$, $C>0$, and $\beta\in(0,1]$. A finite nonzero Borel measure $\mu$ on $\mathbf{G}_{d,k}$ is called \emph{$(\delta,\beta,C)_{d,k}$-spread-out} if the following holds. Let $\mathfrak{s} \in \N_0$, and let $\Theta \subset \spt \mu$ be a Borel set such that
    \begin{displaymath}
        \mathfrak{s}_{d,k}(\theta \mid \Theta,\beta,\delta \to 1) \geq \mathfrak{s} \text{ for all } \theta \in \Theta.
    \end{displaymath}
    Then,
    \begin{displaymath}
        \mu(\Theta) \leq C \cdot 2^{-\beta \mathfrak{s}}.
    \end{displaymath}
    We say that a nonempty bounded set $\Theta \subset \mathbf{G}_{d,k}$ is \emph{$(\delta,\beta,C)_{d,k}$-spread-out} if $\mathcal D_\delta(\Theta)$ is $(\delta,\beta,C)_{d,k}$-spread-out in the sense of Definition \ref{def:spread-out-cube-family}.

    We omit the subscripts $d,k$ when they are clear from context.
\end{definition}

Recall that spread-out families of dyadic cubes were defined in Definition \ref{def:spread-out-cube-family}. In the singular-scale notation, its hypothesis is $\mathfrak{s}_{d,k}(\mathbf J\mid\mathcal J',\beta,\delta\to1)\ge\mathfrak s$ for every $\mathbf J\in\mathcal J'$.

In Section \ref{s:uniformly-perfect-are-spread-out}, we prove that uniformly perfect measures on $\mathbf{G}_{2,1}\cong\R$ are spread-out. Self-similar measures on $\mathbf{G}_{d,k}$ not supported on hyperplanes are spread-out by \cite[Corollary~4.4]{SpreadOutFurstGeneral}.

\begin{remark}\label{rem:spread-out-k-vs-kplus1k}
    Under the natural identification $\mathbf{G}_{k+1,k}=\R^{k\times 1}\cong\R^k$, the notions of $(\delta,\beta,C)_k$-spread-out and $(\delta,\beta,C)_{k+1,k}$-spread-out coincide. Indeed, for $(d,k)=(k+1,k)$ the parameter $y$ in $\X_{k,1}$ is a nonzero scalar, so the model hyperplanes are given by $H_{[x,y,t]}=\{a\in\R^k:a y\cdot x=t\}=\{a\in\R^k:a\cdot(yx)=t\}$, which are precisely the affine hyperplanes in $\R^k$.
\end{remark}

\begin{remark}
    Note that if $\mu$ is $(\delta,\beta,C)_{d,k}$-spread-out, then it is also $(\Delta,\beta',C')_{d,k}$-spread-out for all $\Delta \in 2^{-\N} \cap [\delta,1]$, all $0<\beta'\le\beta$ and all $C'\ge C$. We will often use this without explicit mention.
\end{remark}

The next lemma captures a key property of the spread-out property: it is preserved under projections. In fact, it is this crucial property that motivates the definition of $(\delta,\beta,C)_{d,k}$-spread-out measures in the first place.
\begin{lemma} \label{lem:proj-spread-out}
    Let $\mu$ be a $(\delta,\beta,C)_{d,k}$-spread-out measure on $\mathbf{G}_{d,k}$. For $b\in\R^{d-k}$ with $|b|=1$, define $L_b(\theta):=\theta\cdot b$, and let $\nu_b:= L_b\mu$. Then $\nu_b$ is $(\delta,\beta,C)_k$-spread-out.
\end{lemma}
\begin{proof}
    Let $\Theta \subset \spt \nu_b$ be a Borel set such that $\mathfrak{s}_k(x \mid \Theta,\beta,\delta \to 1) \geq \mathfrak{s}$ for all $x \in \Theta$. We need to show that $\nu_b(\Theta) \leq C 2^{-\beta \mathfrak{s}}$. Write $F:=L_b^{-1}(\Theta)\cap \spt \mu$. Then $\nu_b(\Theta)=\mu(F)$ by definition of push-forward measure. We claim that $\mathfrak{s}_{d,k}(\theta \mid F,\beta,\delta \to 1) \geq \mathfrak{s}$ for all $\theta \in F$.
    To see this, fix $\theta \in F$, let $x:=L_b(\theta)\in \Theta$, and let $r \in 2^{-\N}\cap[\delta,1]$ be a scale which is $\beta$-singular for $(\Theta,x)$ in the sense of Definition \ref{def:spread-out-measure-k}. Thus there exists an affine hyperplane $H\in\A(k,k-1)$ such that
    \begin{displaymath}
        \Theta\cap B(x,r)\subset H(\beta r).
    \end{displaymath}
    Now $L_b$ is $1$-Lipschitz, so
    \begin{displaymath}
        L_b(F\cap B(\theta,r))\subset \Theta\cap B(x,r)\subset H(\beta r).
    \end{displaymath}
    Write $H=\{z\in\R^k:z\cdot e=t\}$ with $|e|=1$. Then 
    \begin{displaymath}  
        L_b^{-1}(H)=\{\omega\in\R^{k \times (d-k)}:(\omega \cdot b) \cdot e=t\}=H_{[e,b,t]},
    \end{displaymath} 
    which is an element of $\mathfrak{H}_{d,k}$. If $\omega\in F\cap B(\theta,r)$, then $L_b(\omega)\in H(\beta r)$, so $|(\omega \cdot b)\cdot e-t|<\beta r$. Since the point $\omega-((\omega \cdot b)\cdot e-t)\,e\otimes b$ lies in $L_b^{-1}(H)$, we get
    \begin{displaymath}
        \dist(\omega,L_b^{-1}(H))\le |(\omega \cdot b)\cdot e-t|\,\|e\otimes b\|<\beta r.
    \end{displaymath}
    Thus $\omega\in \bigl(L_b^{-1}(H)\bigr)(\beta r)$, and hence
    \begin{displaymath}
        F\cap B(\theta,r)\subset \bigl(L_b^{-1}(H)\bigr)(\beta r).
    \end{displaymath}
    This proves the claim.

    Since $\mu$ is $(\delta,\beta,C)_{d,k}$-spread-out, it follows that $\mu(F)\le C 2^{-\beta \mathfrak{s}}$. Recalling that $\mu(F)=\nu_b(\Theta)$, this is the desired conclusion.
\end{proof}

\subsection{From local to global and global to local non-concentration}

Spread-outness is a notion of local non-concentration near model hyperplanes. The next lemma shows that this local information implies a global Frostman type bound for neighborhoods of $H\in\mathfrak{H}_{d,k}$.
\begin{lemma} \label{lem:spread-out-global-non-concentration}
    Let $\delta \in 2^{-\N}$, let $\beta \in (0,1]$, and let $\mu$ be a finite nonzero Borel measure on $\mathbf{G}_{d,k}$ which is $(\delta,\beta,C)_{d,k}$-spread-out. Then for every $[x,y,t] \in \X_{k,d-k}$ and every dyadic scale $r \in 2^{-\N} \cap [\delta,1]$,
    \begin{displaymath}
        \mu\bigl(H_{[x,y,t]}(r)\bigr) \lesssim_{\beta} C \cdot r^{\beta}.
    \end{displaymath}
\end{lemma}
\begin{proof}
    Fix $[x,y,t] \in \X_{k,d-k}$ and $r \in 2^{-\N} \cap [\delta,1]$, and write $\Theta := \spt \mu \cap H_{[x,y,t]}(r)$. The spread-out condition with $\mathfrak{s}=0$ gives $\mu(\mathbf{G}_{d,k})\le C$. If $r>\beta$, then $\mu(\Theta)\le \mu(\mathbf{G}_{d,k}) \le \beta^{-\beta}r^\beta C$, so the claim is immediate. Thus we may assume that $r\le \beta$.

    Let $\theta \in \Theta$, and let $\rho \in 2^{-\N} \cap [r/\beta,1]$. Then
    \begin{displaymath}
        \Theta \cap B(\theta,\rho) \subset H_{[x,y,t]}(r) \subset H_{[x,y,t]}(\beta \rho),
    \end{displaymath}
    so $\rho$ is a $\beta$-singular scale for $(\Theta,\theta)$. Applying the $(\delta,\beta,C)_{d,k}$-spread-out property of $\mu$, we obtain
    \begin{displaymath}
        \mu(\Theta) \lesssim_{\beta} C 2^{-\beta\lfloor-\log(r/\beta)\rfloor}
        \lesssim_{\beta} C r^{\beta}.
    \end{displaymath}
    
    
\end{proof}

Next, we show that if a measure $\mu$ is spread-out at scale $\delta$, then for all intermediate scales $\Delta$ one can find a $\Delta$-cube such that the renormalized restriction of $\mu$ to the cube is spread-out at scale $\delta/\Delta$.
\begin{lemma}\label{lem:spread-out-sub-cube}
Let $\mu$ be a $(\delta,\beta,C)_{d,k}$-spread-out probability measure on $\mathbf G_{d,k}$, and let $\Delta\in2^{-\N}\cap[\delta,1]$. There is a family $\mathcal Q\subset\mathcal D_\Delta(\mathbf G_{d,k})$ with $\mu(\bigcup\mathcal Q)\ge1/2$ and $\mu(Q)>0$ for $Q\in\mathcal Q$ such that $S_Q\,\frac{\mu|_Q}{\mu(Q)}$ is $(\delta/\Delta,\beta,C')_{d,k}$-spread-out for every $Q\in\mathcal Q$, where $C'\lesssim_{d,k}C(1+\log(\Delta/\delta))$ and $S_Q$ is the homothety mapping $Q$ onto $[0,1]^{k\times(d-k)}$.

If $\Theta\subset\mathbf G_{d,k}$ is a bounded $(\delta,\beta,C)_{d,k}$-spread-out set, there is a family $\mathcal Q\subset\mathcal D_\Delta(\Theta)$ with $\sum_{Q\in\mathcal Q}|\Theta\cap Q|_\delta\ge|\Theta|_\delta/2$ such that $S_Q(\Theta\cap Q)$ is $(\delta/\Delta,\beta,C')_{d,k}$-spread-out for every $Q\in\mathcal Q$.
\end{lemma}

\begin{proof}
Suppose that the cubes $Q$ with $\mu(Q)>0$ for which $S_Q\,\frac{\mu|_Q}{\mu(Q)}$ is not $(\delta/\Delta,\beta,C')_{d,k}$-spread-out have total $\mu$-mass greater than $1/2$. For each such $Q$, pull a counterexample back under $S_Q$ to obtain a Borel set $E_Q\subset Q\cap\spt\mu$ and an integer $\mathfrak s_Q\in[1,\log(\Delta/\delta)+1]$ such that
\begin{equation}\label{eq:spread-out-sub-cube-counterexamples}
\begin{aligned}
    \mathfrak s_{d,k}(x\mid E_Q,\beta,\delta\to\Delta)&\ge\mathfrak s_Q,
        &&x\in E_Q,\\
    \mu(E_Q)&>C'2^{-\beta\mathfrak s_Q}\mu(Q).
\end{aligned}
\end{equation}
Pigeonholing the values of $\mathfrak s_Q$ with weights $\mu(Q)$, followed by the choice of one class from a bounded coloring of the $\Delta$-cubes, yields an integer $\mathfrak s\ge1$ and a family $\mathcal Q\subset\mathcal D_\Delta(\mathbf G_{d,k})$ such that
\begin{equation}\label{form19}
    \mathfrak s_Q=\mathfrak s\quad(Q\in\mathcal Q),
    \qquad
    \sum_{Q\in\mathcal Q}\mu(Q)\gtrsim_{d,k}(1+\log(\Delta/\delta))^{-1},
\end{equation}
and distinct cubes in $\mathcal Q$ are $2\Delta$-separated. Set $E:=\bigcup_{Q\in\mathcal Q}E_Q$. By the $2\Delta$-separation,
 \eqref{eq:spread-out-sub-cube-counterexamples} yields
\begin{displaymath}
    \mathfrak s_{d,k}(x\mid E,\beta,\delta\to1)\ge\mathfrak s,
    \qquad x\in E.
\end{displaymath}
The spread-outness of $\mu$ and \eqref{form19} now imply
\begin{displaymath}
    C2^{-\beta\mathfrak s}
    \ge\mu(E)
    =\sum_{Q\in\mathcal Q}\mu(E_Q)
    >C'2^{-\beta\mathfrak s}\sum_{Q\in\mathcal Q}\mu(Q)
    \gtrsim_{d,k}\frac{C'2^{-\beta\mathfrak s}}{1+\log(\Delta/\delta)}.
\end{displaymath}
Choosing $C'\lesssim_{d,k}C(1+\log(\Delta/\delta))$ appropriately yields a contradiction. For the set conclusion, use subfamilies of $\mathcal D_\delta(\Theta\cap Q)$ in place of $E_Q$ and normalized cube counts in place of $\mu$. The map $S_Q$ sends these $\delta$-cubes to dyadic $(\delta/\Delta)$-cubes, so the argument applies without further changes.
\end{proof}

\subsection{Stability of Spread-Out Measures}

We conclude this section with several technical lemmas of the type ``spread-out measures are stable under elementary operations''.

\begin{lemma}[Subset stability]\label{lem:spread-out-subset}
    If $\mu$ is $(\delta,\beta,C)_{d,k}$-spread-out and $A \subset \mathbf{G}_{d,k}$ is a Borel set with $\mu(A)>0$, then $\mu|_A$ is $(\delta,\beta,C)_{d,k}$-spread-out. Consequently, the normalized restriction $\mu(A)^{-1}\mu|_A$ is $(\delta,\beta,C/\mu(A))_{d,k}$-spread-out.

    If $\Theta$ is a bounded $(\delta,\beta,C)_{d,k}$-spread-out set and $\varnothing\ne\Theta'\subset\Theta$, then $\Theta'$ is $(\delta,\beta,C|\Theta|_\delta/|\Theta'|_\delta)_{d,k}$-spread-out. The same holds for subfamilies of dyadic $\delta$-cubes, with cardinalities in place of covering numbers.
\end{lemma}
\begin{proof}
    This follows directly from the definitions.
\end{proof}

\begin{lemma}[Discretization stability]\label{lem:spread-out-discretization}
    Let $\delta \in 2^{-\N}$.
    \begin{enumerate}[label=\textup{(\roman*)}, ref=(\roman*)]
        \item \label{it:spread-out-discretization-measures} If $\mu$ is $(\delta,\beta,C)_{d,k}$-spread-out, then $\mu^{(\delta)}$ is $(\delta,\beta/4,C')_{d,k}$-spread-out with $C' \lesssim_{d,k,\beta} C$.
        \item \label{it:spread-out-discretization-cubes} If $\Theta\subset\mathbf G_{d,k}$ is a bounded $\delta$-separated $(\delta,\beta,C)_{d,k}$-spread-out set, then $m_\Theta$ is $(\delta,\beta/4,C')_{d,k}$-spread-out with $C'\lesssim_{d,k,\beta}C$. Conversely, if $\Theta\subset\delta\Z^{k(d-k)}$ is finite and nonempty and $m_\Theta$ is $(\delta,\beta,C)_{d,k}$-spread-out, then $\Theta$ is $(\delta,\beta,C)_{d,k}$-spread-out.
    \end{enumerate}
\end{lemma}
\begin{proof}
    Let $n:=k(d-k)$ and $K:=\max\left\{1,\left\lceil \log\frac{4\sqrt{n}}{\beta}\right\rceil\right\}$.

    We first record the common geometric input. Let $E \subset \R^n$, and let $F \subset \bigcup\mathcal D_\delta(E)$. Assume that
    \begin{displaymath}
        \mathfrak{s}_{d,k}(\eta \mid E,\beta/4,\delta \to 1) \geq \mathfrak{s}, \qquad \eta \in E.
    \end{displaymath}
    Then every $x \in F$ satisfies
    \begin{equation}\label{tuomas1}
        \mathfrak{s}_{d,k}(x \mid F,\beta,\delta \to 1) \geq \mathfrak{s}-K.
    \end{equation}
    Indeed, fix $x \in F$, and choose $\eta_x \in E\cap\mathcal D_\delta(x)$. Let $r \in 2^{-\N}\cap[\delta,1]$ be a $(\beta/4)$-singular scale for $(E,\eta_x)$ with $r\ge 4\sqrt{n}\delta\beta^{-1}$. Then there exists $[u,v,t]\in \X_{k,d-k}$ such that
    \begin{displaymath}
        E \cap B(\eta_x,r)\subset H_{[u,v,t]}((\beta/4)r).
    \end{displaymath}
    We claim that $r/2$ is a $\beta$-singular scale for $(F,x)$. Indeed, let $y \in F \cap B(x,r/2)$, and choose $\eta_y \in E\cap\mathcal D_\delta(y)$. We have
    \begin{displaymath}
        |\eta_y-\eta_x| \leq |\eta_y-y| + |y-x| + |x-\eta_x|
        \leq 2\sqrt{n}\delta + r/2
        \leq r.
    \end{displaymath}
    Hence $\eta_y \in E \cap B(\eta_x,r)$, so $\eta_y \in H_{[u,v,t]}((\beta/4)r)$. Since $|\eta_y-y|\leq \sqrt{n}\delta \leq (\beta/4)r$, it follows that
    \begin{displaymath}
        y \in H_{[u,v,t]}\bigl((\beta/2)r\bigr)=H_{[u,v,t]}\bigl(\beta(r/2)\bigr).
    \end{displaymath}
    Therefore
    \begin{displaymath}
        F \cap B(x,r/2)\subset H_{[u,v,t]}\bigl(\beta(r/2)\bigr),
    \end{displaymath}
    which proves the claim.

    The only $(\beta/4)$-singular scales for $(E,\eta_x)$ which may fail to produce dyadic $\beta$-singular scales in $[\delta,1]$ for $(F,x)$ are those below $(4\sqrt{n}/\beta)\delta$, and there are at most $K$ such scales. This proves the lower bound \eqref{tuomas1}.

    To prove \ref{it:spread-out-discretization-measures}, let $E \subset \spt \mu^{(\delta)}$ be a Borel set such that $\mathfrak{s}_{d,k}(\eta \mid E,\beta/4,\delta \to 1) \geq \mathfrak{s}$ for all $\eta \in E$. Write $F := \Bigl(\bigcup_{\eta \in E}(\eta+[0,\delta)^n)\Bigr) \cap \spt \mu$. By construction, $\mu^{(\delta)}(E)=\mu(F)$. Now \eqref{tuomas1} shows that every $x \in F$ satisfies $\mathfrak{s}_{d,k}(x \mid F,\beta,\delta \to 1) \geq \mathfrak{s}-K$. Set $C' := C\max\{2^{\beta K},2^{(\beta/4)K}\}\lesssim_{d,k}\beta^{-1}C$.
    Applying the spread-out estimate at level $\mathfrak{s}-K$ when $\mathfrak{s}>K$, and at level $0$ otherwise, we get
    \begin{displaymath}
        \mu^{(\delta)}(E)=\mu(F)\leq C'2^{-(\beta/4)\mathfrak{s}}.
    \end{displaymath}
    This proves \ref{it:spread-out-discretization-measures}.

    To prove \ref{it:spread-out-discretization-cubes}, let $E\subset\Theta$ satisfy $\mathfrak s_{d,k}(\eta\mid E,\beta/4,\delta\to1)\ge\mathfrak s$ for all $\eta\in E$. Apply \eqref{tuomas1} with $F=\bigcup\mathcal D_\delta(E)$. Since $\delta$-separation bounds the number of points of $\Theta$ in each $\delta$-cube by $O_n(1)$, the spread-outness of $\Theta$ implies
    \begin{displaymath}
        m_\Theta(E)=\frac{|E|}{|\Theta|}
        \lesssim_n\frac{|E|_\delta}{|\Theta|_\delta}
        \le C2^{-\beta\max\{\mathfrak s-K,0\}}
        \lesssim_{d,k,\beta}C2^{-(\beta/4)\mathfrak s}.
    \end{displaymath}
    For the converse, let $\mathcal J'\subset\mathcal D_\delta(\Theta)$ have at least $\mathfrak s$ $\beta$-singular scales at every point of its union. The same is true of $\Theta\cap\bigcup\mathcal J'$. Since every occupied cube contains exactly one point of $\Theta$, the bound $m_\Theta(\Theta\cap\bigcup\mathcal J')\le C2^{-\beta\mathfrak s}$ yields the assertion.
\end{proof}

We use the following comparison for weighted finite sets and cube families. If $w>0$ on a nonempty finite family $\mathcal J$, and its values differ by a factor of at most $K$, then
\begin{equation}\label{eq:comparable-weights-counting}
    \frac{|\mathcal J'|}{|\mathcal J|}
    \le K\frac{\sum_{Q\in\mathcal J'}w(Q)}{\sum_{Q\in\mathcal J}w(Q)},
    \qquad \mathcal J'\subset\mathcal J.
\end{equation}

\begin{lemma}[Uniform discretization]\label{lem:uniform-spread-out-discretization}
Let $n:=k(d-k)$, $R\ge1$, $\beta\in(0,1]$, $C\ge1$, and $\delta\in2^{-\N}$ with $\delta\ll_R1$. Let $\mu$ be a probability measure on $[-R,R]^n$, and let $E$ be a Borel set with $\mu(E)>0$. If $\mu$ is $(\delta,\beta,C)_{d,k}$-spread-out, then there is a nonempty family $\mathcal J\subset\mathcal D_\delta(E)$ such that
\begin{enumerate}[label=\textup{(\roman*)}]
    \item the quantities $\mu(E\cap Q)$, $Q\in\mathcal J$, differ by a factor of at most $2$, and $\mu(E\cap\bigcup\mathcal J)\gtrapprox_\delta\mu(E)$;
    \item $\mathcal J$ is $(\delta,\beta,C')_{d,k}$-spread-out with $C'\lessapprox_\delta C/\mu(E)$, and the uniform measure on $z\mathcal J:=\{z(Q):Q\in\mathcal J\}$ is $(\delta,\beta/4,C'')_{d,k}$-spread-out with $C''\lessapprox_\delta C/\mu(E)$.
\end{enumerate}
If, in addition, $s>0$, $C_F\ge1$, and $\mu(B(x,r))\le C_F r^s$ for $x\in\R^n$ and $\delta\le r\le1$, then
\begin{displaymath}
    |\mathcal J|\gtrapprox_\delta \mu(E)C_F^{-1}\delta^{-s}.
\end{displaymath}
\end{lemma}

\begin{proof}
Discard the cubes $Q\in\mathcal D_\delta(E)$ with $\mu(E\cap Q)<\delta^{n+1}\mu(E)$. Since $\mu$ is supported on $[-R,R]^n$, the discarded mass is $O_R(\delta)\mu(E)$. The remaining nonzero masses lie between $\delta^{n+1}\mu(E)$ and $\mu(E)$, so dyadic pigeonholing yields a family $\mathcal J$ satisfying \textup{(i)}.

Let $\mathcal J'\subset\mathcal J$, and assume that every $Q\in\mathcal J'$ has at least $\mathfrak s$ $\beta$-singular scales relative to $\bigcup\mathcal J'$. Every point of $E\cap\bigcup\mathcal J'$ then has at least $\mathfrak s$ $\beta$-singular scales relative to $E\cap\bigcup\mathcal J'$. Hence
\begin{displaymath}
    \mu\left(E\cap\bigcup\mathcal J'\right)\le C2^{-\beta\mathfrak s}.
\end{displaymath}
Apply \eqref{eq:comparable-weights-counting} with $w(Q)=\mu(E\cap Q)$ and $K=2$. By \textup{(i)},
\begin{displaymath}
    \frac{|\mathcal J'|}{|\mathcal J|}
    \le \frac{2\mu(E\cap\bigcup\mathcal J')}
    {\mu(E\cap\bigcup\mathcal J)}
    \lessapprox_\delta \frac{C}{\mu(E)}2^{-\beta\mathfrak s}.
\end{displaymath}
This proves the first assertion in \textup{(ii)}, and the assertion concerning $m_{z\mathcal J}$ follows from Lemma \ref{lem:spread-out-discretization}\textup{\ref{it:spread-out-discretization-cubes}}. Finally, the Frostman hypothesis implies $\mu(E\cap Q)\lesssim_{d,k,s}C_F\delta^s$ for $Q\in\mathcal J$. Combining this with \textup{(i)} proves the cardinality bound.
\end{proof}

\begin{lemma}[Similarity stability]
\label{lem:spread-out-scaling}
For $\lambda>0$ and $a\in\mathbf G_{d,k}$, let $T_{\lambda,a}(\theta):=\lambda^{-1}(\theta-a)$. If $\mu$ is a $(\delta,\beta,C)_{d,k}$-spread-out measure, then the following hold.
\begin{enumerate}[label=\textup{(\roman*)}]
    \item If $\lambda\in2^{-\N}\cap[\delta,1]$, then $T_{\lambda,a}\mu$ is $(\delta/\lambda,\beta,C)_{d,k}$-spread-out.
    \item If $\lambda\in[\delta,\delta^{-1}]$, then $T_{\lambda,a}\mu$ is $(\delta,\beta/2,C')_{d,k}$-spread-out with $C'\lesssim_\beta C\max\{\lambda,\lambda^{-1}\}^{\beta}$.
\end{enumerate}
\end{lemma}
\begin{proof}
The first claim follows from $T_{\lambda,a}^{-1}(H_{[x,y,t]})
    =H_{[x,y,\lambda t+(ay)\cdot x]}$ and tracking the definitions.

For \textup{(ii)}, let $E\subset\spt(T_{\lambda,a}\mu)$ have at least $\mathfrak s$ $(\beta/2)$-singular scales at each point. Pulling back the witnessing hyperplanes and rounding the rescaled radii dyadically shows that every point of $T_{\lambda,a}^{-1}(E)\cap\spt\mu$ has at least $\max\{\mathfrak s-|\log\lambda|-2,0\}$ $\beta$-singular scales. Hence
\begin{displaymath}
    (T_{\lambda,a}\mu)(E)
    \le C2^{-\beta\max\{\mathfrak s-|\log\lambda|-2,0\}}
    \lesssim_\beta C\max\{\lambda,\lambda^{-1}\}^{\beta}2^{-(\beta/2)\mathfrak s}.
\end{displaymath}
This proves the second conclusion.
\end{proof}


            \section{ABC sum-product for spread-out projections}

            \subsection{Main result}

            Recall that Frostman sets were defined in \S\ref{s:discretized-formulations}. Up to a constant depending only on $d$, the dyadic-cube condition there is equivalent to $|A\cap B(x,r)|_{\delta}\leq C r^{\alpha}|A|_{\delta}$ for all $x\in\R^d$ and $r\geq\delta$.

            The main result of this section is the following asymmetric (``ABC'') sum-product estimate.
            \begin{proposition}\label{prop:ABC}
                Let $\alpha \in [0,k)$, $\beta\in (0,d-k]$, and $\gamma\in (0,1]$. Then there exists $\chi>0$, depending only on $\alpha,\beta,\gamma$, such that the following holds for every $\delta\in2^{-\N}$ with $\delta\ll_{\alpha,\beta,\gamma}1$. Let $A\subset \delta\Z^{k}\cap [-1,1]^k$, $B\subset \delta\Z^{d-k}\cap [-1,1]^{d-k}$, and $\Theta \subset \delta\Z^{k\times(d-k)} \cap [-1,1]^{k\times(d-k)}$ be sets satisfying the following properties:
                \begin{itemize}
                    \item[\textup{(A)}] $|A| \leq \delta^{-\alpha}$,
                    \item[\textup{(B)}] $B$ is $(\delta,\beta,\delta^{-\chi})$-Frostman,
                    \item[\textup{(C)}] $\Theta$ is $(\delta,\gamma,\delta^{-\chi})$-spread-out.
                \end{itemize}
                Then there exists $\theta \in \Theta$ such that
                \begin{equation}\label{form1}
                    |\{a + \theta\cdot b : (a,b) \in G\}|_{\delta} \geq \delta^{-\chi}|A|, \qquad\text{for all } G \subset A \times B, \, |G| \geq \delta^{\chi}|A||B|.
                \end{equation}
                Moreover, the constant $\chi > 0$ stays bounded away from zero as long as $(\alpha,\beta,\gamma)$ ranges in a compact subset of $[0,k) \times (0,d-k] \times (0,1]$.
            \end{proposition}


\subsection{Spread-out sets are \texorpdfstring{$L^2$}{L2}-flattening}

A key step in the proof of Proposition \ref{prop:ABC} is to show that spread-out sets are $L^2$-flattening, in the sense of Theorem \ref{thm:spread-out-implies-flattening} below. This strategy is similar to that used in \cite{OS23}, but with some key differences; in particular, the results of \cite{OS23} relied on Bourgain's discretized projection theorem \cite{Bo2}, while here we rely on an inverse theorem of the second author \cite{Shmerkin25Orbita}, recalled below.

We first introduce some notation. For finite sets $A,B \subset \R^d$, we denote by $\mathcal{E}(A,B)$ the \emph{additive energy}
\begin{displaymath}
    \mathcal{E}(A,B) := |\{(a_{1},a_{2},b_{1},b_{2}) \in A^{2} \times B^{2} : a_{1} + b_{1} = a_{2} + b_{2} \}|.
\end{displaymath}
For approximate coincidences, we use the following indexed energy. For a finite set $A\subset\R^k$ and a finite indexed family $\mathbf x=(x_i)_{i\in I}\subset\R^k$, write
\begin{displaymath}
    \mathcal E_\delta(A;\mathbf x)
    :=\bigl|\{(a_1,a_2,i_1,i_2)\in A^2\times I^2:
    a_1+x_{i_1}\in a_2+x_{i_2}+B(0,\delta)\}\bigr|.
\end{displaymath}
This is a hybrid rather than a fully discretized energy: we still count quadruples of points and indices and merely allow a $\delta$-fudge term in the equality.
For a finite set $\Theta\subset\mathbf G_{d,k}$ and $b\in\R^{d-k}$, we abbreviate $\mathcal E_\delta(A;\Theta,b):=\mathcal E_\delta(A;(\theta\cdot b)_{\theta\in\Theta})$.

The goal of this subsection is to prove that spread-out sets are $L^2$-flattening, in the following sense:
\begin{theorem}\label{thm:spread-out-implies-flattening}
    Let $1\le k<d$, let $\alpha\in[0,k)$ and $s\in(0,1]$. Then there exists $\e>0$ such that for every $\delta\in 2^{-\N}$ with $\delta\ll_{\alpha,s}1$ the following holds. If
    \begin{displaymath}
        A\subset \delta\Z^{k}\cap[-1,1]^{k},\qquad
        \Theta\subset \delta\Z^{k\times(d-k)}\cap[-1,1]^{k\times(d-k)}\subset \mathbf{G}_{d,k},
    \end{displaymath}
    where $A$ and $\Theta$ are nonempty, $|A|\le \delta^{-\alpha}$, and $\Theta$ is $(\delta,s,\delta^{-\e})_{d,k}$-spread-out, then
    \begin{equation}\label{eq:L2-flattening-measures}
        \|m_A*(L_bm_\Theta)^{(\delta)}\|_{L^{2}}
        \leq \delta^{\e}\|m_A\|_{L^{2}}
        \qquad \text{for all } b\in[-2,2]^{d-k},\ |b|\geq\delta^\e.
    \end{equation}
    Consequently, for every fixed $C\geq1$,
    \begin{equation}\label{eq:indexed-energy-bound}
        \mathcal E_{C\delta}(A;\Theta,b)
        \lesssim_{C,k}\delta^{2\e}|A||\Theta|^2
        \qquad \text{for all }b\in[-2,2]^{d-k},\ |b|\geq\delta^\e.
    \end{equation}
    Moreover, $\e$ stays bounded away from zero when $(\alpha,s)$ ranges in a compact subset of $[0,k)\times(0,1]$.
\end{theorem}

The proof is based on an ``inverse theorem'' of the second author \cite[Theorem 1.2]{Shmerkin25Orbita}; we restate the part of an immediate corollary that we need. 

We recall the notion of a uniform set and its branching numbers.
\begin{definition}[Uniform sets and branching numbers] \label{def:uniform-set}
Let $\Delta\in 2^{-\N}$ and $m\in\N$, and let $P\subset [0,1)^d$. We say that $P$ is \emph{$(\Delta,m)$-uniform with branching numbers $(N_j)_{j=0}^{m-1}$} if every cube $Q\in\mathcal D_{\Delta^j}(P)$ contains exactly $N_j$ cubes from $\mathcal D_{\Delta^{j+1}}(P)$, for each $0\le j\le m-1$.
\end{definition}

The $L^{2}$-norm of a finitely supported measure $\nu$ is defined by $\|\nu\|_{L^{2}}^{2} := \sum_{z \in \spt(\nu)} \nu(z)^{2}$.

            \begin{theorem}[$L^2$ inverse theorem]\label{thm:inverse}
                Fix a dimension $d \ge 1$. For each $\zeta \in (0,1]$ and $T_{0} \in \N$, there exist $T \geq T_{0}$, $\e = \e(\zeta,T_{0}) > 0$, and $m_0=m_0(\zeta,T_0)\in\N$ such that the following holds for $m \geq m_0$.

                Let $\Delta:=2^{-T}$ and $\delta := \Delta^m$, and let $\mu,\nu$ be measures on $\delta \Z^{d} \cap [\tfrac{1}{3},\tfrac{2}{3}]^{d}$ with $\mu(\R^d) = 1 = \nu(\R^d)$, and
                \begin{displaymath}
                    \|\mu \ast \nu\|_{L^{2}} \geq \delta^{\e}\|\mu\|_{L^{2}}.
                \end{displaymath}
                Then there exist nonempty sets $A \subset \spt \mu$, $B \subset \spt \nu$ and vectors $x_A,x_B\in\delta\Z^d$ such that $A':=A+x_A\subset[0,1)^d$ is $(\Delta,m)$-uniform with branching numbers $(N_j^A)_{j=0}^{m-1}$, $B':=B+x_B\subset[0,1)^d$, and $\nu(B)\ge\delta^\zeta$. Moreover, if $x\in B'$, $0\le j\le m-1$, and $I\in\mathcal D_{\Delta^j}$ is the cube containing $x$, then $x\in\tfrac12 I$. For each $j \in \{0,\ldots,m-1\}$, at least one of the following alternatives holds:
                \begin{itemize}
                    \item[\textup{(S1)}] $N_{j}^{A} \geq \Delta^{-(1 - \zeta)d}$.
                    \item[\textup{(S2)}] For each cube $I \in \mathcal{D}_{\Delta^j}$ there exists $H_{I} \in \A(d,d - 1)$ such that $B'\cap I \subset \bigcup\mathcal{D}_{\Delta^{j+1}}(H_{I})$.
                \end{itemize}
            \end{theorem}

            We are then equipped to prove Theorem \ref{thm:spread-out-implies-flattening}.
            \begin{proof}[Proof of Theorem \ref{thm:spread-out-implies-flattening}]
                Let $\sigma:=s/32$. Choose $0<\zeta\ll_{\alpha,k,s}1$ so that
                \begin{equation}\label{form100}
                    \frac{\alpha}{k(1-\zeta)}\leq\frac{\alpha}{k}+\zeta
                    \qquad\text{and}\qquad
                    \sigma(1-\alpha/k-\zeta)/4\geq\zeta.
                \end{equation}
                Next, choose $T_0\gg_{\alpha,k,s,\zeta}1$ so that
                \begin{equation}\label{choiceT}
                    \sqrt{k}\,2^{-T_0/2}<\sigma, \qquad T_0\geq12,
                    \qquad 4(d-k)2^{-T_0}<\tfrac{1}{100}.
                \end{equation}
                Let $T\geq T_0$ and $\e_{\mathrm{inv}}>0$ be given by Theorem \ref{thm:inverse} in dimension $k$, and write $\Delta:=2^{-T}$. Take $0<\e\ll_{\e_{\mathrm{inv}},\zeta}1$. From \eqref{form100}, we deduce that
                \begin{equation}\label{form101}
                    \sigma(1-\alpha/k-\zeta)/3-2\e
                    >\sigma(1-\alpha/k-\zeta)/4\geq\zeta.
                \end{equation}

                Fix $b \in [-2,2]^{d-k}$ with $|b| \geq \delta^{\e}$, and write $\nu_{\Theta,b}:=(L_bm_\Theta)^{(\delta)}$. We first prove \eqref{eq:L2-flattening-measures}.

                In short, the idea of the proof is as follows: if \eqref{eq:L2-flattening-measures} fails, then we can apply Theorem \ref{thm:inverse} to the measures $m_A$ and $\nu_{\Theta,b}$. The upper bound $|A|\leq \delta^{-\alpha}$ is easily seen to imply that alternative (S2) in Theorem \ref{thm:inverse} holds for a positive proportion of the scales, and this will contradict the spread-outness of $L_b m_{\Theta}$, which in turn follows from the spread-outness of $\Theta$ via Lemmas \ref{lem:spread-out-discretization} and \ref{lem:proj-spread-out}.

                By Lemma \ref{lem:spread-out-discretization}\textup{\ref{it:spread-out-discretization-cubes}}, $m_\Theta$ is $(\delta,s/4,O_{d,k,s}(\delta^{-\e}))_{d,k}$-spread-out. Lemmas \ref{lem:proj-spread-out} and \ref{lem:spread-out-scaling}\textup{(ii)} applied to $L_{b/|b|}m_\Theta$ and dilation by $|b|$ show that $L_bm_\Theta$ is $(\delta,s/8,O_{d,k,s}(\delta^{-2\e}))_k$-spread-out. Lemma \ref{lem:spread-out-discretization}\textup{\ref{it:spread-out-discretization-measures}} therefore implies that $\nu_{\Theta,b}$ is $(\delta,\sigma,O_{d,k,s}(\delta^{-2\e}))_k$-spread-out.

                Assume that \eqref{eq:L2-flattening-measures} fails. After a change of coordinates, perturbing the scale, and some calculations, we may apply Theorem \ref{thm:inverse}. We get sets $A_0\subset A$ and $Y\subset\spt\nu_{\Theta,b}$ and vectors $k_A,k_Y\in\delta\Z^k$ such that $\widetilde A:=A_0+k_A$ is $(\Delta,m)$-uniform, $Y':=Y+k_Y\subset[0,1)^k$ satisfies the centrality conclusion of that theorem, and
                \begin{equation}\label{form102}
                    \nu_{\Theta,b}(Y) \geq \delta^{\zeta}.
                \end{equation}
                Since $|\widetilde A|\leq|A|\leq\delta^{-\alpha}=\Delta^{-\alpha m}$, (S1) can hold at no more than $\alpha m/[k(1-\zeta)]\leq(\alpha/k+\zeta)m$ indices. Thus (S2) holds at least $(1-\alpha/k-\zeta)m$ times.

                Fix such an index $j$ and $y\in Y$, write $y':=y+k_Y$, and take $r\in2^{-\N}\cap[2^{-T/2}\Delta^j,\Delta^j/4]$. If $I\in\mathcal D_{\Delta^j}$ contains $y'$, then Theorem \ref{thm:inverse} yields $y'\in\tfrac12 I$, so $B(y',r)\subset I$. Moreover, (S2) and \eqref{choiceT} imply
                \begin{displaymath}
                    Y'\cap B(y',r)\subset Y'\cap I
                    \subset H_I(\sqrt{k}\Delta^{j+1})
                    \subset H_I(\sigma r).
                \end{displaymath}
                Hence $r$ is $\sigma$-singular for $(Y,y)$. The intervals $[2^{-T/2}\Delta^j,\Delta^j/4]$ corresponding to distinct indices $j$ are disjoint, and each contains $\geq T/3$ dyadic scales. The spread-outness of $\nu_{\Theta,b}$ now yields
                \begin{displaymath}
                    \nu_{\Theta,b}(Y)
                    \lesssim_{d,k,s,T}\delta^{-2\e}
                    2^{-\sigma(T/3)(1-\alpha/k-\zeta)m}
                    =\delta^{\sigma(1-\alpha/k-\zeta)/3-2\e}
                    <\delta^\zeta,
                \end{displaymath}
                for $\delta\ll_{d,k,s,T,\e,\zeta}1$, by \eqref{form101}. This contradicts \eqref{form102}, proving \eqref{eq:L2-flattening-measures}.

                Returning to the original measures and scale, let $x_\theta\in\delta\Z^k$ be the lower-left endpoint of the dyadic $\delta$-cube containing $\theta\cdot b$. If a tuple is counted by $\mathcal E_{C\delta}(A;\Theta,b)$, then
                \begin{displaymath}
                    (a_1+x_{\theta_1})-(a_2+x_{\theta_2})\in\delta\Z^k
                    \cap B\bigl(0,(C+2\sqrt{k})\delta\bigr).
                \end{displaymath}
                For each $r\in\delta\Z^k\cap B(0,(C+2\sqrt{k})\delta)$, Cauchy--Schwarz bounds the number of tuples with $(a_1+x_{\theta_1})-(a_2+x_{\theta_2})=r$ by $|A|^2|\Theta|^2\|m_A*\nu_{\Theta,b}\|_{L^2}^2$. Summing over the $O_{C,k}(1)$ possible values of $r$, we get
                \begin{displaymath}
                    \mathcal E_{C\delta}(A;\Theta,b)
                    \lesssim_{C,k}|A|^2|\Theta|^2
                    \|m_A*\nu_{\Theta,b}\|_{L^2}^2
                    \leq\delta^{2\e}|A||\Theta|^2.
                \end{displaymath}
                This proves \eqref{eq:indexed-energy-bound}.

                The parameters $\zeta,T_0,T,$ and $\e$ can be chosen uniformly when $(\alpha,s)$ ranges in a compact subset of $[0,k)\times(0,1]$. This completes the proof of the theorem.
            \end{proof}

            \subsection{Preliminaries on additive energy} We proceed toward the proof of Proposition \ref{prop:ABC}. We need the asymmetric Balog-Szemer\'edi-Gowers theorem, see the book of Tao and Vu, \cite[Theorem 2.35]{MR2289012}. We state the result in the following slightly weaker form (following \cite[Theorem 3.2]{Shmerkin19}):
        \begin{theorem}[Asymmetric Balog-Szemer\'edi-Gowers theorem]\label{thm:BSG}
            Given $\eta>0$, there exists $\zeta>0$ such that the following holds for $\delta\in2^{-\N}$ with $\delta\ll_\eta1$. Let $A,B \subset (\delta \cdot \Z^k) \cap [0,1]^k$, and assume that $\mathcal{E}(A,B) \geq \delta^{\zeta}|A||B|^{2}$.

            Then, there exist subsets $A' \subset A$ and $B' \subset B$ with the properties
            \begin{displaymath} 
                |A'||B'| \geq \delta^{\eta}|A||B| \quad \text{and} \quad |A' + B'|_{\delta} \leq \delta^{-\eta}|A|.
            \end{displaymath}
        \end{theorem}

        By applying Theorem \ref{thm:BSG} twice, we get simultaneous control on $|A'\pm B'|$.
        \begin{corollary}\label{cor:two-sided-BSG}
            Given $\eta>0$, there exists $\zeta=\zeta(\eta)>0$ such that the following holds for $\delta\in2^{-\N}$ with $\delta\ll_\eta1$. Let $A,B\subset(\delta\cdot\Z^k)\cap[0,1]^k$, and assume that $\mathcal E(A,B)\geq\delta^{\zeta}|A||B|^2$.
            
            Then there are $A'\subset A$ and $B'\subset B$ such that
            \begin{displaymath}
                |A'||B'|\geq\delta^\eta|A||B|,
                \qquad
                |A'\pm B'|_\delta \leq \delta^{-\eta}|A|.
            \end{displaymath}
        \end{corollary}
        \begin{proof}
            Let $\xi(t)>0$ be an exponent supplied by Theorem \ref{thm:BSG} with parameter $t$, set
            \begin{displaymath}
                \eta_1:=\min\{\eta/3,\xi(\eta/2)/3\},
                \qquad
                \zeta:=\xi(\eta_1),
            \end{displaymath}
            and apply Theorem \ref{thm:BSG} to $A,B$ with parameter $\eta_1$. This gives $A_1\subset A$ and $B_1\subset B$ such that
            \begin{displaymath}
                |A_1||B_1|\geq\delta^{\eta_1}|A||B|,
                \qquad
                |A_1+B_1|_\delta\leq\delta^{-\eta_1}|A|.
            \end{displaymath}
            Cauchy--Schwarz and the bound $|A_1|\geq\delta^{\eta_1}|A|$ yield
            \begin{displaymath}
                \mathcal E(A_1,-B_1)=\mathcal E(A_1,B_1)
                \gtrsim_k\frac{|A_1|^2|B_1|^2}{|A_1+B_1|_\delta}
                \geq\delta^{2\eta_1}|A_1||B_1|^2.
            \end{displaymath}
            Since $2\eta_1<\xi(\eta/2)$, the implicit constant is absorbed when $\delta\ll_{k,\eta}1$. Apply Theorem \ref{thm:BSG} to $A_1,-B_1$ with parameter $\eta/2$. We obtain $A'\subset A_1$ and $B'\subset B_1$ satisfying the claim.
        \end{proof}

        We also need the following version of the Pl\"unnecke-Ruzsa inequality:
        \begin{lemma}[Pl\"unnecke-Ruzsa inequality]\label{PRIneq}
            Let $\delta \in 2^{-\N}$, let $A,B_{1},B_{2} \subset \R^k$ be sets, and assume that $|A + B_{i}|_{\delta} \leq K_{i}|A|_{\delta}$ for $i \in \{1,2\}$, for some constants $K_{i} \geq 1$. Then there exists a subset $A' \subset A$ with $|A'|_{\delta} \geq \tfrac{1}{2}|A|_{\delta}$ such that
            \begin{displaymath}
                |A' + B_{1} + B_{2}|_{\delta} \lesssim K_{1}K_{2} |A|_{\delta}.
            \end{displaymath}
        \end{lemma}
        This form of the inequality is due to Ruzsa \cite{MR2314377}. To be accurate, these statements are not formulated in terms of $\delta$-covering numbers, but one may consult \cite[Corollary 3.4]{MR4283564} to see how to handle the simple reduction.

        \subsection{A Kaufman-type projection lemma}

        We will require a quantitative estimate for projections of Frostman sets under spread-out measures, inspired by Kaufman's projection theorem \cite{Ka}.

        For a finite set $B\subset\R^{d-k}$, $\theta\in\mathbf G_{d,k}$, $\delta\in2^{-\N}$, and $b\in B$, write
        \begin{displaymath}
            [b]_{\theta}^B := \{b'\in B: \theta \cdot b' \in\mathcal{D}_{\delta}(\theta\cdot b)\}.
        \end{displaymath}
        \begin{lemma}\label{lem:projection-fibres}
            Let $C,D\geq1$, $0<\beta<\gamma\leq1$, $\rho\in(0,1]$, and $\delta\in2^{-\N}$. Let $B\subset\delta\Z^{d-k}\cap[-1,1]^{d-k}$ be $(\delta,\beta,C)$-Frostman with $|B|\leq\delta^{-\beta}$, and let $\mu$ be a $(\delta,\gamma,D)_{d,k}$-spread-out probability measure. Then there exists $\Theta'\subset\mathbf G_{d,k}$ with $\mu(\Theta')\geq1-\rho$ such that
            \begin{displaymath}
                \bigl|\{b\in B:|[b]^B_\theta|>C'\rho^{-1}CDL\}\bigr|
                \leq L^{-1}|B|,
                \qquad \theta\in\Theta',\ L\geq1,
            \end{displaymath}
            where $C'=C'(d,k,\beta,\gamma)$.
        \end{lemma}

        \begin{proof}
          Lemma \ref{lem:spread-out-global-non-concentration}, followed by an annular decomposition and the Frostman bound for $B$, gives
            \begin{displaymath}
                \int \frac{1}{|B|^2}\sum_{b\in B}|[b]^B_\theta|\,d\mu(\theta)
                \lesssim_{\beta,\gamma}
                D\delta^\gamma\sum_{j=0}^{\log(1/\delta)}2^{j\gamma}\,C2^{-j\beta}
                +|B|^{-1} \lesssim_{\beta,\gamma}CD\delta^\beta.
            \end{displaymath}
            Since $|B|\leq\delta^{-\beta}$, Markov's inequality provides a set $\Theta'$ of $\mu$-measure at least $1-\rho$ such that
            \begin{displaymath}
                \sum_{b\in B}|[b]^B_\theta|
                \leq C'\rho^{-1}CD|B|,
                \qquad \theta\in\Theta'.
            \end{displaymath}
            The conclusion follows from Markov's inequality on $B$.
        \end{proof}

        We use the following standard decomposition lemma.
        \begin{lemma} \label{lem:frostman-decomposition}
            Let $n\ge 1$, $\delta \in 2^{-\N}$, $\beta>0$, and $C,K\ge 1$. If $B\subset \R^{n}$ is a bounded $(\delta,\beta,C)$-Frostman set, then there exist sets $B_1,\ldots,B_M\subset B$ such that the families $\mathcal{D}_{\delta}(B_{j})$ are disjoint, and:
            \begin{enumerate}[label=\textup{(\roman*)}]
                \item Each $B_j$ is $(\delta,\beta,2^n CK)$-Frostman.
                \item $|B_j|_{\delta}\le \delta^{-\beta}$ for all $j\in\{1,\ldots,M\}$.
                \item $\bigl|B\setminus (B_1\cup\ldots\cup B_M)\bigr|_{\delta} \le K^{-1}|B|_{\delta}$.
            \end{enumerate}
            In particular $M \lesssim_{n} CK|B|_{\delta}\delta^{\beta}$.
        \end{lemma}
        
        \begin{proof}
            Suppose first that $|B|_\delta\leq K\delta^{-\beta}$. Partition $\mathcal D_\delta(B)$ into families $\mathcal P_j$ of cardinality between $(2K)^{-1}|B|_\delta$ and $\delta^{-\beta}$, discarding fewer than $(2K)^{-1}|B|_\delta$ cubes, and set $B_j:=B\cap\bigcup\mathcal P_j$. Then (ii)--(iii) hold, while
            \begin{displaymath}
                |B_j\cap I|_\delta\leq Cr^\beta|B|_\delta
                \leq 2CKr^\beta|B_j|_\delta,\qquad I\in\mathcal D_r,\ r\geq\delta.
            \end{displaymath}
            Thus (i) holds, and $M\leq2K\lesssim_n CK|B|_\delta\delta^\beta$.

            Assume now that $|B|_\delta>K\delta^{-\beta}$. We first claim (cf.~\cite[Lemma~2.7]{OS23}) that every $(\delta,\beta,C')$-Frostman set $A$ with $C'\geq1$ and $|A|_\delta>\delta^{-\beta}$ contains a subset $A'=A\cap\bigcup\mathcal P$, where $\mathcal P\subset\mathcal D_\delta$, such that
            \begin{displaymath}
                2^{-n}\delta^{-\beta}\leq|A'|_\delta\leq\delta^{-\beta}
            \end{displaymath}
            and $A'$ is $(\delta,\beta,2^nC')$-Frostman. To prove the claim, set $N:=\lfloor\delta^{-\beta}\rfloor$ and assign each dyadic cube $J$ meeting $A$, with $\delta\leq\ell(J)\leq1$, either the floor or the ceiling of $N|A\cap J|_\delta/|A|_\delta$. First choose the unit-cube weights to sum to $N$, and then proceed down the dyadic tree, choosing the child weights to sum to the assigned parent weight. At scale $\delta$, the assigned integers are $0$ or $1$, since $N/|A|_\delta<1$. Let $\mathcal P$ consist of the cubes assigned $1$, and set $A':=A\cap\bigcup\mathcal P$. Additivity yields
            \begin{displaymath}
                |A'|_\delta=N\in[\tfrac12\delta^{-\beta},\delta^{-\beta}],
            \end{displaymath}
            and, for every dyadic $J$ with $\delta\leq\ell(J)\leq1$,
            \begin{displaymath}
                |A'\cap J|_\delta
                \leq\left\lceil\frac{N|A\cap J|_\delta}{|A|_\delta}\right\rceil
                \leq\lceil C'\ell(J)^\beta N\rceil
                \leq2C'\ell(J)^\beta N.
            \end{displaymath}
            Since $n\geq1$, this proves the claim.

            Starting with $R_0=B$, apply this extraction to $R_j$ while $|R_j|_\delta>K^{-1}|B|_\delta$. At every such step, $R_j$ is $(\delta,\beta,CK)$-Frostman because it is a subset of $B$. The extracted sets have disjoint dyadic $\delta$-cubes, the terminal remainder satisfies (iii), and each extracted set satisfies (i)--(ii). Their lower cardinality bound also yields $M\leq2^n|B|_\delta\delta^\beta$, completing the proof.
        \end{proof}

        \subsection{Proof of Proposition \ref{prop:ABC}}

        \begin{proof}
            If $A=\varnothing$, the conclusion is immediate. Let $\tau:=\min\{\beta,\gamma/8\}$. Let $\e=\e(\alpha,\gamma)>0$ be given by Theorem \ref{thm:spread-out-implies-flattening}, decreased so that $\e\leq1/2$. For $t>0$, let $\zeta(t)>0$ be such that Corollary \ref{cor:two-sided-BSG} holds with exponent $t$. More precisely, we will apply it for sets in some fixed bounded domain that may differ from the unit cube; this is achieved by an appropriate change of coordinates. Choose $\eta>0$ so that $100\eta/\tau<\e$, and then set
            \begin{displaymath}
                \chi:=\min\{\e/10,\eta/100,\zeta(\eta/2)/40\}.
            \end{displaymath}
            \begin{samepage}
            The following diagram summarizes the parameter dependencies:
            
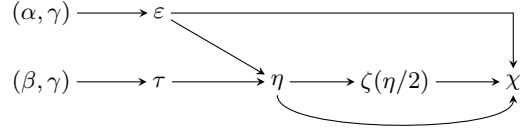
\begin{center}
            \begin{tikzpicture}[x=0.65cm,y=0.65cm,every node/.style={font=\small,inner sep=2pt},>=stealth]
                \node (ag) at (0,1.4) {$(\alpha,\gamma)$};
                \node (bg) at (0,0) {$(\beta,\gamma)$};
                \node (eps) at (2.4,1.4) {$\e$};
                \node (tau) at (2.4,0) {$\tau$};
                \node (eta) at (4.8,0) {$\eta$};
                \node (zeta) at (7.2,0) {$\zeta(\eta/2)$};
                \node (chi) at (9.6,0) {$\chi$};
                \draw[->] (ag) -- (eps);
                \draw[->] (bg) -- (tau);
                \draw[->] (eps) -- (eta);
                \draw[->] (tau) -- (eta);
                \draw[->] (eta) -- (zeta);
                \draw[->] (zeta) -- (chi);
                \draw[->] (eps.east) -- (9.6,1.4) -- (chi.north);
                \draw[->] (eta.south) .. controls +(0,-0.8) and +(0,-0.8) .. (chi.south);
            \end{tikzpicture}
            \end{center}
            \end{samepage}

            Suppose that \eqref{form1} fails for every $\theta\in\Theta$. Choose $G_\theta\subset A\times B$ such that
            \begin{displaymath}
                |G_\theta|\geq\delta^\chi|A||B|,
                \qquad
                |\{a+\theta\cdot b:(a,b)\in G_\theta\}|_\delta
                <\delta^{-\chi}|A|.
            \end{displaymath}
            Since $\tau\leq\beta$, the set $B$ is $(\delta,\tau,\delta^{-\chi})$-Frostman. If $|B|>\delta^{-\tau}$, apply Lemma \ref{lem:frostman-decomposition} with exponent $\tau$ and $K=\delta^{-2\chi}$. Its pieces satisfy
            \begin{displaymath}
                |B_j|\leq\delta^{-\tau},\qquad
                B_j\text{ is }(\delta,\tau,\delta^{-4\chi})\text{-Frostman},
            \end{displaymath}
            since $\delta\ll_{\tau,\chi}1$. The discarded set has cardinality at most $\delta^{2\chi}|B|$. For each $\theta$, the remaining part of $G_\theta$ has cardinality at least $\tfrac12\delta^\chi|A||B|$. Hence
            \begin{displaymath}
                \sum_j|B_j|
                \mathbf 1_{\{|G_\theta\cap(A\times B_j)|
                \geq\delta^{2\chi}|A||B_j|\}}
                \geq\tfrac14\delta^\chi|B|.
            \end{displaymath}
            Averaging this inequality over $\theta$, we obtain an index $j$ and a set $\Theta'\subset\Theta$ such that
            \begin{displaymath}
                |\Theta'|\geq\delta^{2\chi}|\Theta|,
                \qquad
                |G_\theta\cap(A\times B_j)|
                \geq\delta^{2\chi}|A||B_j|,
                \quad \theta\in\Theta'.
            \end{displaymath}
            Replace $B,\Theta,G_\theta$ by $B_j,\Theta',G_\theta\cap(A\times B_j)$, respectively. Lemma \ref{lem:spread-out-subset} shows that the new $\Theta$ is $(\delta,\gamma,\delta^{-4\chi})_{d,k}$-spread-out. If $|B|\leq\delta^{-\tau}$, no replacement is needed. Thus in both cases we may assume that
            \begin{displaymath}
                |B|\leq\delta^{-\tau},\qquad
                B\text{ is }(\delta,\tau,\delta^{-4\chi})\text{-Frostman},
                \qquad
                \Theta\text{ is }(\delta,\gamma,\delta^{-4\chi})_{d,k}
                \text{-spread-out},
            \end{displaymath}
            while, for every $\theta\in\Theta$,
            \begin{displaymath}
                |G_\theta|\geq\delta^{2\chi}|A||B|,
                \qquad
                |\{a+\theta\cdot b:(a,b)\in G_\theta\}|_\delta
                <\delta^{-\chi}|A|.
            \end{displaymath}

            Lemma \ref{lem:spread-out-discretization}\textup{\ref{it:spread-out-discretization-cubes}} shows that $m_\Theta$ is $(\delta,\gamma/4,O_{d,k,\gamma}(\delta^{-4\chi}))_{d,k}$-spread-out. Apply Lemma \ref{lem:projection-fibres} to $m_\Theta$ with exponents $\tau,\gamma/4$ and
            \begin{displaymath}
                C=\delta^{-4\chi},\qquad D\lesssim_{d,k,\gamma}\delta^{-4\chi},\qquad
                \rho=\delta^\chi,\qquad L=\delta^{-3\chi}.
            \end{displaymath}
            After restricting $\Theta$ to a subset of relative cardinality at least $1-\delta^\chi$, Lemma \ref{lem:spread-out-subset} allows us to assume that $\Theta$ is $(\delta,\gamma,\delta^{-5\chi})_{d,k}$-spread-out and
            \begin{displaymath}
                \bigl|\{b\in B:|[b]^B_\theta|>\delta^{-13\chi}\}\bigr|
                \leq\delta^{3\chi}|B|,
                \qquad \theta\in\Theta.
            \end{displaymath}
            Replace $G_\theta$ by its intersection with
            \begin{displaymath}
                A\times\{b\in B:|[b]^B_\theta|\leq\delta^{-13\chi}\}.
            \end{displaymath}
            For $\delta\ll_\chi1$, this intersection has cardinality at least $\delta^{3\chi}|A||B|$.

            Write $X_\theta:=(\theta\cdot B)^{(\delta)}$, and let $\widetilde G_\theta\subset A\times X_\theta$ be the image of $G_\theta$ under
            \begin{displaymath}
                (a,b)\mapsto\bigl(a,z(\mathcal D_\delta(\theta\cdot b))\bigr).
            \end{displaymath}
            Then
            \begin{displaymath}
                |\widetilde G_\theta|\geq\delta^{17\chi}|A||X_\theta|,
                \qquad
                |X_\theta|\geq\delta^{14\chi}|B|.
            \end{displaymath}
            Since $|\theta\cdot b-z(\mathcal D_\delta(\theta\cdot b))|\leq\sqrt{k}\delta$, the restricted sumset satisfies
            \begin{displaymath}
                |\{a+x:(a,x)\in\widetilde G_\theta\}|
                \lesssim_k\delta^{-\chi}|A|.
            \end{displaymath}
            Cauchy--Schwarz applied to the sums indexed by $\widetilde G_\theta\subset A\times X_\theta$ yields (see \cite[Remark 3.17]{OrponenShmerkin26})
            \begin{displaymath}
                \mathcal E(A,X_\theta)
                \geq\frac{|\widetilde G_\theta|^2}{|\{a+x:(a,x)\in\widetilde G_\theta\}|}
                \gtrsim_{d,k}\delta^{35\chi}|A||X_\theta|^2.
            \end{displaymath}
            Since $35\chi<\zeta(\eta/2)$, Corollary \ref{cor:two-sided-BSG} gives $A_\theta\subset A$ and $Z_\theta\subset X_\theta$ such that
            \begin{displaymath}
                |A_\theta||Z_\theta|
                \geq\delta^{\eta/2}|A||X_\theta|,
                \qquad
                |A_\theta\pm Z_\theta|_\delta
                \leq \delta^{-\eta/2}|A|.
            \end{displaymath}
            For each $z\in Z_\theta$, choose $b_z\in B$ such that $\theta\cdot b_z\in z+[0,\delta)^k$, and denote $\{b_z:z\in Z_\theta\}$ by $B_\theta$. For $\delta\ll_{\eta,\chi}1$,
            \begin{displaymath}
                |A_\theta||B_\theta|
                \geq\delta^{\eta/2}|A||X_\theta|
                \geq\delta^{\eta/2+14\chi}|A||B|
                \geq\delta^{5\eta}|A||B|.
            \end{displaymath}
            By the definition of $B_\theta$,
            \begin{equation}\label{form2}
                |A_\theta\pm\theta\cdot B_\theta|_\delta
                \leq\delta^{-\eta}|A|,
                \qquad \theta\in\Theta.
            \end{equation}

            Set $F_\theta:=A_\theta\times B_\theta$. By Cauchy--Schwarz, the set
            \begin{displaymath}
                \Gamma:=\{(\theta_1,\theta_2)\in\Theta^2:
                |F_{\theta_1}\cap F_{\theta_2}|
                \geq\tfrac12\delta^{10\eta}|A||B|\}
            \end{displaymath}
            satisfies, for $\delta\ll_\eta1$,
            \begin{displaymath}
                |\Gamma|\geq\delta^{11\eta}|\Theta|^2,
                \qquad
                |F_{\theta_1}\cap F_{\theta_2}|
                \geq\delta^{11\eta}|A||B|,
                \quad(\theta_1,\theta_2)\in\Gamma.
            \end{displaymath}

            For $(\theta_1,\theta_2)\in\Gamma$, apply Lemma \ref{PRIneq} and Cauchy--Schwarz to $A_{\theta_1}\cap A_{\theta_2}$ and $B_{\theta_1}\cap B_{\theta_2}$, using \eqref{form2}. We obtain
            \begin{displaymath}
                \mathcal E_{\sqrt{k}\delta}
                \bigl(A;((\theta_1-\theta_2)\cdot b)_{b\in B}\bigr)
                \geq\delta^{49\eta}|A||B|^2,
                \qquad(\theta_1,\theta_2)\in\Gamma.
            \end{displaymath}
            The definition of indexed energy yields the identity
            \begin{equation}\label{eq:ABC-six-tuple}
                \sum_{b_1,b_2\in B}
                \mathcal E_{\sqrt{k}\delta}(A;\Theta,b_1-b_2)
                =
                \sum_{\theta_1,\theta_2\in\Theta}
                \mathcal E_{\sqrt{k}\delta}
                \bigl(A;((\theta_1-\theta_2)\cdot b)_{b\in B}\bigr) \geq\delta^{60\eta}|A||B|^2|\Theta|^2.
            \end{equation}
            For fixed $b_1,b_2$, the indexed energy in this sum is $O_k(|A||\Theta|^2)$. Since $B$ is $(\delta,\tau,\delta^{-4\chi})$-Frostman,
            \begin{displaymath}
                \bigl|\{(b_1,b_2)\in B^2:
                |b_1-b_2|<\delta^{(65\eta+4\chi)/\tau}\}\bigr|
                \lesssim_{d,k,\tau}\delta^{65\eta}|B|^2.
            \end{displaymath}
            Their contribution to the left-hand side of \eqref{eq:ABC-six-tuple} is negligible. Hence there are $b_1,b_2\in B$ such that
            \begin{displaymath}
                |b_1-b_2|\geq\delta^{(65\eta+4\chi)/\tau}
                \geq\delta^\e,
                \qquad
                \mathcal E_{\sqrt{k}\delta}(A;\Theta,b_1-b_2)
                \geq\delta^{61\eta}|A||\Theta|^2.
            \end{displaymath}
            On the other hand, it follows from $5\chi\leq\e$ and \eqref{eq:indexed-energy-bound} that
            \begin{displaymath}
                \mathcal E_{\sqrt{k}\delta}(A;\Theta,b_1-b_2)
                \lesssim_k\delta^{2\e}|A||\Theta|^2.
            \end{displaymath}
            This is impossible for $\delta\ll_{k,\eta,\e}1$, since $61\eta<2\e$.

            If $(\alpha,\beta,\gamma)$ ranges in a compact subset of the parameter domain, then $\tau$ stays bounded away from zero and $\e$ can be chosen uniformly by Theorem \ref{thm:spread-out-implies-flattening}. The parameters $\eta$ and $\chi$ can therefore also be chosen uniformly.
        \end{proof}

    \section{Spread-out projections of sticky sets}

    \subsection{Main result and initial setup}

    \subsubsection{Main result}

    Recall the Frostman and regularity conditions from Definition \ref{def:Ahlfors-regular} and the paragraph preceding it.

    In this section, we identify
    \begin{displaymath}
        \R^d = \R^k \times \R^{d-k} = V \oplus W.
    \end{displaymath}
    For $\theta\in \mathbf{G}_{d,k}$, we let
    \begin{displaymath}
        \pi_{\theta}: V \oplus W \to V \cong \R^k, \qquad \pi_{\theta}(x,y) = x + \theta\cdot y.
    \end{displaymath}
    Thus,
    \begin{displaymath}
        \ker \pi_{\theta} = \{(-\theta\cdot w,w) : w \in W\},
    \end{displaymath}
    and so the fibers of $\pi_{\theta}$ are affine $(d-k)$-planes parallel to $\ker \pi_{\theta}$. This provides the identification, discussed in the introduction, between $\mathbf{G}_{d,k}$ and a Zariski-open and dense subset of the Grassmannian of $k$-planes in $\R^d$.

    \begin{theorem}\label{thm:AhlforsProjection}
        Let $1 \leq k < d$, let $s \in (0,1]$, $t \in (0,d]$, and $u \in [0,\min\{t,k\})$. Then there exists $\e=\e(k,d,s,t,u)>0$ such that the following holds for every $\delta\in2^{-\N}$ with $\delta\ll_{k,d,s,t,u}1$.

        Let $P\subset[-1,1]^d$ be $(\delta,t,\delta^{-\e})_d$-regular, and let $\Theta \subset [-1,1]^{k\times (d-k)}\subset\mathbf{G}_{d,k}$ be a $\delta$-separated $(\delta,s,\delta^{-\e})_{d,k}$-spread-out set. Then there exists $\Theta' \subset \Theta$ with $|\Theta'|_{\delta} \geq (1-\delta^{\e})|\Theta|_{\delta}$ such that for all $\theta \in \Theta'$,
        \begin{displaymath}
            |\pi_{\theta}(P')|_{\delta} \geq \delta^{-u} \qquad\textup{for all } P' \subset P,\ |P'|_{\delta} \geq \delta^{\e}|P|_{\delta}.
        \end{displaymath}
    \end{theorem}

    \subsubsection{High multiplicity sets}\label{s:highMultiplicity}

    We begin by defining the notion of ``multiplicity''. This terminology is taken from \cite{MR4388762}; here we extend it to higher dimensions. We denote
    \begin{displaymath}
        [x]_{\theta} = \pi_{\theta}^{-1}\{\pi_{\theta}(x)\} = x + \{ (-\theta\cdot w, w) : w \in W \} \quad\text{for } x \in \R^{d},\, \theta \in \mathbf{G}_{d,k}.
    \end{displaymath}

    \begin{definition}\label{def:mult} 
        Let $P \subset \R^{d}$, let $r\in2^{-\Z}$ and $R\in2^{-\Z}\cup\{\infty\}$ satisfy $r\leq R$, and let $x \in P$. For $\theta \in \mathbf{G}_{d,k}$, we define the following \emph{multiplicity number}:
        \begin{displaymath}
            \m_{P,\theta}(x \mid [r,R]) := \bigl|B(x,R) \cap P(r)\cap [x]_{\theta}\bigr|_{r}.
        \end{displaymath} 
        Thus, $\m_{P,\theta}(x \mid [r,R])$ is the smallest number of dyadic $r$-cubes needed to cover $B(x,R) \cap P(r)\cap [x]_{\theta}$.
    \end{definition}

    \begin{definition}[High multiplicity sets]\label{def:highMult}
        Let $r\in2^{-\Z}$ and $R\in2^{-\Z}\cup\{\infty\}$ satisfy $r\leq R$, let $M>0$, and let $\theta \in \mathbf{G}_{d,k}$. For $P \subset \R^{d}$, we define the \emph{high multiplicity set}
        \begin{displaymath}  
            \mathbf{H}_{\theta}(P,M, [r,R]) = \{x \in P : \m_{P,\theta}(x \mid [r,R]) \geq M\}.
        \end{displaymath}
    \end{definition}

    We now reformulate Theorem \ref{thm:AhlforsProjection} in the language of high-multiplicity sets:
    \begin{proposition}\label{prop:AhlforsProjection-multiplicity}
        Let $t \in (0,d]$, $s \in (0,1]$. For every $\sigma > \max\{t - k,0\}$ there exists $\e>0$ such that the following holds for every $\delta\in2^{-\N}$ with $\delta\ll_{s,t,\sigma}1$. Let $\mu$ be a $(\delta,t,\delta^{-\e})_d$-regular measure, and let $\nu$ be a finitely supported $(\delta,s,\delta^{-\e})_{d,k}$-spread-out probability measure on $[-1,1]^{k\times(d-k)}$. Then there exists $\Theta'\subset\spt\nu$ with $\nu(\Theta')\geq1-\delta^\e$ such that
        \begin{equation}\label{form7}
            \mu\bigl(B(0,1) \cap \mathbf{H}_{\theta}(\spt \mu,\delta^{-\sigma},[\delta,1])\bigr) \leq \delta^{\e}, \qquad \theta \in \Theta'.
        \end{equation}
    \end{proposition}

    \begin{remark}  \label{rem:single-projection-to-nearly-all}
        Proposition \ref{prop:AhlforsProjection-multiplicity} is formally equivalent to a version asserting that \eqref{form7} holds for a single $\theta\in\spt\nu$. Indeed, suppose this version holds with exponent $\e$, and consider the full conclusion with exponent $\e/2$. If it fails, the set $B$ of directions for which the left-hand side of \eqref{form7} exceeds $\delta^{\e/2}$ has $\nu(B)>\delta^{\e/2}$. By Lemma \ref{lem:spread-out-subset}, $\nu|_B/\nu(B)$ is $(\delta,s,\delta^{-\e})_{d,k}$-spread-out. The single direction claim applied to this measure contradicts the definition of $B$.
    \end{remark}

    \begin{remark}
        Proposition \ref{prop:AhlforsProjection-multiplicity} is analogous to \cite[Theorem 4.7]{OrponenShmerkin26}, and the proof follows the same general strategy: here we reduce Proposition \ref{prop:AhlforsProjection-multiplicity} to Proposition \ref{prop:ABC}, whereas \cite[Theorem 4.7]{OrponenShmerkin26} was reduced to another ABC sum-product theorem, namely \cite[Theorem 1.7]{OrponenShmerkin26}. The spread-out property of the direction measure $\nu$ simplifies one step: the branching-scale analysis needed in \cite[Theorem 4.7]{OrponenShmerkin26} is unnecessary here, because spread-outness is strongly scale invariant. On the other hand, we work with projections between arbitrary dimensions rather than from $\R^{2}$ to $\R$, which requires a number of technical adaptations. We have tried to keep the paper self-contained, but we still refer to \cite{OrponenShmerkin26} when the proofs differ only notationally.
    \end{remark}

    \subsubsection{Proof of Theorem \ref{thm:AhlforsProjection}}
    We conclude this section by proving Theorem \ref{thm:AhlforsProjection}, assuming Proposition \ref{prop:AhlforsProjection-multiplicity}.

    \begin{proof}[Proof of Theorem \ref{thm:AhlforsProjection} assuming Proposition \ref{prop:AhlforsProjection-multiplicity}]

        Dilating $\R^d$ by $1/(4\sqrt d)$, we may assume that $P\subset B(0,1/3)$. There is no loss of generality in assuming that $P$ has at most $1$ point in each $\delta$-cube.

        By the argument in Remark \ref{rem:single-projection-to-nearly-all}, it suffices to find a single point $\theta \in \Theta$ such that the conclusion of the theorem holds.

        Fix $u \in [0,\min\{t,k\})$, and set $\sigma=[\max\{t-k,0\}+(t-u)]/2$. By Lemma \ref{lem:spread-out-discretization}\textup{\ref{it:spread-out-discretization-cubes}}, we may apply Proposition \ref{prop:AhlforsProjection-multiplicity} to $m_P$ and $m_\Theta$ at scale $2\delta$, with $s/4,2\e$ in place of $s,\e$. For $\e\ll_{d,k,s,t,u}1$ and $\delta\ll_{d,k,s,t,u,\e}1$, its hypotheses hold, so there is $\theta\in\Theta$ such that
        \begin{equation} \label{eq:application-proj-multiplicity}
            m_P\bigl(\mathbf{H}_{\theta}(P,\delta^{-\sigma},[2\delta,1])\bigr) \leq (2\delta)^{2\e}.
        \end{equation}
        Let $P' \subset P$ be a subset with $|P'| \geq \delta^{\e}|P|$. Since $(2\delta)^{2\e} \leq \tfrac{1}{2}\delta^{\e}$ for $\delta\ll_\e1$, \eqref{eq:application-proj-multiplicity} yields a further subset $P'' \subset P'$ with $|P''| \geq \tfrac{1}{2}\delta^{\e}|P| \gtrsim \delta^{2\e - t}$ such that $\mathbf{H}_{\theta}(P,\delta^{-\sigma},[2\delta,1])\cap P'' = \varnothing$. This means that, for every $p\in P''$, we have
        \begin{displaymath}
            |P(2\delta) \cap [p]_{\theta}|_{2\delta} < \delta^{-\sigma},
        \end{displaymath}
        where we omitted $B(p,1)$ since $P\subset B(0,1/3)$ and hence $P(2\delta)\subset B(p,1)$ for all $p\in P$ when $\delta\ll_d1$. Hence,
        \begin{displaymath}
            \bigl|P(\delta)\cap [p]_{\theta}(\delta)\bigr|_\delta \lesssim \delta^{-\sigma} \qquad \text{for all } p \in P''.
        \end{displaymath}
        This implies that
        \begin{displaymath}
            |\pi_{\theta}(P'')|_{\delta} \gtrsim \delta^{\sigma}|P''| \gtrsim \delta^{\sigma + 2\e-t} \geq \delta^{-u},
        \end{displaymath}
        where the last inequality holds if $\e\ll_{s,t,u}1$. Since $P'' \subset P'$, the proof is complete.
    \end{proof}

    \subsubsection{An inductive scheme to prove Proposition \ref{prop:AhlforsProjection-multiplicity}}\label{s:induction}

    We use the following notation to organize the iteration in $\sigma$.
    \begin{terminology}[$\mathsf{Projection}(s,t;\sigma)$]
        We write $\mathsf{Projection}(s,t;\sigma)$ for the assertion of Proposition \ref{prop:AhlforsProjection-multiplicity} with parameters $s,t,\sigma$.
    \end{terminology}

        The proof of Proposition \ref{prop:AhlforsProjection-multiplicity} will be based on iterating the following key proposition:
        \begin{proposition}\label{prop:AhlforsProjections-induction} 
            Let $t \in (0,d)$, $s \in (0,1]$, and $\max\{t - k,0\} < \sigma \le d-k$. Then there exists $\zeta = \zeta(s,t,\sigma) > 0$, which stays bounded away from $0$ as long as $\sigma$ stays bounded away from $\max\{t - k,0\}$, such that $\mathsf{Projection}(s,t;\sigma) \Longrightarrow \mathsf{Projection}(s,t;\sigma - \zeta)$. If $\sigma=d-k$, then $\mathsf{Projection}(s,t;d-k-\zeta)$ holds.
        \end{proposition}

        We then complete the proof of Proposition \ref{prop:AhlforsProjection-multiplicity} assuming Proposition \ref{prop:AhlforsProjections-induction}.

        \begin{proof}[Proof of Proposition \ref{prop:AhlforsProjection-multiplicity} assuming Proposition \ref{prop:AhlforsProjections-induction}]

            Fix $t \in (0,d]$ and $s \in (0,1]$. Proposition \ref{prop:AhlforsProjection-multiplicity} is trivial for $t = d$, since for every $\sigma > d-k = \max\{t - k,0\}$, the high-multiplicity sets in question are empty for $0<\delta\ll_\sigma1$. So, we may assume that $t \in (0,d)$.

            Let $\Sigma(s,t) \geq \max\{t - k,0\}$ be the infimum over $\sigma > \max\{t - k,0\}$ for which $\mathsf{Projection}(s,t;\sigma)$ holds. Since $\mathbf{H}_{\theta}(P,\delta^{-\sigma'},[\delta,1]) \subset \mathbf{H}_{\theta}(P,\delta^{-\sigma},[\delta,1])$ for $\sigma'\geq\sigma$, this set of parameters is upward closed, and the geometric bound $\m_{P,\theta}(x\mid[\delta,1])\lesssim_d\delta^{-(d-k)}$ gives $\Sigma(s,t)\leq d-k$.

            Suppose that $\Sigma(s,t)>\max\{t-k,0\}$. If $\Sigma(s,t)=d-k$, the endpoint assertion of Proposition \ref{prop:AhlforsProjections-induction} yields $\mathsf{Projection}(s,t;d-k-\zeta)$ for some $\zeta>0$, a contradiction. If $\Sigma(s,t)<d-k$, choose $\sigma\in(\Sigma(s,t),d-k)$ with $0<\sigma-\Sigma(s,t)\ll_{s,t}1$. Then $\mathsf{Projection}(s,t;\sigma)$ holds, while the uniform lower bound for $\zeta(s,t,\sigma)$ near $\Sigma(s,t)$ ensures that $\sigma-\zeta(s,t,\sigma)<\Sigma(s,t)$. Proposition \ref{prop:AhlforsProjections-induction} again contradicts the definition of $\Sigma(s,t)$.
        \end{proof}

        \subsection{Proof of Proposition \ref{prop:AhlforsProjections-induction}}
    
        \subsubsection{Small slices imply sparse slices}

            We use the following local variant of the high multiplicity sets. Fix $C_{\mathrm{loc}}\geq 4^{d-k}$ so that every affine $(d-k)$-plane intersected with a ball of radius $R$ meets fewer than $C_{\mathrm{loc}}(R/r)^{d-k}$ dyadic $r$-cubes for $0<r\leq R$.
            \begin{definition}[Local high multiplicity sets]\label{def:local-high-multiplicity}
                Let $\delta \in (0,\tfrac{1}{2}]$, $\rho \in (0,1]$, $\sigma \in (0,d-k]$, and $\theta \in \mathbf{G}_{d,k}$. For $P \subset \R^{d}$, define
                \begin{equation*}
                    \mathbf{H}_{\theta,\mathrm{loc}}(P,\sigma,\delta,\rho) := \bigcup_{\delta \leq r \leq R \leq 8} \mathbf{H}_{\theta}(P,C_{\mathrm{loc}}(R/r)^{\sigma},[r,R]),
                \end{equation*}
                where the union is over dyadic $r,R \in [\delta,8]$ satisfying $r/R \leq \rho$.
            \end{definition}

            The following is a higher dimensional, measure version of \cite[Theorem 4.21]{OrponenShmerkin26}.
            \begin{lemma}\label{lem:local-slice}
                There is a constant $C_{\mathrm{slice}}\geq12$ with the following property.
                
                Suppose $\mathsf{Projection}(s,t;\sigma)$ holds with constants $\Delta_{\mathrm{proj}},\e_{\mathrm{proj}}>0$. Let $0<\eta\leq1$, $0<a\leq\eta\e_{\mathrm{proj}}/C_{\mathrm{slice}}$, and $b>0$ satisfy $2b+6a<\eta\e_{\mathrm{proj}}$.
                
                If $\delta\in 2^{-\N}$ satisfies $\delta\ll_{s,t,\sigma,\Delta_{\mathrm{proj}},\e_{\mathrm{proj}},\eta,a,b}1$, $\mu$ is $(\delta,t,\delta^{-b})_d$-regular, and $\nu$ is a finitely supported $(\delta,s,\delta^{-b})_{d,k}$-spread-out probability measure on $[-1,1]^{k\times(d-k)}$, then there is $G\subset\spt\nu$ with $\nu(G)\geq1-\delta^a$ such that
                \begin{displaymath}
                    \mu\bigl(B(0,1)\cap\mathbf H_{\theta,\mathrm{loc}}(\spt\mu,\sigma,\delta,\delta^\eta)\bigr)\leq\delta^a,
                    \qquad \theta\in G.
                \end{displaymath}
            \end{lemma}
            We indicate the modifications required to adapt the proof of \cite[Theorem 4.21]{OrponenShmerkin26}. Replace lines and tubes by affine $(d-k)$-planes and $k$-codimensional slabs, and replace each cardinality pigeonhole in the plane set by a $\nu$-mass pigeonhole. After restricting to the bad directions and fixing the dyadic radii in Definition \ref{def:local-high-multiplicity}, the normalized plane measure has spread-out constant at most $\delta^{-(b+6a)}$. For the ball $B=B(z,R)$ selected in that proof, the measure
            \begin{displaymath}
                (4R)^{-t}\bigl(x\mapsto(x-z)/(4R)\bigr)\mu
            \end{displaymath}
            is $(\delta/(4R),t,\delta^{-b})_d$-regular. At the projection scale $\Delta=r/(4R)\in[\delta/32,\delta^\eta/4]$, it is therefore $(\Delta,t,\Delta^{-b/\eta})_d$-regular, while the normalized plane measure is $(\Delta,s,\Delta^{-(b+6a)/\eta})_{d,k}$-spread-out. The inequality $b+6a<\eta\e_{\mathrm{proj}}$ yields both hypotheses in $\mathsf{Projection}(s,t;\sigma)$. If $\Delta<\delta$, lowering the base scale adds at most five singular scales. Finally, $C_{\mathrm{loc}}\geq4^{d-k}\geq4^\sigma$ absorbs the geometric covering constant. The remainder of the proof of \cite[Theorem 4.21]{OrponenShmerkin26} is unchanged.

            \subsubsection{Fixing parameters}\label{s:parameters}

            Fix $(s,t,\sigma)$ as in Proposition \ref{prop:AhlforsProjections-induction}; in particular, $t\in(0,d)$, $s\in(0,1]$, and $\max\{t-k,0\}<\sigma\leq d-k$. Set
            \begin{align*}
                \sigma_{0} & = \tfrac{1}{2}(\max\{t-k,0\} + \min\{\sigma,t\}), \\
                \alpha     & = t-\sigma_{0}.
            \end{align*}
            Then $\alpha \in [0,k)$ and $\sigma \in (0,d-k]$. Apply Proposition \ref{prop:ABC} with parameters $\alpha,\sigma,s/8$, and let $\chi=\chi(\alpha,\sigma,s/8)>0$ and $\Delta_{\mathrm{ABC}}>0$ be the constants provided by Proposition \ref{prop:ABC}.

As long as $\sigma$ stays bounded away from $\max\{t-k,0\}$, the triple $(t-\sigma_0,\sigma,s/8)$ ranges in a compact subset of the domain of Proposition \ref{prop:ABC}. Its uniformity assertion therefore shows that $\chi$ stays bounded away from zero.

Set $\zeta := \tfrac{1}{100}\min\{\sigma-\sigma_0,\chi\}$. This stays bounded away from $0$ as long as $\sigma$ stays bounded away from $\max\{t-k,0\}$. Moreover,
\begin{equation}\label{eq:zeta-def-consequences}
    4\zeta \leq \sigma-\sigma_0 \qquad\text{and}\qquad 25\zeta \leq \chi.
\end{equation}
If $\sigma<d-k$, assume that $\mathsf{Projection}(s,t;\sigma)$ holds with constants $\Delta_{\mathrm{proj}},\e_{\mathrm{proj}}>0$. If $\sigma=d-k$, set $\Delta_{\mathrm{proj}}=\e_{\mathrm{proj}}=1$. In either case, suppose that $\mathsf{Projection}(s,t;\sigma-\zeta)$ fails. Choose $\e\in(0,\tfrac12]$ so that
\begin{equation}\label{eq:epsilon-def}
    \e \leq \min\left\{\frac{\e_{\mathrm{proj}}^2}{100C_{\mathrm{slice}}^2},\frac{\zeta}{100},\frac{\zeta^2}{16d^2}\right\}.
\end{equation}
Then $(d-k)\sqrt\e\leq\zeta/4$, and
\begin{equation}\label{eq:spread-out-loss}
    24\e+5\zeta\leq\chi/10.
\end{equation}
The failure of $\mathsf{Projection}(s,t;\sigma-\zeta)$ gives some $\delta\in2^{-\N}$ with $\delta\ll_{d,k,s,t,\sigma,\Delta_{\mathrm{ABC}},\Delta_{\mathrm{proj}},\e_{\mathrm{proj}},C_{\mathrm{slice}},\e,\zeta}1$, a $(\delta,t,\delta^{-\e})_d$-regular measure $\mu$ with support $P$, and a finitely supported $(\delta,s,\delta^{-\e})_{d,k}$-spread-out probability measure on $[-1,1]^{k\times(d-k)}$ for which the set of directions violating the conclusion has measure greater than $\delta^\e$. Restrict and normalize this measure to that set, call the resulting measure $\nu$, and set $\Theta:=\spt\nu$. By Lemma \ref{lem:spread-out-subset}, $\nu$ is $(\delta,s,\delta^{-2\e})_{d,k}$-spread-out and
\begin{equation}\label{form9}
    \mu\big(B(0,1) \cap \mathbf{H}_{\theta}(P,\delta^{-\sigma + \zeta},[\delta,1])\big) \geq \delta^{\e}, \qquad \theta \in \Theta.
\end{equation}
Thus the parameters are fixed in the order
\begin{displaymath}
    (d,k,C_{\mathrm{slice}}),\quad(s,t,\sigma),\quad(\sigma_0,\alpha),\quad
    (\chi,\Delta_{\mathrm{ABC}}),\quad\zeta,\quad(\e_{\mathrm{proj}},\Delta_{\mathrm{proj}}),
    \quad\e,\quad\delta.
\end{displaymath}

\subsubsection{\texorpdfstring{The sets $P_{\theta}$}{The sets P theta}}

If $\sigma=d-k$, the choice of $C_{\mathrm{loc}}$ implies that $\mathbf H_{\theta,\mathrm{loc}}(P,\sigma,\delta,\delta^{\sqrt\e})$ is empty, so set $G:=\Theta$. If $\sigma<d-k$, apply Lemma \ref{lem:local-slice} with $\eta=\sqrt\e$, $a=5\e$, and $b=2\e$; its hypotheses follow from \eqref{eq:epsilon-def}. In either case, there is $G\subset\Theta$ with $\nu(G)\geq1-\delta^{5\e}$ and
\begin{equation} \label{form51}
    \mu\bigl(B(0,1) \cap \mathbf{H}_{\theta,\mathrm{loc}}(P,\sigma,\delta,\delta^{\sqrt{\e}})\bigr) \leq \delta^{5\e}, \qquad \theta\in G.
\end{equation}
Replace $\nu$ by $\nu|_G/\nu(G)$ and $\Theta$ by its support. Lemma \ref{lem:spread-out-subset} shows that $\nu$ is $(\delta,s,\delta^{-3\e})$-spread-out. For $\theta \in \Theta$, define
\begin{displaymath}
    P_{\theta} := B(0,1) \cap \mathbf{H}_{\theta}(P,\delta^{-\sigma + \zeta},[\delta,1]) \, \setminus \, \mathbf{H}_{\theta,\mathrm{loc}}(P,\sigma,\delta,\delta^{\sqrt{\e}}).
\end{displaymath}
For $\delta\ll_\e1$, equations \eqref{form9} and \eqref{form51} imply $\mu(P_{\theta}) \geq \tfrac{1}{2}\delta^{\e}$ for $\theta \in \Theta$.

\subsubsection{\texorpdfstring{Fixing a good $\sqrt{\delta}$-cube $I_{0}$}{Fixing a good cube}}\label{s:I0}

Set $\Delta:=\sqrt\delta$ and apply Lemma \ref{lem:spread-out-sub-cube} to $\nu$ with $\beta=s$. We obtain $I_0\in\mathcal D_\Delta(\Theta)$ such that, for $\delta\ll_{s,\e}1$, the measure  $S_{I_0}\,\frac{\nu|_{I_0}}{\nu(I_0)}$ is $(\Delta,s/2,\delta^{-4\e})_{d,k}$-spread-out.

\subsubsection{\texorpdfstring{Projecting the sets $P_{\theta}$}{Projecting the sets P theta}}

The inclusion $P_{\theta} \subset \mathbf{H}_{\theta}(P,\delta^{-\sigma + \zeta},[\delta,1])$ yields an upper bound on the $\delta$-covering number of $\pi_{\theta}(P_{\theta})$.
\begin{lemma}\label{lemma6}
    Let $\theta \in \Theta$, let $\rho \in [\delta^{1 - \sqrt{\e}},\tfrac12\delta^{\sqrt{\e}}]$, and let $\mathbf{T}=\pi_{\theta}^{-1}\{x\} + B(0,\rho)$ for some $x\in\R^k$. Then,
    \begin{equation}\label{eq:projection-slab-covering}
        |\pi_{\theta}(P_{\theta} \cap \mathbf{T})|_{\delta} \lesssim \delta^{-\e - \zeta} \cdot \left(\tfrac{\rho}{\delta} \right)^{t - \sigma}.
    \end{equation}
    Moreover, $|\pi_{\theta}(P_{\theta})|_{\delta} \lesssim \delta^{-\e - \zeta} \cdot \delta^{\sigma - t}$.
\end{lemma}
\begin{proof}
Choose a maximal $(2d^2\delta)$-separated subset of $\pi_\theta(P_\theta\cap\mathbf T)$, and choose a point of $P_\theta\cap\mathbf T$ above each element. By the definition of $P_\theta$, each corresponding fiber meets $B(0,2)\cap P(\delta)\cap\mathbf T$ in at least $\delta^{-\sigma+\zeta}$ dyadic $\delta$-cubes. The separation ensures that these cube families are disjoint, so
\begin{equation}\label{eq:projection-slab-incidence}
    |\pi_\theta(P_\theta\cap\mathbf T)|_\delta\,\delta^{-\sigma+\zeta}
    \lesssim |B(0,2)\cap P(\delta)\cap\mathbf T|_\delta.
\end{equation}
This inequality also holds with $\mathbf T=\R^d$.

To prove \eqref{eq:projection-slab-covering}, assume that $\rho\leq\delta^{\sqrt\e}$ is dyadic. If $P_\theta\cap\mathbf T=\varnothing$, there is nothing to prove. Otherwise, fix $p\in P_\theta\cap\mathbf T$. Since $p\notin\mathbf H_{\theta,\mathrm{loc}}(P,\sigma,\delta,\delta^{\sqrt\e})$, the scales $[8\rho,8]$ yield
\begin{displaymath}
    |B(p,8)\cap P(8\rho)\cap[p]_\theta|_{8\rho}
    <C_{\mathrm{loc}}\rho^{-\sigma}.
\end{displaymath}
Orthogonal projection onto $[p]_\theta$ moves every point of $B(0,2)\cap P(\delta)\cap\mathbf T$ by at most $2\rho$ into this set. Thus $B(0,2)\cap P(\delta)\cap\mathbf T$ is covered by $O_d(\rho^{-\sigma})$ balls of radius $O_d(\rho)$. Regularity and \eqref{eq:projection-slab-incidence} yield
\begin{displaymath}
    |\pi_\theta(P_\theta\cap\mathbf T)|_\delta\,\delta^{-\sigma+\zeta}
    \lesssim\delta^{-\e}\rho^{-\sigma}(\rho/\delta)^t,
\end{displaymath}
which proves \eqref{eq:projection-slab-covering}. Taking $\mathbf T=\R^d$ in \eqref{eq:projection-slab-incidence} and using $|B(0,2)\cap P(\delta)|_\delta\lesssim\delta^{-\e-t}$ proves the global bound.
\end{proof}

\subsubsection{\texorpdfstring{Finding a good projection $\pi_{\theta_{0}}$}{Finding a good projection}}

Since $\mu(B(0,1))\leq\delta^{-\e}$ and $\mu(P_\theta)\geq\tfrac12\delta^\e$, Cauchy--Schwarz gives
\begin{displaymath}
    \int_{I_0}\int_{I_0}\mu(P_\theta\cap P_{\theta'})\,d\nu(\theta)\,d\nu(\theta')
    \geq\delta^{4\e}\nu(I_0)^2
\end{displaymath}
for $\delta\ll_\e1$. Hence there are $\theta_0\in\Theta\cap I_0$ and $E\subset\Theta\cap I_0$ such that $\nu(E)\geq\delta^{6\e}\nu(I_0)$ and
\begin{equation}\label{form37}
    \mu(P_{\theta_{0}} \cap P_{\theta}) \geq \delta^{5\e}, \qquad \theta \in E.
\end{equation}
Replace $\nu$ by $\nu|_E/\nu(E)$ and $\Theta$ by its support. Since $\nu(E)\geq\delta^{6\e}\nu(I_0)$, Lemma \ref{lem:spread-out-subset} shows that the pushforward of $\nu$ under $S_{I_0}$ is $(\Delta,s/2,\delta^{-10\e})_{d,k}$-spread-out.

Let $T(x,y):=(x+\theta_0\cdot y,y)$. The identity $\pi_{\theta-\theta_0}(T(x,y))=\pi_\theta(x,y)$ shows that $T$ preserves the fibers after replacing $\theta$ by $\theta-\theta_0$. Since $|\theta_0|=O_d(1)$, we have that $T\mu$ is $(\delta,t,\delta^{-2\e})_d$-regular for $\delta\ll_{d,k,\e}1$.
Push $P$ and $\mu$ forward by $T$, push $\nu$ forward under $\theta\mapsto\theta-\theta_0$, and define the transformed family by $P_{\theta-\theta_0}:=T(P_\theta)$. We keep the notation $\nu,\Theta,P,\mu$, and $P_\theta$. Thus $\theta_0=0$, the pushforward of $\nu$ under $\theta\mapsto\Delta^{-1}\theta$ is $(\Delta,s/2,\delta^{-10\e})_{d,k}$-spread-out, and \eqref{form37} holds for every $\theta\in\Theta$. We write $P_0:=P_{\theta_0}$ and $\pi:=\pi_0$. This change of coordinates introduces $O_d(1)$ constants in several estimates below, which can be absorbed into powers of $\delta$.

Recall that $[x]_{\theta} := \pi_{\theta}^{-1}\{\pi_{\theta}(x)\}$ and $\R^d = V \oplus W$, where $\dim V=k$ and $\dim W=d-k$. Thus, $[x]_0=x+W$. Moreover, $\|\theta\|\lesssim_{d,k}\Delta$ for $\theta\in\Theta$, so $[x]_{\theta}$ forms an angle $O(\Delta)$ with $W$.

\subsubsection{\texorpdfstring{Finding a good $\sqrt{\delta}$-slab $\mathbf{T}_{0}$}{Finding a good slab}}

We define a $\rho$-slab to be a set of the form $\mathbf{T}_{\theta,x}(\rho)=\pi_{\theta}^{-1}\{x\} + B(0,\rho)$, where $x\in\R^k$ and $\theta\in \mathbf{G}_{d,k}$. Given a slab $\mathbf{T}=\mathbf{T}_{\theta,x}(\rho)$, we define $C\mathbf{T}=\mathbf{T}_{\theta,x}(C\rho)$.

Fix a dimensional constant $C_d\gg_d1$ such that $P_\theta\subset B(0,C_d)$ for every $\theta\in\Theta$. Let $\mathcal{B}_{\Delta}$ be the cover of $B(0,C_d) \cap P$ by balls of radius $\Delta$ whose centers form a maximal $\Delta$-separated subset of $B(0,C_d) \cap P$. Then $\mathcal{B}_{\Delta}$ has bounded overlap, and regularity gives $|\mathcal{B}_{\Delta}| \lesssim \delta^{-2\e}\Delta^{-t}=\Delta^{-t-4\e}$.

We note that
\begin{equation}\label{form38}
    \int_{\Theta} \sum_{B \in \mathcal{B}_{\Delta}} \mu(P_{0} \cap P_{\theta} \cap B) \, d\nu(\theta) \geq \int_{\Theta} \mu(P_{0} \cap P_{\theta}) \, d\nu(\theta) \stackrel{\eqref{form37}}{\geq} \delta^{5\e}.
\end{equation}
Let $\mathcal{B}_{\Delta}^{\mathrm{heavy}}$ consist of the balls $B\in\mathcal B_\Delta$ such that
\begin{displaymath}
    \int_{\Theta}\mu(P_{0}\cap P_{\theta}\cap B)\,d\nu(\theta)>\Delta^{t+20\e}.
\end{displaymath}
The contribution of the remaining balls to the left-hand side of \eqref{form38} is at most $|\mathcal B_\Delta|\Delta^{t+20\e}\lesssim \Delta^{16\e}$. Hence,
\begin{equation}\label{form43}
    \int_{\Theta} \sum_{B \in \mathcal{B}_{\Delta}^{\mathrm{heavy}}} \mu(P_{0} \cap P_{\theta} \cap B) \, d\nu(\theta) \geq \delta^{6\e}.
\end{equation}
For $B\in\mathcal B_\Delta^{\mathrm{heavy}}$, the higher-dimensional form of \cite[Lemma 4.55]{OrponenShmerkin26}, which follows with the same proof, yields
\begin{displaymath}
    |\pi(B\cap P_0)|_\delta\gtrsim \delta^{2\e-t}\left(\frac{\delta}{\Delta}\right)^\sigma\mu(B\cap P_0).
\end{displaymath}
The definition of $\mathcal B_\Delta^{\mathrm{heavy}}$ implies $\mu(B\cap P_0)>\Delta^{t+20\e}$; hence, for $\delta\ll_{d,k,t,\sigma,\e}1$,
\begin{equation}\label{form40} |\pi(B \cap P_{0})|_{\delta} \geq \Delta^{\sigma - t+25\e}, \qquad B \in \mathcal{B}_{\Delta}^{\mathrm{heavy}}. \end{equation}
Let $\mathcal{Q}_{\Delta} \subset \mathcal{D}_{\Delta}(\R^{k})$ be the collection of dyadic $\Delta$-cubes $Q$ such that $Q \cap \pi(B) \neq \varnothing$ for some $B \in \mathcal{B}_{\Delta}^{\mathrm{heavy}}$, and define $\mathbf T_Q:=\pi^{-1}\{\operatorname{center}(Q)\}+B(0,C_d\Delta)$. Since $\pi(B)$ has diameter $O(\Delta)$ for every ball $B$ of radius $\Delta$, every heavy ball meets at least one slab $\mathbf{T}_{Q}$ with $Q \in \mathcal{Q}_{\Delta}$, and at most $O(1)$ such slabs. We write
    $\mathcal{T}_{\Delta}:=\{\mathbf{T}_{Q} : Q \in \mathcal{Q}_{\Delta}\}$.
We claim that
\begin{equation}\label{form42}
    |\mathcal{T}_{\Delta}| \leq \Delta^{\sigma - t - 3\zeta}.
\end{equation}
Indeed, for each $Q \in \mathcal{Q}_{\Delta}$, choose one heavy ball $B_Q \in \mathcal{B}_{\Delta}^{\mathrm{heavy}}$ with $Q \cap \pi(B_Q) \neq \varnothing$. Taking $C_d\gg_d1$, the set $\pi(B_Q \cap P_0)$ is contained in $C_dQ$, so by \eqref{form40},
\begin{displaymath}
    |\pi(P_0) \cap C_dQ|_{\delta} \geq |\pi(B_Q \cap P_0)|_{\delta} \geq \Delta^{\sigma - t+25\e}.
\end{displaymath}
Since the enlarged cubes $C_dQ$, $Q \in \mathcal{Q}_{\Delta}$, have bounded overlap, it follows that
\begin{displaymath}
    |\pi(P_{0})|_{\delta} \gtrsim |\mathcal{T}_{\Delta}|\Delta^{\sigma - t+25\e}.
\end{displaymath}
On the other hand, Lemma \ref{lemma6}, with $2\e$ in place of $\e$, gives $|\pi(P_0)|_\delta\lesssim\Delta^{2\sigma-2t-4\e-2\zeta}$. Combining the last two estimates yields $|\mathcal T_\Delta|\lesssim\Delta^{\sigma-t-29\e-2\zeta}$, which implies \eqref{form42} by \eqref{eq:epsilon-def}.

\begin{figure}[htbp]
\centering
\begin{tikzpicture}[scale=1.05,transform shape,every node/.style={font=\small}]
    \begin{scope}[shift={(0.15,0.45)},rotate=8]
        \foreach \y/\lab in {0/{\mathbf T_1},1.02/{\mathbf T_2},2.42/{\mathbf T_3},4.00/{\mathbf T_4}} {
            \fill[black!2] (0,\y) rectangle (4.2,\y+0.7);
            \draw[thick] (0,\y) rectangle (4.2,\y+0.7);
            \foreach \x in {0.7,1.4,2.1,2.8,3.5} {
                \draw[black!25] (\x,\y) -- (\x,\y+0.7);
            }
            \node[anchor=west] at (4.4,\y+0.35) {$\lab$};
        }

        \draw[<->,black!60] (0,0.83) -- (4.2,0.83)
            node[midway,below,fill=white,inner sep=1pt,font=\scriptsize] {$1$};
        \draw[<->,black!60] (-0.18,0) -- (-0.18,0.7)
            node[midway,left,fill=white,inner sep=1pt,font=\scriptsize] {$\Delta$};

        \foreach \x/\y in {0/0,1.4/0,2.8/0,
                0.7/1.02,2.1/1.02,3.5/1.02,
                0/2.42,0.7/2.42,3.5/2.42,
                1.4/4.00,2.1/4.00,2.8/4.00} {
            \foreach \dy in {0.16,0.35,0.55} {
                \fill[purple!16]
                    (\x+0.08,\y+\dy-0.045) rectangle (\x+0.62,\y+\dy+0.045);
                \draw[purple!32,line width=0.35pt]
                    (\x+0.08,\y+\dy-0.045) rectangle (\x+0.62,\y+\dy+0.045);
            }
        }

        \foreach \x/\y/\a/\b/\c in {
                0/0/0.16/0.43/0.55,
                1.4/0/0.49/0.20/0.34,
                2.8/0/0.30/0.57/0.23,
                0.7/1.02/0.54/0.24/0.42,
                2.1/1.02/0.18/0.48/0.57,
                3.5/1.02/0.45/0.58/0.21,
                0/2.42/0.37/0.19/0.58,
                0.7/2.42/0.21/0.53/0.35,
                3.5/2.42/0.60/0.30/0.48,
                1.4/4.00/0.24/0.55/0.38,
                2.1/4.00/0.50/0.28/0.60,
                2.8/4.00/0.35/0.61/0.22} {
            \fill[purple!75] (\x+\a,\y+0.16) circle (1.6pt);
            \fill[purple!75] (\x+\b,\y+0.35) circle (1.6pt);
            \fill[purple!75] (\x+\c,\y+0.55) circle (1.6pt);
        }

        \draw[->,thick] (0,-0.42) -- (1.55,-0.42)
            node[midway,below] {$W$};

        \draw[thick] (5.45,-0.05) -- (5.45,4.75);
        \node[anchor=south west] at (5.55,4.55) {$\pi(P_0)$};
        \foreach \y in {0.16,0.35,0.55,
                1.18,1.37,1.57,
                2.58,2.77,2.97,
                4.16,4.35,4.55} {
            \fill[red!75] (5.45,\y) circle (1.8pt);
        }
    \end{scope}
\end{tikzpicture}
\caption{The $\Delta$-slabs and their heavy $\Delta$-balls. Each heavy ball has a large $\delta$-scale projection by \eqref{form40}; comparison with the upper bound for $|\pi(P_0)|_\delta$ yields \eqref{form42}.}
\end{figure}
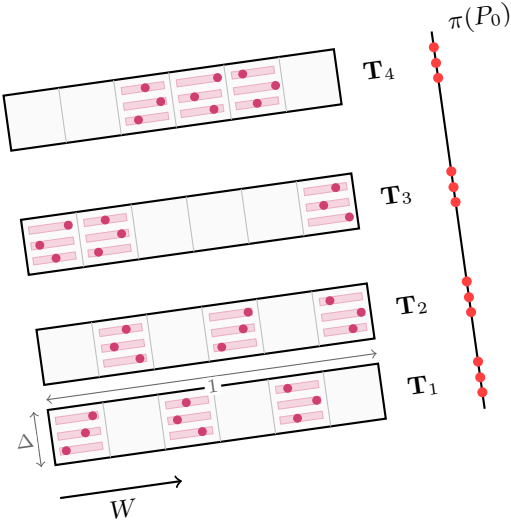

Next, for every $\mathbf{T} \in \mathcal{T}_{\Delta}$, write $\mathcal{B}(\mathbf{T}) := \{B \in \mathcal{B}_{\Delta}^{\mathrm{heavy}} : B \cap \mathbf{T} \neq \varnothing\}$. Since the union of the slabs in $\mathcal{T}_{\Delta}$ covers all the heavy balls,
\begin{equation}\label{form44}
    \sum_{\mathbf{T} \in \mathcal{T}_{\Delta}} \int_{\Theta} \sum_{B \in \mathcal{B}(\mathbf{T})} \mu(P_{0} \cap P_{\theta} \cap B) \, d\nu(\theta) \stackrel{\eqref{form43}}{\geq} \delta^{6\e}.
\end{equation}
By \eqref{form42}, \eqref{form44}, and $12\e\leq\zeta$, there exists $\mathbf T_0\in\mathcal T_\Delta$ such that
\begin{equation}\label{form64}
    \int_{\Theta} \sum_{B \in \mathcal{B}(\mathbf{T}_0)} \mu(P_{0} \cap P_{\theta} \cap B) \, d\nu(\theta) \geq \Delta^{t - \sigma + 4\zeta}.
\end{equation}
For $\theta\in\Theta$ and $C_d\gg_{d,k}1$, the set $P_\theta\cap2\mathbf T_0$ is contained in $\mathbf T_{\theta,x}(C_d\Delta)$ for a suitable $x\in\R^k$, since $P_\theta\subset B(0,C_d)$ and $\|\theta\|\lesssim_{d,k}\Delta$. We deduce from Lemma \ref{lemma6} that
\begin{equation}\label{eq:sparse-projection-upper}
    |\pi_{\theta}(P_{\theta} \cap 2\mathbf{T}_{0})|_{\delta} \lesssim \Delta^{\sigma-t-4\e-2\zeta} \leq \Delta^{\sigma - t - 4\zeta}, \qquad \theta \in \Theta.
\end{equation}

\begin{lemma}\label{lemma8}
    We have
    \begin{displaymath}
        |\{B\in\mathcal{B}(\mathbf{T}_{0}):B\cap B(x,R)\neq\varnothing\}| \lesssim \left(\frac{R}{\Delta} \right)^{\sigma}, \qquad x \in \R^{d}, \, R \in [\delta^{-\sqrt{\e}}\Delta,1].
    \end{displaymath}
    In particular, $|\mathcal{B}(\mathbf{T}_{0})| \lesssim \Delta^{-\sigma}$.
\end{lemma}
The lemma follows from the counting argument in \cite[Lemma 4.75]{OrponenShmerkin26}, with $\Delta$-balls in a slab replacing disks in a tube; the argument is the same.

By Lemma \ref{lemma8} and the bound $\mu(B)\lesssim\Delta^{t-4\e}$ for $B\in\mathcal B(\mathbf T_0)$,
\begin{displaymath}
    \sum_{B\in\mathcal B(\mathbf T_0)}\mu(B\cap P_0\cap P_\theta)\lesssim\Delta^{t-\sigma-4\e},\qquad \theta\in\Theta.
\end{displaymath}
Let $\overline\Theta\subset\Theta$ consist of the parameters for which the sum on the left is at least $\Delta^{t-\sigma+5\zeta}$. Comparing this pointwise bound with \eqref{form64} gives
\begin{equation}\label{form11}
    \nu(\overline\Theta)\geq\Delta^{5\zeta}.
\end{equation}
By definition,
\begin{equation}\label{form67}
    \sum_{B\in\mathcal B(\mathbf T_0)}\mu(B\cap P_0\cap P_\theta)\geq\Delta^{t-\sigma+5\zeta},\qquad \theta\in\overline\Theta.
\end{equation}

\subsubsection{\texorpdfstring{The sets $A$ and $A_{\theta}$}{The sets A and Atheta}}

Let $A$ be the set of lower-left endpoints of the cubes in $\mathcal{D}_{\delta}(\pi(P_{0}\cap 2\mathbf{T}_{0}))$, and write $I_x$ for the cube with lower-left endpoint $x\in A$. Then
\begin{equation}\label{form71}
    |A|=|\pi(P_{0}\cap 2\mathbf{T}_{0})|_{\delta}
    \stackrel{\eqref{eq:sparse-projection-upper}}{\leq}\Delta^{\sigma-t-4\zeta}.
\end{equation}
For $\theta\in\overline{\Theta}$, define
\begin{equation}\label{form75}
    A_{\theta}:=\left\{x\in A:
    |\{B\in\mathcal{B}(\mathbf{T}_{0}):\pi^{-1}(I_x)\cap B\cap P_{0}\cap P_{\theta}\neq\varnothing\}|
    \geq\Delta^{-\sigma+10\zeta}\right\}.
\end{equation}

\begin{lemma}\label{lem:Atheta-lower-bound}
    For $\theta\in\overline{\Theta}$,
    \begin{equation}\label{form70}
        |A_{\theta}|\geq\Delta^{\sigma-t+6\zeta}.
    \end{equation}
\end{lemma}

\begin{proof}[Proof of Lemma \ref{lem:Atheta-lower-bound}]
    \cite[Lemma 4.82]{OrponenShmerkin26} extends in a straightforward way to higher dimensions, with affine $(d-k)$-planes in place of lines. This gives
    \begin{equation}\label{form65}
        \mu\bigl(\pi^{-1}(I_x)\cap B\cap P_{0}\bigr)
        \lesssim\Delta^{-\sigma}\delta^{t-2\e},
        \qquad x\in A,\quad B\in\mathcal{B}(\mathbf{T}_{0}).
    \end{equation}
    If $x\in A\setminus A_{\theta}$, then fewer than $\Delta^{-\sigma+10\zeta}$ balls contribute to the double sum over $A\setminus A_\theta$ and $\mathcal B(\mathbf T_0)$. Hence, by \eqref{form65}, \eqref{form71}, and $\delta=\Delta^2$,
    \begin{align*}
        \sum_{x\in A\setminus A_{\theta}}\sum_{B\in\mathcal{B}(\mathbf{T}_{0})}
        \mu\bigl(\pi^{-1}(I_x)\cap B\cap P_{0}\cap P_{\theta}\bigr)
        &\lesssim |A|\Delta^{-2\sigma+10\zeta}\delta^{t-2\e}\\
        &\lesssim\Delta^{t-\sigma+6\zeta-4\e}
        =o\bigl(\Delta^{t-\sigma+5\zeta}\bigr).
    \end{align*}
    Since the cubes $I_x$, $x\in A$, cover $\pi(P_{0}\cap 2\mathbf{T}_{0})$, \eqref{form67} gives
    \begin{displaymath}
        \sum_{x\in A}\sum_{B\in\mathcal{B}(\mathbf{T}_{0})}
        \mu\bigl(\pi^{-1}(I_x)\cap B\cap P_{0}\cap P_{\theta}\bigr)
        \geq\Delta^{t-\sigma+5\zeta}.
    \end{displaymath}
    For $\delta\ll_{d,k,t,\sigma,\e,\zeta}1$, the part with $x\in A_{\theta}$ is therefore at least $\tfrac12\Delta^{t-\sigma+5\zeta}$. On the other hand, \eqref{form65} and Lemma \ref{lemma8} yield
    \begin{displaymath}
        \sum_{x\in A_{\theta}}\sum_{B\in\mathcal{B}(\mathbf{T}_{0})}
        \mu\bigl(\pi^{-1}(I_x)\cap B\cap P_{0}\cap P_{\theta}\bigr)
        \lesssim |A_{\theta}|\Delta^{-2\sigma}\delta^{t-2\e}.
    \end{displaymath}
    Comparing the two bounds and using $4\e\leq\zeta$ proves \eqref{form70}.
\end{proof}

\subsubsection{\texorpdfstring{The set $Y$}{The set Y}}

For $B\in\mathcal{B}(\mathbf{T}_{0})$, let $y(B)$ be the lower-left endpoint of the dyadic $\Delta$-cube containing the $W$-coordinate of the center of $B$, and set
\begin{displaymath}
    Y:=\{y(B):B\in\mathcal{B}(\mathbf{T}_{0})\}\subset\Delta\Z^{d-k}\cap[-2,2]^{d-k}.
\end{displaymath}
The centers lie in one $O(\Delta)$-slab parallel to $W$, so the bounded overlap of $\mathcal{B}_{\Delta}$ implies that $B\mapsto y(B)$ has fibers of cardinality $O(1)$. By \eqref{form67}, the Frostman bound $\mu(B)\leq\Delta^{t-4\e}$, and Lemma \ref{lemma8},
\begin{equation}\label{form73}
    |Y|\sim|\mathcal{B}(\mathbf{T}_{0})|,
    \qquad |Y|\gtrsim\Delta^{-\sigma+6\zeta},
    \qquad |Y|\lesssim\Delta^{-\sigma}.
\end{equation}
For $y\in W$ and $r\in[\delta^{-\sqrt{\e}}\Delta,1]$, Lemma \ref{lemma8} and \eqref{form73} yield
\begin{displaymath}
    |Y\cap B(y,r)|\lesssim(r/\Delta)^{\sigma}
    \lesssim\Delta^{-7\zeta}r^{\sigma}|Y|.
\end{displaymath}
For $y\in W$ and $r\in[\Delta,\delta^{-\sqrt{\e}}\Delta]$, we deduce from the lattice bound and \eqref{eq:epsilon-def} that
\begin{displaymath}
    |Y\cap B(y,r)|\lesssim(r/\Delta)^{d-k}
    \leq\delta^{-(d-k)\sqrt{\e}}\leq\Delta^{-\zeta}
    \lesssim\Delta^{-7\zeta}r^{\sigma}|Y|.
\end{displaymath}
Thus, $Y$ is $(\Delta,\sigma,\Delta^{-7\zeta})_{d-k}$-Frostman.

For $C_d\gg 1$, the set $A$ is contained in $\delta\Z^k\cap(x_A+[-C_d\Delta,C_d\Delta]^k)$ for some $x_A\in\delta\Z^k$. Since $A_{\theta}\subset A$, Equations \eqref{form71}, \eqref{form70}, and \eqref{form73} yield, for $\delta\ll_{t,\sigma,\e,\zeta}1$,
\begin{equation}\label{form74}
    \Delta^{\sigma-t+6\zeta}\leq|A|\leq\Delta^{\sigma-t-4\zeta},
    \qquad |Y|\leq\Delta^{-\sigma-\zeta}.
\end{equation}

\subsubsection{Finding a contradiction to Proposition \ref{prop:ABC}}

For $\theta\in\overline{\Theta}$ and $x\in A_{\theta}$, set
\begin{displaymath}
    Y_{x,\theta}:=\{y(B):B\in\mathcal{B}(\mathbf{T}_{0}),\ 
    \pi^{-1}(I_x)\cap B\cap P_{0}\cap P_{\theta}\neq\varnothing\},
\end{displaymath}
and let $G_{\theta}:=\{(x,y):x\in A_{\theta},\ y\in Y_{x,\theta}\}$. The bounded cardinality of the fibers of $B\mapsto y(B)$ and \eqref{form75} imply that $|Y_{x,\theta}|\gtrsim\Delta^{-\sigma+10\zeta}$. Consequently,
\begin{equation}\label{form88}
    |G_{\theta}|\gtrsim\Delta^{-t+16\zeta}
    \stackrel{\eqref{form74}}{\gtrsim}\Delta^{21\zeta}|A||Y|,
    \qquad \theta\in\overline{\Theta}.
\end{equation}

If $(x,y)\in G_{\theta}$, then there are $B\in\mathcal{B}(\mathbf{T}_{0})$ with $y=y(B)$ and $(x',y')\in\pi^{-1}(I_x)\cap B\cap P_{0}\cap P_{\theta}$. Hence $|x-x'|\lesssim\delta$ and $|y-y'|\lesssim\Delta$. Since $\|\theta\|\lesssim_{d,k}\Delta$ for $\theta\in\overline\Theta$, $|\pi_{\theta}(x,y)-\pi_{\theta}(x',y')|\lesssim\delta+\Delta^2\lesssim\delta$. Therefore,
\begin{equation}\label{form76}
    |\pi_{\theta}(G_{\theta})|_{\delta}
    \lesssim|\pi_{\theta}(P_{\theta}\cap2\mathbf{T}_{0})|_{\delta}
    \stackrel{\eqref{eq:sparse-projection-upper}}{\leq}\Delta^{\sigma-t-4\zeta}
    \stackrel{\eqref{form74}}{\leq}\Delta^{-10\zeta}|A|,
    \qquad \theta\in\overline{\Theta}.
\end{equation}

Let $D(x,y):=(\Delta^{-1}(x-x_A),y)$ and $A':=\Delta^{-1}(A-x_A)$. Then
\begin{displaymath}
    A'\times Y=D(A\times Y)
    \subset(\Delta\Z^k\cap[-C_d,C_d]^k)\times(\Delta\Z^{d-k}\cap[-2,2]^{d-k}),
\end{displaymath}
and $|A'|=|A|\leq\Delta^{\sigma_0-t}=\Delta^{-\alpha}$ by \eqref{form74} and \eqref{eq:zeta-def-consequences}. Since $7\zeta\leq\chi$, the $(\Delta,\sigma,\Delta^{-7\zeta})_{d-k}$-Frostman bound shows that $Y$ is $(\Delta,\sigma,\Delta^{-\chi})$-Frostman.

\begin{figure}[htbp]
\centering
\begin{tikzpicture}[scale=0.82,transform shape,every node/.style={font=\small}]
    \begin{scope}[rotate=9]
        \fill[black!2] (0,0) rectangle (7.2,1.2);
        \draw[thick] (0,0) rectangle (7.2,1.2);

        \foreach \x in {1.2,2.4,3.6,4.8,6.0} {
            \draw[black!25] (\x,0) -- (\x,1.2);
        }

        \foreach \x/\lab in {0.6/B_1,1.8/B_2,3.0/B_3,4.2/B_4,5.4/B_5,6.6/B_6} {
            \node[black!55] at (\x,1.43) {$\lab$};
        }

        \foreach \x in {1.8,5.4} {
            \node[black!30] at (\x,0.6) {$\emptyset$};
        }

        \foreach \x in {0,2.4,3.6,6.0} {
            \foreach \y in {0.26,0.53,0.85} {
                \fill[purple!16]
                    (\x+0.12,\y-0.07) rectangle (\x+1.08,\y+0.07);
                \draw[purple!40,line width=0.45pt]
                    (\x+0.12,\y-0.07) rectangle (\x+1.08,\y+0.07);
            }
        }

        \foreach \x/\y in {
                0.22/0.26,0.74/0.53,0.48/0.85,
                3.23/0.26,3.03/0.85,
                4.55/0.53,3.86/0.85,
                6.68/0.26,6.43/0.53,6.92/0.85} {
            \fill[purple!75] (\x,\y) circle (1.8pt);
        }

        \draw[->,thick] (0,-0.45) -- (2.1,-0.45)
            node[midway,below] {$W$};
        \draw[->,thick] (7.75,0.1) -- (7.75,1.1)
            node[midway,right] {$V$};
        \draw[<->] (0,-0.2) -- (7.2,-0.2)
            node[midway,fill=white,inner sep=1pt] {$1$};
        \draw[<->] (7.35,0) -- (7.35,1.2)
            node[midway,left,fill=white,inner sep=1pt] {$\Delta$};

        \foreach \x/\y in {
                0.08/0.26,0.08/0.53,0.08/0.85,
                2.48/0.26,2.48/0.85,
                3.68/0.53,3.68/0.85,
                6.08/0.26,6.08/0.53,6.08/0.85} {
            \draw[purple!85,line width=2pt] (\x,\y-0.06) -- (\x,\y+0.06);
        }
    \end{scope}

    \draw[->,thick] (7.75,1.28) -- (8.50,1.28);
    \node[align=center,font=\scriptsize] (rescale) at (7.45,3.70)
        {$V$-coordinate rescaled\\by $\Delta^{-1}$};
    \draw[->,black!45] (rescale.south) .. controls (7.55,2.75) and (7.85,2.0) .. (8.08,1.4);

    \begin{scope}[shift={(8.55,0.10)}]
        \draw[thick] (0,0) rectangle (2.55,2.55);
        \foreach \x in {0.425,0.85,1.275,1.7,2.125} {
            \draw[black!25] (\x,0) -- (\x,2.55);
        }
        \foreach \y in {0.425,0.85,1.275,1.7,2.125} {
            \draw[black!25] (0,\y) -- (2.55,\y);
        }
        \foreach \x/\y in {
                0.213/0.638,0.213/1.488,0.213/1.913,
                1.063/0.638,1.063/1.913,
                1.488/1.488,1.488/1.913,
                2.338/0.638,2.338/1.488,2.338/1.913} {
            \fill[purple!75] (\x,\y) circle (2.3pt);
        }
        \foreach \y in {0.638,1.488,1.913} {
            \draw[purple!70,line width=1.4pt] (2.72,\y) -- (2.93,\y);
        }
        \draw[black!45] (2.825,0.51) -- (2.825,2.04);
        \node at (1.275,-0.36) {$Y$};
        \node at (3.08,1.275) {$A'$};
        \node at (1.275,2.93) {$D(G_\theta)\subset A'\times Y$};
    \end{scope}
\end{tikzpicture}
\caption{The incidences inside $\mathbf T_0$ become a large graph $D(G_\theta)\subset A'\times Y$ after rescaling the $V$-coordinate. Its cardinality is bounded from below in \eqref{form88}, while \eqref{form76} bounds its projection.}
\end{figure}
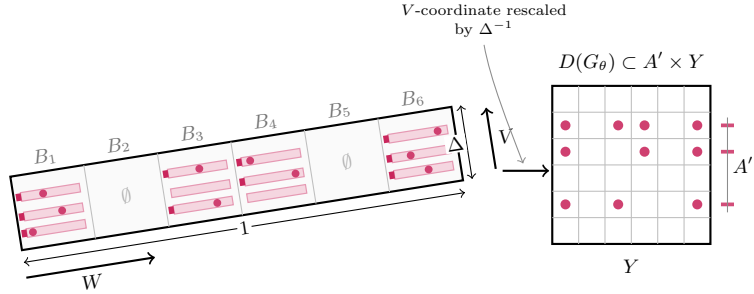

Let $\lambda$ be the pushforward of $\nu|_{\overline\Theta}/\nu(\overline\Theta)$ under $\theta\mapsto\Delta^{-1}\theta$. By \eqref{form11} and Lemma \ref{lem:spread-out-subset}, $\lambda$ is $(\Delta,s/2,\delta^{-10\e}\Delta^{-5\zeta})_{d,k}$-spread-out. Apply Lemma \ref{lem:uniform-spread-out-discretization} with $E=\spt\lambda$, and let $\widetilde\Theta$ be the set of lower-left endpoints of the resulting family. Then it follows from \eqref{eq:spread-out-loss} that, for $\delta\ll_{\chi,\e,\zeta}1$, the set
$\widetilde\Theta$ and the measure $m_{\widetilde\Theta}$ are $(\Delta,s/8,\Delta^{-\chi})_{d,k}$-spread-out. 

For each $\vartheta\in\widetilde{\Theta}$, choose $\theta_{\vartheta}\in\overline{\Theta}$ such that $\Delta^{-1}\theta_{\vartheta}$ belongs to the half-open $\Delta$-cube with lower-left endpoint $\vartheta$, and set $G'_{\vartheta}:=D(G_{\theta_{\vartheta}})$. Since $25\zeta\leq\chi$, \eqref{form88} yields, for $\delta\ll_{d,k,\chi,\zeta}1$,
\begin{equation}\label{form86}
    |G'_{\vartheta}|\gtrsim\Delta^{21\zeta}|A'||Y|
    \geq\Delta^{25\zeta}|A'||Y|
    \geq\Delta^{\chi}|A'||Y|.
\end{equation}
The identity $\pi_{\Delta^{-1}\theta}(D(G_{\theta}))=\Delta^{-1}(\pi_{\theta}(G_{\theta})-x_A)$ and the bound $\|\vartheta-\Delta^{-1}\theta_{\vartheta}\|\lesssim\Delta$ imply, since $Y\subset[-2,2]^{d-k}$, that
\begin{equation}\label{form89}
    |\pi_{\vartheta}(G'_{\vartheta})|_{\Delta}
    \lesssim|\pi_{\theta_{\vartheta}}(G_{\theta_{\vartheta}})|_{\delta}
    \stackrel{\eqref{form76}}{\lesssim}\Delta^{-10\zeta}|A'|
    <\Delta^{-\chi}|A'|.
\end{equation}

Fix $q\in\N$ such that $2^{-q}\max\{C_d,2\}\leq1$, and set
\begin{displaymath}
    \widehat\Delta:=2^{-q}\Delta,
    \qquad \widehat A:=2^{-q}A',
    \qquad \widehat Y:=2^{-q}Y,
    \qquad
    \widehat G_\vartheta
    :=\{(2^{-q}x,2^{-q}y):(x,y)\in G'_\vartheta\}.
\end{displaymath}
Then
\begin{displaymath}
    \widehat A\subset\widehat\Delta\Z^k\cap[-1,1]^k,
    \qquad
    \widehat Y\subset\widehat\Delta\Z^{d-k}\cap[-1,1]^{d-k},
    \qquad
    \widetilde\Theta\subset\widehat\Delta\Z^{k\times(d-k)}
    \cap[-1,1]^{k\times(d-k)},
\end{displaymath}
and $|\widehat A|=|A'|\leq\Delta^{-\alpha}\leq\widehat\Delta^{-\alpha}$.
Passing from $(\Delta,Y)$ to $(\widehat\Delta,\widehat Y)$, and lowering the base scale of $m_{\widetilde\Theta}$ from $\Delta$ to $\widehat\Delta$, change the Frostman and spread-out constants only by factors depending on $q$. We have
\begin{displaymath}
    Y\text{ is }(\Delta,\sigma,\Delta^{-7\zeta})\text{-Frostman},
    \qquad
    m_{\widetilde\Theta}\text{ is }
    (\Delta,s/8,\Delta^{-24\e-5\zeta})_{d,k}\text{-spread-out}.
\end{displaymath}
Since $7\zeta<\chi$ and $24\e+5\zeta\leq\chi/10$, for $\delta\ll_{d,k,s,\chi,\e,\zeta}1$ the set $\widehat Y$ is $(\widehat\Delta,\sigma,\widehat\Delta^{-\chi})$-Frostman and, by Lemma \ref{lem:spread-out-discretization}\textup{\ref{it:spread-out-discretization-cubes}}, $\widetilde\Theta$ is $(\widehat\Delta,s/8,\widehat\Delta^{-\chi})_{d,k}$-spread-out. Finally, from \eqref{form86} and \eqref{form89}, we deduce that
\begin{displaymath}
    |\widehat G_\vartheta|
    \geq\widehat\Delta^\chi|\widehat A||\widehat Y|,
    \qquad
    |\pi_\vartheta(\widehat G_\vartheta)|_{\widehat\Delta}
    =|\pi_\vartheta(G'_\vartheta)|_{\Delta}
    <\widehat\Delta^{-\chi}|\widehat A|.
\end{displaymath}
For $\delta\ll_{q,\Delta_{\mathrm{ABC}}}1$, we have $\widehat\Delta\leq\Delta_{\mathrm{ABC}}$, so Proposition \ref{prop:ABC} applied to $\widehat A,\widehat Y,\widetilde\Theta$ contradicts the two displayed bounds for $\widehat G_\vartheta$. This completes the proof of Proposition \ref{prop:AhlforsProjections-induction}.


\section{Multiscale and sticky Furstenberg estimates}\label{sec:multiscaleFurstenberg}

\subsection{Plane and spread-out configurations}

A key tool in \cite{OrponenShmerkin26} was the multiscale incidence estimate \cite[Proposition 5.2]{OS23}. Our goal is to derive an analogous estimate for spread-out configurations. We start by defining the relevant notation.

For the rest of this section, we fix the dimensions $1\le k<d$, and use the notation $A \lessapprox_{\delta} B$ to mean $A \leq B(\log(1/\delta))^{C}$, where the implicit exponent $C$ may depend on the parameters in Theorem \ref{thm:AhlforsFurstenberg}, namely $d,k,s,t,\eta$. We define $\gtrapprox_{\delta}$ and $\approx_{\delta}$ similarly.

It will sometimes be convenient to identify dyadic cubes with their lower-left endpoints, using the notation $z(Q)$ from Section \ref{s:discretized-formulations}. Given a point $x \in \delta\Z^D$ and a dyadic scale $\Delta \ge\delta$, we write $x^{(\Delta)} := z(\mathcal{D}_{\Delta}(x))$. We will also silently identify a family of dyadic cubes with its union whenever there is no risk of confusion. 

Recall the plane configurations $(P,\mathbf\Theta)$ and their Furstenberg sets $\mathcal F(P,\mathbf\Theta)$ from Definition \ref{def:planeConfigurations}. We emphasize that this definition does not require that $z\mathbf{\Theta}(p) = \{\theta \in \mathbf{G}_{d,k} : V_{\theta}+p \in \mathcal{F}\}$. That is, we allow for ``stray'' affine planes that pass through $p$ but belong to $\mathcal{F}$ by virtue of having the form $V_{\theta}+q$ for some $\theta\in z\mathbf{\Theta}(q)$, with $q\neq p$.

We will be mostly interested in spread-out configurations, as defined in Definition \ref{def:intro-spreadout}. Spread-out configurations are similar to \emph{nice configurations} introduced in \cite[Definition 5.1]{OS23}, and employed in the statement of the Furstenberg set theorem in \cite[Theorem 5.2]{OrponenShmerkin26}. The difference is that we have replaced Frostman sets in $S^1$ by spread-out sets in $\mathbf{G}_{d,k}$.

\subsection{Multiscale Furstenberg estimate}

We can now state the main result of this section, which is a multiscale estimate for spread-out configurations. It is the analog of \cite[Proposition 5.2]{OS23} in the setting of spread-out configurations, and the proof follows a similar outline. Because our current setting differs from that of \cite{OS23} in some important ways, we give a full proof of Proposition \ref{prop:induction-on-scales} below, even if it involves some arguments that are very similar to those in \cite[Proposition 5.2]{OS23}.

\begin{proposition}\label{prop:induction-on-scales}
    Let $C>0$, $\beta\in(0,1]$, $\delta,\Delta \in 2^{-\N}$, and $\overline{\Delta}\in2^{-\N_0}$ with $\delta \leq \Delta\leq \overline{\Delta}$. Let $(P_0,\mathbf{\Theta}_0)$ be a $(\delta,\beta,C,M_0)$-spread-out configuration. Then the following hold:
    \begin{enumerate}[label=\textup{(\roman*)}, ref=(\roman*)]
        \item \label{it-induction-i} There are $M\in\N$, a set $P\subset P_0$, and families $\mathbf\Theta(p)\subset\mathbf\Theta_0(p)$, $p\in P$, such that $(P,\mathbf\Theta)$ is a $(\delta,\beta,C_1,M)$-spread-out configuration, where $C_1\approx_\delta C$ and $M\approx_\delta M_0$. Moreover, $|P|_{\Delta}\approx_{\delta} |P_0|_{\Delta}$ and $|P\cap Q| \approx_{\delta}|P_0\cap Q|$ for each $Q\in\mathcal{D}_{\Delta}(P)$.
        \item \label{it-induction-ii} There exist $M_{\Delta} \in \N$, a set $P_{\Delta} \subset \Delta \Z^{d} \cap [0,1]^{d}$, and families $\mathbf{\Theta}_{\Delta}(q) \subset \mathcal{D}_{\Delta}([-1,1]^{k \times (d-k)})$, $q\in P_{\Delta}$, such that $(P_{\Delta},\mathbf{\Theta}_{\Delta})$ is a $(\Delta,\beta,C',M_{\Delta})$-spread-out configuration, with $C' \approx_{\delta} C$. Moreover,
              \begin{enumerate}[label=\textup{(C\alph*)}]
                  \item $P_{\Delta} = \{z(Q) : Q \in \mathcal{D}_{\Delta}(P)\}$,
                  \item For each $q=z(Q) \in P_{\Delta}$, we have $\bigcup_{p \in P\cap Q} \mathbf{\Theta}(p) \subset \mathbf{\Theta}_{\Delta}(q)$.
              \end{enumerate}
              Further, there is $L_{\overline{\Delta}} \in \N$ such that, for each $q\in P_{\Delta}$ and each $\mathbf{I}\in \mathcal{D}_{\overline{\Delta}}(\mathbf{\Theta}_{\Delta}(q))$, we have
             \begin{displaymath} |\{\mathbf{J} \in \mathbf{\Theta}_{\Delta}(q) : \mathbf{J} \subset \mathbf{I}\}| = L_{\overline{\Delta}}. \end{displaymath}
        \item \label{it-induction-iii} For each $Q\in \mathcal{D}_{\Delta}(P)$, there exist $M_Q \in \N$, a set $P^{Q} \subset (\delta/\Delta)\Z^{d} \cap [0,1]^{d}$, and families
              $\mathbf{\Theta}^{Q}(p) \subset \mathcal{D}_{\delta/\Delta}([-1,1]^{k \times (d-k)})$, $p\in P^{Q}$,
              such that $(P^{Q},\mathbf{\Theta}^{Q})$ is a $(\delta/\Delta,\beta,C_Q,M_Q)$-spread-out configuration, with $C_Q \approx_{\delta} C$. Moreover,
               \begin{enumerate}[label=\textup{(Z\alph*)}]
                   \item $P^{Q} = S_Q(P\cap Q)$,
                   \item For each $p=S_Q(p_0) \in P^{Q}$, we have $\mathbf{\Theta}^{Q}(p) \subset \mathcal{D}_{\delta/\Delta}\bigl(\mathbf{\Theta}(p_0)\bigr)$.
                   \item\label{it:induction-zoomed-direction-multiplicity} There is $m_Q\in\N$ such that, for $p=S_Q(p_0)\in P^Q$ and $\mathbf J\in\mathbf\Theta^Q(p)$,
                   \begin{displaymath}
                       m_Q\le
                       |\{\mathbf I\in\mathbf\Theta(p_0):\mathbf I\subset\mathbf J\}|
                       <4m_Q,
                       \qquad
                       m_QM_Q\approx_\delta|\mathbf\Theta(p_0)|.
                   \end{displaymath}
               \end{enumerate}
        \item \label{it-induction-iv} Let $\mathcal{F}_0$, $\mathcal{F}_{\Delta},\mathcal{F}^{Q} \subset \mathbf{A}_{d,k} \cong \R^{(d-k)(k+1)}$ denote the Furstenberg sets, respectively, for the original configuration and the configurations in \ref{it-induction-ii} and \ref{it-induction-iii}. Then
        \begin{equation} \label{eq:lower-bound-Furstenberg}
            \frac{|\mathcal{F}_0|_{\delta}}{M_0} \gtrapprox_{\delta}
            \frac{|\mathcal{F}_{\Delta}|_{\Delta}}{M_{\Delta}} \cdot \max_{Q \in \mathcal{D}_{\Delta}(P)} \frac{|\mathcal{F}^{Q}|_{\delta/\Delta}}{M_Q}.
        \end{equation}
        Here $|\mathcal{F}_0|_{\delta}$ etc. refer to covering numbers relative to the Euclidean metric on $\R^{(d-k)(k+1)}$.
    \end{enumerate}

\end{proposition}

Note that no comparison between $|P|$ and $|P_0|$ is asserted. However, if the cardinalities $|P_0\cap Q|$, $Q\in\mathcal D_\Delta(P_0)$, are roughly comparable, then \ref{it-induction-i} implies that the same holds for  $|P|$ and $|P_0|$.

\subsection{Almost covering}

Before proving Proposition \ref{prop:induction-on-scales}, we need to introduce the notion of an almost cover and prove that spread-out families are robust under almost covering.

\begin{definition}[$\alpha$-almost cover]\label{def:almostCover}
    Let $\delta,\Delta \in 2^{-\N}$ with $\delta \leq \Delta$, and let $\alpha > 0$. Let $J$ be a finite index set, and let $\{\mathcal{P}_{j}\}_{j \in J}$ be a family of subsets of $\mathcal{D}_{\delta}(\R^D)$. A family $\mathcal{Q} \subset \mathcal{D}_{\Delta}(\R^D)$ is an \emph{$\alpha$-almost cover of $\{\mathcal{P}_{j}\}_{j \in J}$} if the following properties hold for some $\mathfrak{m} \in 2^{\N_0}$:
    \begin{enumerate}[label=\textup{$(\alpha_{\arabic*})$}]
        \item $|\{(j,p) \in J \times \mathcal{D}_{\delta}(\R^D) : p \in \mathcal{P}_{j}, \, p \subset Q\}| \geq \mathfrak{m}$ for all $Q \in \mathcal{Q}$;
        \item $\mathfrak{m}|\mathcal{Q}| \geq \alpha \sum_{j \in J} |\mathcal{P}_{j}|$.
    \end{enumerate}
    We will use this terminology both for $D = d$ and $D = k(d - k)$.
\end{definition}

The definition of spread-outness has been tailored so that it is robust under almost covering:
\begin{lemma}\label{lem:spread-out-almost-cover-families}
    Let $\delta \leq \Delta$, let $\mathbf{\Theta}_{j} \subset \mathcal{D}_{\delta}(\mathbf{G}_{d,k})$, $j \in J$, be $(\delta,\beta,C)_{d,k}$-spread-out families, and let $\mathbf{\Theta}_{\Delta} \subset \mathcal{D}_{\Delta}(\mathbf{G}_{d,k})$ be an $\alpha$-almost cover of $\{\mathbf{\Theta}_{j}\}_{j \in J}$. Then $\mathbf{\Theta}_{\Delta}$ is $(\Delta,\beta,\alpha^{-1}C)_{d,k}$-spread-out.
\end{lemma}

\begin{proof}
    Let $\mathbf\Theta_\Delta'\subset\mathbf\Theta_\Delta$ have at least $\mathfrak s$ $\beta$-singular scales at every point, and set $\mathbf\Theta_j':=\{\mathbf I\in\mathbf\Theta_j:\mathbf I\subset\bigcup\mathbf\Theta_\Delta'\}$. Since $\bigcup\mathbf\Theta_j'\subset\bigcup\mathbf\Theta_\Delta'$, every point of $\bigcup\mathbf\Theta_j'$ has at least $\mathfrak s$ $\beta$-singular scales relative to $\mathbf\Theta_j'$ between $\Delta$ and $1$. If $\mathfrak m$ is the multiplicity in Definition \ref{def:almostCover}, spread-outness and the almost-cover inequalities imply
    \begin{displaymath}
        \mathfrak m|\mathbf\Theta_\Delta'|
        \le\sum_{j\in J}|\mathbf\Theta_j'|
        \le C2^{-\beta\mathfrak s}\sum_{j\in J}|\mathbf\Theta_j|
        \le\alpha^{-1}C2^{-\beta\mathfrak s}\mathfrak m|\mathbf\Theta_\Delta|.
    \end{displaymath}
    Dividing by $\mathfrak m$ proves the claim.
\end{proof}

\begin{corollary}\label{cor:spread-out-coarsening}
Let $\mathbf\Theta\subset\mathcal D_\delta(\mathbf G_{d,k})$ be $(\delta,\beta,C)_{d,k}$-spread-out, and let $r\in2^{-\N}\cap[\delta,1]$. If the nonzero cardinalities $|\{\mathbf I\in\mathbf\Theta:\mathbf I\subset\mathbf J\}|$, $\mathbf J\in\mathcal D_r(\mathbf\Theta)$, differ by a factor of at most $K$, then $\mathcal D_r(\mathbf\Theta)$ is $(r,\beta,2KC)_{d,k}$-spread-out.
\end{corollary}
\begin{proof}
Round the minimum descendant count down to a dyadic integer. Then $\mathcal D_r(\mathbf\Theta)$ is a $(2K)^{-1}$-almost cover of $\{\mathbf\Theta\}$, so Lemma \ref{lem:spread-out-almost-cover-families} applies.
\end{proof}

Almost covers can be produced by pigeonholing. We record their existence in the following form:
\begin{proposition}\label{prop:almost-cover}
    Let $\delta,\Delta \in 2^{-\N}$ with $\delta \leq \Delta$. Let $P$ be a finite index set, and assume that for every $p \in P$, we are given a finite family $\mathbf{\Theta}(p) \subset \mathcal{D}_{\delta}([-1,1]^{k \times (d-k)})$. Then there exists a collection $\mathbf{\Theta}_{\Delta} \subset \mathcal{D}_{\Delta}([-1,1]^{k \times (d-k)})$ which $\alpha$-almost covers the family $\{\mathbf{\Theta}(p) : p \in P\}$ for some $\alpha \gtrapprox_{\delta} 1$. Moreover, the number $\mathfrak m$ in Definition \ref{def:almostCover} may be chosen so that
    \begin{displaymath}
        \mathfrak m
        \leq
        |\{(p,\mathbf I):p\in P,\ \mathbf I\in\mathbf\Theta(p),\
        \mathbf I\subset\mathbf J\}|
        <2\mathfrak m,
        \qquad \mathbf J\in\mathbf\Theta_\Delta.
    \end{displaymath}
\end{proposition}
\begin{proof}
    Set $S:=\sum_{p\in P}|\mathbf\Theta(p)|$. The claim is trivial if $S=0$, so assume that $S>0$. For
    \begin{displaymath}
        \overline{\mathbf\Theta}_\Delta
        :=\mathcal D_\Delta\left(\bigcup_{p\in P}\mathbf\Theta(p)\right)
    \end{displaymath}
    and $\mathbf J\in\overline{\mathbf\Theta}_\Delta$, write
    \begin{displaymath}
        n(\mathbf J)
        :=|\{(p,\mathbf I):p\in P,\ \mathbf I\in\mathbf\Theta(p),\
        \mathbf I\subset\mathbf J\}|.
    \end{displaymath}
    We have
    \begin{displaymath}
        \sum_{\mathbf J\in\overline{\mathbf\Theta}_\Delta}n(\mathbf J)=S,
        \qquad
        |\overline{\mathbf\Theta}_\Delta|\lesssim_{d,k}\Delta^{-k(d-k)}.
    \end{displaymath}
    The cubes satisfying $n(\mathbf J)<c\Delta^{k(d-k)}S$ for $c\ll_{d,k} 1$ contribute at most $S/2$ to the sum. For the remaining cubes, we can dyadically pigeonhole $\mathfrak m\in 2^{\N_0}$ such that
    \begin{displaymath}
       \mathfrak m|\mathbf\Theta_\Delta| \ge \frac{1}{2} \sum_{\mathbf J\in\mathbf\Theta_\Delta}n(\mathbf J)
        \gtrapprox_\delta S \quad\text{where } \mathbf\Theta_\Delta
        :=\{\mathbf J\in\overline{\mathbf\Theta}_\Delta:
        \mathfrak m\leq n(\mathbf J)<2\mathfrak m\}.
    \end{displaymath}
    This is property $(\alpha_2)$ for some $\alpha\gtrapprox_\delta1$, while $n(\mathbf J)\geq\mathfrak m$ is property $(\alpha_1)$. The definition of $\mathbf\Theta_\Delta$ gives the asserted upper bound.
\end{proof}

\subsection{Zoomed-in configurations}

The next lemma will help us deal with the zoomed-in configurations in Proposition \ref{prop:induction-on-scales}\ref{it-induction-iii}.
\begin{lemma}\label{lem:spread-out-single-set-zoomed-in}
    Let $\mathbf{\Theta}_0 \subset \mathcal{D}_{\delta}([-1,1]^{k\times(d-k)})$ be a $(\delta,\beta,C)$-spread-out family, and let $\Delta \in 2^{-\N} \cap [\delta,1]$. Then there exist a subfamily $\overline{\mathbf{\Theta}}\subset \mathbf{\Theta}_0$, a family $\mathbf{\Theta}_{\delta/\Delta} \subset \mathcal{D}_{\delta/\Delta}([-1,1]^{k\times(d-k)})$, and $m \in 2^{\N_0}$ such that
    \begin{enumerate}[label=\textup{(\roman*)}, ref=(\roman*)]
        \item \label{it:single-zoom-i} $\mathbf{\Theta}_{\delta/\Delta} =\mathcal{D}_{\delta/\Delta}(\overline{\mathbf{\Theta}})$;
        \item \label{it:single-zoom-ii} $\mathbf{\Theta}_{\delta/\Delta}$ is $(\delta/\Delta,\beta,C')$-spread-out, where $C' \approx_{\delta} C$;
        \item \label{it:single-zoom-iii} $|\{\mathbf{I}\in \overline{\mathbf{\Theta}}:\mathbf{I}\subset \mathbf{J}\}| \sim m$ for all $\mathbf{J} \in \mathbf{\Theta}_{\delta/\Delta}$;
        \item \label{it:single-zoom-iv} $m|\mathbf{\Theta}_{\delta/\Delta}| \approx_{\delta} |\overline{\mathbf{\Theta}}| \gtrapprox_{\delta} |\mathbf{\Theta}_0|$.
    \end{enumerate}
\end{lemma}

\begin{proof}
    Apply Proposition \ref{prop:almost-cover}, with a singleton $P$ and $\delta/\Delta$ in place of $\Delta$, to the family $\mathbf\Theta_0$. Let $\mathbf\Theta_{\delta/\Delta}$ be the resulting almost cover and let $m$ be its multiplicity. Define
    \begin{displaymath}
        \overline{\mathbf\Theta}
        :=\{\mathbf I\in\mathbf\Theta_0:\mathbf I\subset\mathbf J
        \text{ for some }\mathbf J\in\mathbf\Theta_{\delta/\Delta}\}.
    \end{displaymath}
    Proposition \ref{prop:almost-cover} yields  $|\{\mathbf I\in\overline{\mathbf\Theta}:\mathbf I\subset\mathbf J\}|\in [m,2m]$ for all $\mathbf J\in\mathbf\Theta_{\delta/\Delta}$, and
    $m|\mathbf\Theta_{\delta/\Delta}|
        \approx_\delta|\overline{\mathbf\Theta}|
        \gtrapprox_\delta|\mathbf\Theta_0|$.
    These are \ref{it:single-zoom-i}, \ref{it:single-zoom-iii}, and \ref{it:single-zoom-iv}. Since $\mathbf\Theta_{\delta/\Delta}$ is an $\alpha$-almost cover of $\mathbf\Theta_0$ with $\alpha\gtrapprox_\delta1$, Lemma \ref{lem:spread-out-almost-cover-families} shows that it is $(\delta/\Delta,\beta,C')$-spread-out with $C'\approx_\delta C$. This is \ref{it:single-zoom-ii}.
\end{proof}

\subsection{Proof of Proposition \ref{prop:induction-on-scales}}

In this section we prove Proposition \ref{prop:induction-on-scales}. Throughout the proof, we identify each Furstenberg set with the union of the $\mathbf{A}_{d,k}$-cubes of its defining scale that it intersects. We also use the same notation for the family of these cubes; for Furstenberg sets, $\|\cdot\|$ without a scale subscript is the number of affine cubes.

\subsubsection{Construction of the coarse-scale families}\label{s651}

For each $Q \in \mathcal{D}_{\Delta}(P_0)$, write $q:=z(Q)$ and apply Proposition \ref{prop:almost-cover} to the family $\{\mathbf{\Theta}_0(p)\}_{p \in P_0 \cap Q}$. We get a preliminary coarse family $\overline{\mathbf{\Theta}}_{\Delta}(q) \subset \mathcal{D}_{\Delta}([-1,1]^{k \times (d-k)})$ and a number $H_Q \geq 1$ such that
\begin{enumerate}[label=\textup{(A\arabic*)}, ref=(A\arabic*)]
    \item \label{it:A1} $|\{(p,\mathbf{I}) \in (P_0 \cap Q)\times \mathcal{D}_{\delta}(\R^{k \times (d-k)}) : \mathbf{I} \in \mathbf{\Theta}_0(p),\, \mathbf{I} \subset \mathbf{J}\}| \geq H_Q$ for $\mathbf{J} \in \overline{\mathbf{\Theta}}_{\Delta}(q)$;
    \item \label{it:A2} $H_Q |\overline{\mathbf{\Theta}}_{\Delta}(q)| \gtrapprox_{\delta} |P_0 \cap Q|M_0$.
\end{enumerate}
Moreover, every $\overline{\mathbf{\Theta}}_{\Delta}(q)$ is $(\Delta,\beta,\overline{C}_{\Delta})$-spread-out by Lemma \ref{lem:spread-out-almost-cover-families}, with $\overline{C}_{\Delta} \approx_{\delta} C$.

By dyadic pigeonholing, after restricting to a subcollection $\overline{\mathcal{Q}} \subset \mathcal{D}_{\Delta}(P_0)$ of comparable cardinality and then discarding all points outside $\bigcup \overline{\mathcal{Q}}$, we may assume that
\begin{displaymath}
    \overline{M}_{\Delta} \leq |\overline{\mathbf{\Theta}}_{\Delta}(q)| < 2\overline{M}_{\Delta}, \qquad q=z(Q),\; Q \in \mathcal{D}_{\Delta}(P_0).
\end{displaymath}

For $Q \in \mathcal{D}_{\Delta}(P_0)$ and $p \in P_0 \cap Q$, define the preliminary fine family
\begin{displaymath}
    \widetilde{\mathbf{\Theta}}(p)
    :=
    \{\mathbf{I} \in \mathbf{\Theta}_0(p) : \mathbf{I} \subset \mathbf{J} \text{ for some } \mathbf{J} \in \overline{\mathbf{\Theta}}_{\Delta}(z(Q))\}.
\end{displaymath}
Let $\widetilde{\mathcal{F}} := \bigcup_{p \in P_{0}} \{V_{\theta} + p : \theta \in z\widetilde{\mathbf{\Theta}}(p)\}$ be the Furstenberg set associated to the configuration $(P_0,\widetilde{\mathbf{\Theta}})$.

For $q \in P_{\Delta}^{\mathrm{pre}}:=\{z(Q) : Q \in \mathcal{D}_{\Delta}(P_0)\}$, write
\begin{displaymath}
    \overline{\mathcal{F}}_{\Delta}(q)
    :=
    \mathcal{D}_{\Delta}(\{V_{\theta}+q : \theta \in z\overline{\mathbf{\Theta}}_{\Delta}(q)\})
    \subset \mathcal{D}_{\Delta}(\mathbf{A}_{d,k}),
\end{displaymath}
and let $\overline{\mathcal{F}}_{\Delta}:=\bigcup_{q \in P_{\Delta}^{\mathrm{pre}}} \overline{\mathcal{F}}_{\Delta}(q)$. Fix $A\gg_{d,k}1$. For $\mathbf{F} \in \mathcal{D}_{\Delta}(\mathbf{A}_{d,k})$, let
\begin{displaymath}
    \mathcal{N}_{\Delta}(\mathbf{F})
    :=
    \{\mathbf{F}' \in \mathcal{D}_{\Delta}(\mathbf{A}_{d,k}) : \dist(\mathbf{F},\mathbf{F}') \leq A\Delta\}.
\end{displaymath}
Choose $A\gg_{d,k}1$ so that the following holds. If $Q \in \mathcal{D}_{\Delta}(P_0)$, $q=z(Q)$, $p \in P_0 \cap Q$, $\mathbf{I} \in \widetilde{\mathbf{\Theta}}(p)$, and $\mathbf{I} \subset \mathbf{J} \in \overline{\mathbf{\Theta}}_{\Delta}(q)$, then
\begin{equation}\label{form8}
    \mathcal{D}_{\Delta}(V_{z(\mathbf{I})}+p)
    \in
    \mathcal{N}_{\Delta}\bigl(\mathcal{D}_{\Delta}(V_{z(\mathbf{J})}+q)\bigr).
\end{equation}
Indeed, \eqref{eq:affine-incidence-coordinates} shows that the affine coordinates of $V_{z(\mathbf I)}+p$ and $V_{z(\mathbf J)}+q$ differ by $O(\Delta)$. Thus, writing
\begin{displaymath}
    \widetilde{\mathcal{F}}(Q)
    :=
    \bigcup_{p \in P_0 \cap Q}\{V_{\theta}+p : \theta \in z\widetilde{\mathbf{\Theta}}(p)\},
\end{displaymath}
we have
\begin{displaymath}
    \mathcal{D}_{\Delta}(\widetilde{\mathcal{F}}(Q))
    \subset
    \bigcup_{\mathbf{F} \in \overline{\mathcal{F}}_{\Delta}(q)} \mathcal{N}_{\Delta}(\mathbf{F}),
    \qquad q=z(Q).
\end{displaymath}
For $j \geq 0$, let
\begin{displaymath}
    \overline{\mathcal{F}}_{\Delta,j}
    :=
    \left\{\mathbf{F} \in \overline{\mathcal{F}}_{\Delta} : 2^j \leq \left|\widetilde{\mathcal{F}} \cap \bigcup \mathcal{N}_{\Delta}(\mathbf{F})\right| < 2^{j+1}\right\}.
\end{displaymath}
Each $\mathbf{F} \in \overline{\mathcal{F}}_{\Delta}$ has the form $\mathbf{F}=\mathcal{D}_{\Delta}(V_{z(\mathbf{J})}+q)$ for some $q=z(Q) \in P_{\Delta}^{\mathrm{pre}}$ and $\mathbf{J}\in \overline{\mathbf{\Theta}}_{\Delta}(q)$. By \ref{it:A1}, there exist $p \in P_0 \cap Q$ and $\mathbf{I} \in \mathbf{\Theta}_0(p)$ such that $\mathbf{I} \subset \mathbf{J}$. Then $\mathbf{I}\in \widetilde{\mathbf{\Theta}}(p)$, and \eqref{form8} shows that $\widetilde{\mathcal{F}} \cap \bigcup \mathcal{N}_{\Delta}(\mathbf{F}) \neq \varnothing$. Since every affine $\Delta$-cube contains at most $(\Delta/\delta)^{(k+1)(d-k)} \leq \delta^{-(k+1)(d-k)}$ affine $\delta$-cubes, and every neighborhood $\mathcal{N}_{\Delta}(\mathbf{F})$ has bounded cardinality, dyadic pigeonholing yields an index $j$ such that, writing
\begin{displaymath}
    N_{\Delta}:=2^j
    \qquad\text{and}\qquad
    \mathcal{F}_{\Delta}^{\mathrm{pre}}:=\overline{\mathcal{F}}_{\Delta,j},
\end{displaymath}
we have
\begin{displaymath}
    \sum_{q \in P_{\Delta}^{\mathrm{pre}}} |\mathcal{F}_{\Delta}^{\mathrm{pre}} \cap \overline{\mathcal{F}}_{\Delta}(q)|
    \gtrapprox_{\delta}
    \sum_{q \in P_{\Delta}^{\mathrm{pre}}} |\overline{\mathcal{F}}_{\Delta}(q)|
    \sim
    |P_{\Delta}^{\mathrm{pre}}|\overline{M}_{\Delta},
\end{displaymath}
and
\begin{equation}\label{form6}
    N_{\Delta} \leq \left|\widetilde{\mathcal{F}} \cap \bigcup \mathcal{N}_{\Delta}(\mathbf{F})\right| < 2N_{\Delta}, \qquad \mathbf{F} \in \mathcal{F}_{\Delta}^{\mathrm{pre}}.
\end{equation}

From another dyadic pigeonholing, we get a subset $\mathcal{Q} \subset \mathcal{D}_{\Delta}(P_0)$ with $|\mathcal{Q}| \approx_{\delta} |\mathcal{D}_{\Delta}(P_0)|$ such that
\begin{displaymath}
    |\mathcal{F}_{\Delta}^{\mathrm{pre}} \cap \overline{\mathcal{F}}_{\Delta}(q)| \approx_{\delta} \overline{M}_{\Delta}, \qquad q=z(Q),\; Q \in \mathcal{Q}.
\end{displaymath}
We now define the final coarse plane families by
\begin{displaymath}
    \mathbf{\Theta}_{\Delta}(q)
    :=
    \{\mathbf{J} \in \overline{\mathbf{\Theta}}_{\Delta}(q) : \mathcal{D}_{\Delta}(V_{z(\mathbf{J})}+q) \in \mathcal{F}_{\Delta}^{\mathrm{pre}}\},
    \qquad q=z(Q),\; Q \in \mathcal{Q}.
\end{displaymath}
Then
\begin{displaymath}
    |\mathbf{\Theta}_{\Delta}(q)| \approx_{\delta} \overline{M}_{\Delta}, \qquad q \in P_{\Delta}:=\{z(Q) : Q \in \mathcal{Q}\},
\end{displaymath}
and every $\mathbf{\Theta}_{\Delta}(q)$ is still $(\Delta,\beta,C_{\Delta})$-spread-out with $C_{\Delta} \approx_{\delta} C$. The cardinality bound follows from \eqref{eq:affine-incidence-counting}, with $P_\theta=\{q\}$.

We now impose the scale-$\overline{\Delta}$-uniformity from \textup{\ref{it-induction-ii}} on these coarse families. Dyadic pigeonholing at scale $\overline{\Delta}$ yields dyadic numbers $M_{\Delta},L_{\overline{\Delta}} \geq 1$. After refining $P_{\Delta}$ and $\mathcal{Q}$ with losses of at most $(\log(1/\delta))^{O(1)}$ and retaining exactly $L_{\overline{\Delta}}$ members of $\mathbf\Theta_{\Delta}(q)$ in each occupied $\overline{\Delta}$-cube, we may assume that
\begin{displaymath}
    M_{\Delta} \leq |\mathbf{\Theta}_{\Delta}(q)| < 2M_{\Delta},
    \qquad
    |\{\mathbf{J} \in \mathbf{\Theta}_{\Delta}(q) : \mathbf{J} \subset \mathbf{I}\}| = L_{\overline{\Delta}}
\end{displaymath}
for all $q \in P_{\Delta}$ and $\mathbf{I} \in \mathcal{D}_{\overline{\Delta}}(\mathbf{\Theta}_{\Delta}(q))$. Define
\begin{displaymath}
    \mathcal{F}_{\Delta}
    :=
    \bigcup_{Q \in \mathcal{Q}} \bigcup_{\mathbf{J} \in \mathbf{\Theta}_{\Delta}(z(Q))}
    \mathcal{D}_{\Delta}(V_{z(\mathbf{J})} + z(Q))
    \subset \mathcal{F}_{\Delta}^{\mathrm{pre}}.
\end{displaymath}
Thus $\mathcal{F}_{\Delta}$ is the family of dyadic $\Delta$-cubes meeting the Furstenberg set associated to $(P_{\Delta},\mathbf{\Theta}_{\Delta})$. Since $\mathcal{F}_{\Delta} \subset \mathcal{F}_{\Delta}^{\mathrm{pre}}$, it holds
\begin{displaymath}
    N_{\Delta} \leq \left|\widetilde{\mathcal{F}} \cap \bigcup \mathcal{N}_{\Delta}(\mathbf{F})\right| < 2N_{\Delta}, \qquad \mathbf{F} \in \mathcal{F}_{\Delta}.
\end{displaymath}
The verification of \ref{it-induction-ii} will be completed after the final  set $P$ has been fixed.

\subsubsection{Refinement of the fine families}

For $Q \in \mathcal{Q}$ and $p \in P_0 \cap Q$, let
\begin{equation}\label{form14}
    \mathbf{\Theta}(p)
    :=
    \{\mathbf{I} \in \mathbf{\Theta}_0(p) : \mathbf{I} \subset \mathbf{J} \text{ for some } \mathbf{J} \in \mathbf{\Theta}_{\Delta}(z(Q))\}.
\end{equation}
For fixed $Q \in \mathcal{Q}$, we have
\begin{align*}
    \sum_{p \in P_0 \cap Q} |\mathbf{\Theta}(p)|
     & = \sum_{\mathbf{J} \in \mathbf{\Theta}_{\Delta}(z(Q))} |\{(p,\mathbf{I}) \in (P_0 \cap Q)\times \mathcal{D}_{\delta}(\R^{k \times (d-k)}) : \mathbf{I} \in \mathbf{\Theta}_0(p),\, \mathbf{I} \subset \mathbf{J}\}| \\
     & \stackrel{\ref{it:A1}}{\geq} |\mathbf{\Theta}_{\Delta}(z(Q))|H_Q \gtrapprox_{\delta} |\overline{\mathbf{\Theta}}_{\Delta}(z(Q))|H_Q \stackrel{\ref{it:A2}}{\gtrapprox}_{\delta} |P_0 \cap Q|M_0.
\end{align*}
Since $|\mathbf{\Theta}(p)| \leq |\mathbf{\Theta}_0(p)| < 2M_0$, there exists a subset $P_Q \subset P_0 \cap Q$ such that
\begin{displaymath}
    |P_Q| \approx_{\delta} |P_0 \cap Q|
    \qquad\text{and}\qquad
    |\mathbf{\Theta}(p)| \gtrapprox_{\delta} M_0, \quad p \in P_Q.
\end{displaymath}
Every $\mathbf{\Theta}(p)$, $p \in P_Q$, remains $(\delta,\beta,\widetilde{C})$-spread-out with $\widetilde{C} \approx_{\delta} C$.

By dyadic pigeonholing, after refining each $P_Q$ without changing notation, we may find a dyadic number $M_Q^{\mathrm{fine}}$ such that
\begin{displaymath}
    M_Q^{\mathrm{fine}} \leq |\mathbf{\Theta}(p)| < 2M_Q^{\mathrm{fine}}, \qquad p \in P_Q.
\end{displaymath}
Since $M_Q^{\mathrm{fine}} \approx_{\delta} M_0$ for every $Q \in \mathcal{Q}$, another pigeonholing yields a subcollection, still denoted by $\mathcal{Q}$, and a dyadic number $M \approx_{\delta} M_{0}$ such that
\begin{equation}\label{form4}
    |\mathcal{Q}| \approx_{\delta} |\mathcal{D}_{\Delta}(P_0)|,
    \qquad
    M \leq |\mathbf{\Theta}(p)| < 2M, \qquad p \in P_Q,\; Q \in \mathcal{Q}.
\end{equation}
We keep the corresponding subset of $P_{\Delta}$ and the corresponding coarse families $\mathbf{\Theta}_{\Delta}(q)$, still denoting them by $P_{\Delta}$ and $\mathbf{\Theta}_{\Delta}$.

\subsubsection{Construction of the zoomed-in configurations}

We now construct the zoomed-in configurations, without further refining the families $\mathbf{\Theta}(p)$. Fix $Q \in \mathcal{Q}$, write $q:=z(Q)$ and $\overline{\delta}:=\delta/\Delta$, and let $P^{Q}:=S_Q(P_Q)$. For $p_0 \in P_Q$, apply Lemma \ref{lem:spread-out-single-set-zoomed-in} to the spread-out family $\mathbf{\Theta}(p_0)$ at scale $\Delta$. We get, for $p=S_Q(p_0) \in P^{Q}$, a spread-out family $\overline{\mathbf{\Theta}}^{Q}(p) \subset \mathcal{D}_{\overline{\delta}}(\mathbf{\Theta}(p_0))$ and a dyadic number $m(p_0) \geq 1$ such that
\begin{displaymath}
    |\{\mathbf{I} \in \mathbf{\Theta}(p_0) : \mathbf{I} \subset \mathbf{J}\}| \sim m(p_0), \qquad \mathbf{J} \in \overline{\mathbf{\Theta}}^{Q}(p),
\end{displaymath}
and
\begin{displaymath}
    m(p_0)|\overline{\mathbf{\Theta}}^{Q}(p)| \approx_{\delta} |\mathbf{\Theta}(p_0)| \sim M.
\end{displaymath}
Moreover, every $\overline{\mathbf{\Theta}}^{Q}(p)$ is $(\overline{\delta},\beta,\overline{C}_Q)$-spread-out, with $\overline{C}_Q \approx_{\delta} C$.

Dyadically pigeonhole $m(p_0)$ and $|\overline{\mathbf{\Theta}}^{Q}(S_Q(p_0))|$. After refining $P_Q$ by a subset of cardinality $\approx_\delta|P_Q|$, which we continue to denote by $P_Q$, there are dyadic numbers $m_Q,M_Q \geq 1$ such that, writing again $P^{Q}=S_Q(P_Q)$,
\begin{displaymath}
    m_Q\le m(p_0)<2m_Q
    \qquad\text{and}\qquad
    M_Q\le|\overline{\mathbf{\Theta}}^{Q}(S_Q(p_0))|<2M_Q,
    \qquad
    p_0 \in P_Q.
\end{displaymath}
Set $\mathbf{\Theta}^{Q}(p):=\overline{\mathbf{\Theta}}^{Q}(p)$. Then $m_QM_Q \approx_\delta M$, and each $\mathbf{\Theta}^{Q}(p)$ is $(\overline{\delta},\beta,C_Q)$-spread-out with $C_Q \approx_{\delta} C$. The multiplicity bounds in Lemma \ref{lem:spread-out-single-set-zoomed-in}, together with \eqref{form4}, prove \ref{it:induction-zoomed-direction-multiplicity}. Since $\mathbf{\Theta}^{Q}(p) \subset \mathcal{D}_{\overline{\delta}}(\mathbf{\Theta}(p_0))$ for $p=S_Q(p_0)$, this proves \ref{it-induction-iii}.

We now define
\begin{displaymath}
    P:=\bigcup_{Q \in \mathcal{Q}} P_Q.
\end{displaymath}
Every $P_Q$ is nonempty, so $\mathcal D_\Delta(P)=\mathcal Q$. The relations $|P_Q|\approx_\delta|P_0\cap Q|$, $Q\in\mathcal Q$, and $|\mathcal Q|\approx_\delta|P_0|_\Delta$ from \eqref{form4} yield
\begin{displaymath}
    |P|_{\Delta}=|\mathcal{Q}| \approx_{\delta} |P_0|_{\Delta}
    \qquad\text{and}\qquad
    |P \cap Q|=|P_Q| \approx_{\delta} |P_0 \cap Q|, \quad Q \in \mathcal{D}_{\Delta}(P).
\end{displaymath}
In particular, $P_{\Delta}=\{z(Q):Q \in \mathcal{D}_{\Delta}(P)\}$. By \eqref{form4} and the spread-out property of the families $\mathbf{\Theta}(p)$, the pair $(P,\mathbf{\Theta})$ is a $(\delta,\beta,C_1,M)$-spread-out configuration with $C_1 \approx_{\delta} C$ and $M\approx_\delta M_0$. This proves \ref{it-induction-i}. Moreover, by construction,
\begin{displaymath}
    \bigcup_{p \in P \cap Q} \mathbf{\Theta}(p) \subset \mathbf{\Theta}_{\Delta}(z(Q)), \qquad Q \in \mathcal{D}_{\Delta}(P),
\end{displaymath}
so \ref{it-induction-ii} is now fully verified.

\subsubsection{Proof of the Furstenberg estimate}

It remains to prove \ref{it-induction-iv}.

For $Q \in \mathcal{Q}$, let $\mathcal{F}(Q)\subset \widetilde{\mathcal{F}}$ denote the family of affine $\delta$-cubes arising from pairs $(p_0,\theta)$ with $p_0 \in P_Q$ and $\theta \in z\mathbf{\Theta}(p_0)$. For each affine $\overline{\delta}$-cube in $\mathcal{F}^{Q}$, choose one representative pair $(S_Q(p_0),\mathbf{J})$ with $p_0 \in P_Q$ and $\mathbf{J} \in \mathbf{\Theta}^{Q}(S_Q(p_0))$. Since $\mathbf{J}$ contains $\approx_{\delta} m_Q$ elements of $\mathbf{\Theta}(p_0)$, this representative gives rise to $\approx_{\delta} m_Q$ affine $\delta$-cubes in $\mathcal{F}(Q)$. After rescaling $Q$ to $[0,1]^d$, each such affine $\delta$-cube lies within $O_{d,k}(\overline\delta)$ of the affine $\overline\delta$-cube represented by $(S_Q(p_0),\mathbf J)$. Hence each affine $\delta$-cube arises from $O_{d,k}(1)$ chosen representatives, and therefore
\begin{equation}\label{localPacketLowerBound}
    |\mathcal{F}(Q)|
    \gtrapprox_{\delta}
    m_Q |\mathcal{F}^{Q}|
    \approx_{\delta}
    \frac{M}{M_Q}|\mathcal{F}^{Q}|,
    \qquad Q \in \mathcal{Q}.
\end{equation}

Now fix $\mathbf{F} \in \mathcal{F}_{\Delta}$ (the Furstenberg set defined at the end of Section \ref{s651}). Since $\mathbf{F} \in \mathcal{F}_{\Delta}^{\mathrm{pre}}$, we have
\begin{displaymath}
    N_{\Delta} \leq \left|\widetilde{\mathcal{F}} \cap \bigcup \mathcal{N}_{\Delta}(\mathbf{F})\right| < 2N_{\Delta}
\end{displaymath}
by \eqref{form6}. Since $\widetilde{\mathcal{F}} \subset \mathcal{F}_0$, and the families $\mathcal{N}_{\Delta}(\mathbf{F})$, $\mathbf{F} \in \mathcal{F}_{\Delta}$, have bounded overlap, this implies
\begin{equation}\label{form3}
    |\mathcal{F}_0| \geq |\widetilde{\mathcal{F}}| \gtrsim N_{\Delta}|\mathcal{F}_{\Delta}|.
\end{equation}
On the other hand, since $\mathcal{F}(Q)\subset \widetilde{\mathcal{F}}$, for $Q \in \mathcal{Q}$ with $q=z(Q)$,
\begin{displaymath}
    \mathcal{F}(Q) \subset \bigcup_{\mathbf{F} \in \mathcal{F}_{\Delta}\cap \mathcal{D}_{\Delta}(\{V_{\theta}+q : \theta \in z\mathbf{\Theta}_{\Delta}(q)\})} \left(\widetilde{\mathcal{F}} \cap \bigcup \mathcal{N}_{\Delta}(\mathbf{F})\right).
\end{displaymath}
Indeed, if an affine $\delta$-cube in $\mathcal{F}(Q)$ arises from $(p,\mathbf{I})$ with $p \in P_Q$ and $\mathbf{I} \in \mathbf{\Theta}(p)$, then $\mathbf{I} \subset \mathbf{J}$ for some $\mathbf{J} \in \mathbf{\Theta}_{\Delta}(q)$ by construction of $\mathbf{\Theta}(p)$ (recall \eqref{form14}). By the definition of $\mathcal{F}_{\Delta}$, the dyadic cube $\mathbf{F}_{\mathbf{J}}:=\mathcal{D}_{\Delta}(V_{z(\mathbf{J})}+q)$ belongs to $\mathcal{F}_{\Delta}$. Moreover, \eqref{form8} yields $\mathcal{D}_{\Delta}(V_{z(\mathbf{I})}+p) \in \mathcal{N}_{\Delta}(\mathbf{F}_{\mathbf{J}})$. Therefore
\begin{displaymath}
    |\mathcal{F}(Q)|
    \lesssim
    |\mathcal{F}_{\Delta}\cap \mathcal{D}_{\Delta}(\{V_{\theta}+q : \theta \in z\mathbf{\Theta}_{\Delta}(q)\})|\,N_{\Delta}
    \sim
    M_{\Delta}N_{\Delta}.
\end{displaymath}
Combining this upper bound with the lower bound \eqref{localPacketLowerBound}, we get, for every $Q \in \mathcal{Q}$,
\begin{equation}\label{eq:induction-scale-N-lower}
    N_{\Delta}
    \gtrapprox_{\delta}
    \frac{M}{M_{\Delta}M_Q}|\mathcal{F}^{Q}|.
\end{equation}
Substituting \eqref{eq:induction-scale-N-lower} into \eqref{form3}, and taking the maximum over $Q \in \mathcal{Q}=\mathcal{D}_{\Delta}(P)$, we deduce that
\begin{displaymath}
    |\mathcal{F}_0|
    \gtrapprox_{\delta}
    \frac{|\mathcal{F}_{\Delta}|}{M_{\Delta}}
    \cdot
    M\cdot
    \max_{Q \in \mathcal{D}_{\Delta}(P)} \frac{|\mathcal{F}^{Q}|}{M_Q}.
\end{displaymath}
Since $M \approx_{\delta} M_0$ (recall \eqref{form4}), we conclude that
\begin{displaymath}
    \frac{|\mathcal{F}_0|}{M_0}
    \gtrapprox_{\delta}
    \frac{|\mathcal{F}_{\Delta}|}{M_{\Delta}}
    \cdot
    \max_{Q \in \mathcal{D}_{\Delta}(P)} \frac{|\mathcal{F}^{Q}|}{M_Q}.
\end{displaymath}
This completes the proof of Proposition \ref{prop:induction-on-scales}.

\subsection{Iteration over several scales}

We will need to iterate Proposition \ref{prop:induction-on-scales} over several scales. The following standard definition keeps track of the sizes
of uniform sets at all scales.
\begin{definition}[Branching function]\label{def:branching-function}
If $P$ is $(\Delta,m)$-uniform with branching numbers $(N_j)_{j=0}^{m-1}$, its \emph{branching function} is the piecewise-linear function $f=f_P:[0,m]\to\R$ determined by
\begin{displaymath}
    f(i):=\log_{1/\Delta}|P|_{\Delta^i}
    =\sum_{j=0}^{i-1}\log_{1/\Delta}N_j,
    \qquad i=0,\ldots,m.
\end{displaymath}
It satisfies $f(0)=0$ and is non-decreasing and $d$-Lipschitz.
\end{definition}

Repeated dyadic pigeonholing yields the following statement; see \cite[Lemma 7.3]{OS23}.
\begin{lemma} \label{lem:uniformization}
    Fix $\Delta\in 2^{-\N}$ and $m\in \N$, and write $\delta=\Delta^m$. For any set $P\subset [0,1)^d$ there is a subset $P'\subset P$ such that $P'$ is $(\Delta,m)$-uniform and
    \begin{displaymath}
        |P'|_{\delta}
        \ge (4(d+1)\log(1/\Delta))^{-m}|P|_{\delta}.
    \end{displaymath}
\end{lemma}

If $S$ is $(\Delta,m)$-uniform and $T\subset S$ satisfies $|T|_{\Delta^m}\ge\Delta^u|S|_{\Delta^m}$, then each occupied $\Delta^i$-cube of $S$ contains the same number of occupied $\Delta^m$-cubes, so
\begin{equation}\label{eq:uniform-subset-branching}
    \log_{1/\Delta}|S|_{\Delta^i}-u
    \le\log_{1/\Delta}|T|_{\Delta^i}
    \le\log_{1/\Delta}|S|_{\Delta^i}
    \qquad(0\le i\le m,\ i\in\Z).
\end{equation}

\begin{corollary}\label{cor:multiscale-induction-on-scales}
    Fix $C,\gamma>0$, $\beta\in(0,1]$, and $J_0\in\N$. Let $\Delta\in2^{-\N}$ satisfy $\Delta\ll_{d,k,\beta,J_0,\gamma}1$, take $m\gg_{d,k,\beta,J_0,\gamma,\Delta}1$, and write $\delta=\Delta^m$. Let $1\le J\le J_0$ and let $0=a_0<a_1<\cdots<a_J=m$ be integers. Let $(P_0,\mathbf\Theta_0)$ be a $(\delta,\beta,C,M_0)$-spread-out configuration whose point set is $(\Delta,m)$-uniform, and let $f$ be the branching function of $P_0$. For $1\le j\le J$, there is an intermediate configuration $(P_j,\mathbf\Theta_j)$ such that
    \begin{displaymath}
        (P_j,\mathbf\Theta_j)\text{ is }
        (\Delta^{a_j-a_{j-1}},\beta,\delta^{-\gamma}C,M_j)\text{-spread-out},
    \end{displaymath}
    with Furstenberg set $\mathcal F_j$, and
    \begin{equation}\label{eq:multiscale-induction-product}
        \frac{|\mathcal F(P_0,\mathbf\Theta_0)|_\delta}{M_0}
        \gtrapprox_\delta
        \prod_{j=1}^J
        \frac{|\mathcal F_j|_{\Delta^{a_j-a_{j-1}}}}{M_j},
    \end{equation}
    where the implicit exponent in $\gtrapprox_\delta$ may depend on $J_0$. Moreover, the  sets $P_j$ are $(\Delta,a_j-a_{j-1})$-uniform, and their branching functions $g_j$ satisfy
    \begin{equation}\label{eq:multiscale-induction-branching}
        \left|g_j(i)-f(a_{j-1}+i)+f(a_{j-1})\right|
        \le\gamma m  \quad\text{for } 0\le i\le a_j-a_{j-1}.
    \end{equation}
\end{corollary}

\begin{proof}
    We first obtain the asserted configurations with spread-out constants at most
    \begin{displaymath}
        C(\log(1/\delta))^{O_{\beta,J_0}(1)}
    \end{displaymath}
    and with the right-hand side of \eqref{eq:multiscale-induction-branching} replaced by
    \begin{equation}\label{eq:multiscale-induction-branching-error}
        O_{J_0}\bigl(m\log_{1/\Delta}(4(d+1)\log(1/\Delta))\bigr)
        +O_{\Delta,J_0}(\log m).
    \end{equation}
    We argue by induction on $J$ for these auxiliary conclusions. The case $J=1$ follows by taking $(P_1,\mathbf\Theta_1)=(P_0,\mathbf\Theta_0)$. Suppose that $J\ge2$, and apply Proposition \ref{prop:induction-on-scales} to the scale triple $(\Delta^m,\Delta^{a_1},1)$. Take the coarse configuration as $(P_1,\mathbf\Theta_1)$ and choose a zoomed-in configuration attaining the maximum in \eqref{eq:lower-bound-Furstenberg}. Lemma \ref{lem:uniformization} yields uniform subsets of both point sets containing at least $(4(d+1)\log(1/\Delta))^{-m}$ of their points; restrict the shadings accordingly. Apply the induction hypothesis to the zoomed-in configuration and the scales $0,a_2-a_1,\ldots,a_J-a_1=m-a_1$. These restrictions only decrease the right-hand side of \eqref{eq:lower-bound-Furstenberg}; multiplying that inequality by \eqref{eq:multiscale-induction-product} for the zoomed-in configuration, we get \eqref{eq:multiscale-induction-product} for $(P_0,\mathbf\Theta_0)$.

    Before uniformization, Proposition \ref{prop:induction-on-scales}\ref{it-induction-i}, \ref{it-induction-ii}\textup{(Ca)}, and \ref{it-induction-iii}\textup{(Za)} show that the coarse point set is a subset of
    $\{z(Q):Q\in\mathcal D_{\Delta^{a_1}}(P_0)\}$ of cardinality $\approx_\delta|P_0|_{\Delta^{a_1}}$, while the chosen zoomed-in set is a subset of $S_Q(P_0\cap Q)$ of cardinality $\approx_\delta|P_0\cap Q|$ for some $Q\in\mathcal D_{\Delta^{a_1}}(P_0)$. By \eqref{eq:uniform-subset-branching}, the branching functions of the uniformized coarse and zoomed-in point sets differ from $f(i)$ and $f(a_1+i)-f(a_1)$, respectively, by at most
    \begin{displaymath}
        m\log_{1/\Delta}(4(d+1)\log(1/\Delta))+O_\Delta(\log m).
    \end{displaymath}
    Combining these bounds with the induction hypothesis proves \eqref{eq:multiscale-induction-branching} with $\gamma m$ replaced by \eqref{eq:multiscale-induction-branching-error}. Since $J\le J_0$, the conditions on $\Delta$ and $m$ imply that the spread-out constant is at most $\delta^{-\gamma}C$ and that \eqref{eq:multiscale-induction-branching-error} is at most $\gamma m$.
\end{proof}

\subsection{Sticky spread-out Furstenberg sets}\label{sec:Furstenberg-sticky}

We are now able to complete the proof of Theorem \ref{thm:AhlforsFurstenberg}. We use the affine coordinates \eqref{eq:affine-incidence-coordinates}. Recall from Section \ref{s:bigcell} that the Grassmannian and Euclidean metrics are bi-Lipschitz equivalent on bounded subsets of these coordinates.

The proof follows \cite[Proof of Theorem 5.2]{OrponenShmerkin26}, with Proposition \ref{prop:induction-on-scales} replacing \cite[Proposition 5.2]{OS23}. 
\begin{proof}[Proof of Theorem \ref{thm:AhlforsFurstenberg}]
    Take $0<\eta\ll_{t}1$. Apply Theorem \ref{thm:AhlforsProjection} in dimensions $(d,d-k)$ with $(s,t,u)=(s/4,t,\min\{t,d-k\}-\eta/4)$. Choose $\e\ll_{\eta} 1$ such that the conclusion of that theorem is valid with parameter $3\e$, and set $N:=\e^{-3}\in\N$. There is $C_{\mathrm{loss}}=C_{\mathrm{loss}}(s,t,\eta,\e)$ such that the product of the non polylogarithmic losses from the $N-1$ applications of Proposition \ref{prop:induction-on-scales}, the regularity estimates for the sets $\check P_j$, and Lemmas \ref{lem:spread-out-discretization} and \ref{lem:spread-out-almost-cover-families} is at most $\delta^{-C_{\mathrm{loss}}\e'}$. Choose $0<\e'\ll_{s,t,\eta,\e}1$ such that $C_{\mathrm{loss}}\e'\le\e^4/4$.

    Set
    \begin{displaymath}
        \tau:=\delta^{1/N}=\delta^{\e^3},
        \qquad
        \Delta_j:=\delta\tau^{-j}=\tau^{N-j},
        \qquad 0\leq j\leq N-1.
    \end{displaymath}
    Then $\Delta_0=\delta$, $\Delta_{j-1}/\Delta_j=\tau$, and $\Delta_j\leq\tau$ for $0\leq j\leq N-1$. Taking $\delta\ll_{s,t,\eta,\e}1$, Theorem \ref{thm:AhlforsProjection} applies at scale $\tau$. Since $N=\e^{-3}$, we may also require that the product of all polylogarithmic factors from these applications is at most $\delta^{-\e^4/4}$. The product of all losses is therefore at most
    \begin{displaymath}
        \delta^{-\e^4/2}=\tau^{-\e/2}\leq\tau^{-\e}.
    \end{displaymath}
    Consequently, every spread-out and regularity constant used in the recursion is at most $\tau^{-\e}$.

    We now define the multiscale recursion, using checks for the configurations rescaled to $[0,1]^d$. Set $(P_0,\mathbf\Theta_0):=(P,\mathbf\Theta)$ and $M_0:=M$, and let $\mathcal F_0:=\mathcal F$. Suppose that the $(\Delta_{j-1},s,\delta^{-O(\e')},M_{j-1})$-spread-out configuration $(P_{j-1},\mathbf\Theta_{j-1})$ has been defined.
    
    By dyadic pigeonholing, there are $\mathcal Q_j\subset\mathcal D_{\Delta_j}(P_{j-1})$ and  $N_j^{\mathrm{pts}}\in 2^{\N_0}$ such that
    \begin{equation}\label{eq:sticky-point-occupancy-pigeonholing}
        N_j^{\mathrm{pts}}\leq |P_{j-1}\cap Q|<2N_j^{\mathrm{pts}}
        \quad(Q\in\mathcal Q_j),
        \qquad
        \sum_{Q\in\mathcal Q_j}|P_{j-1}\cap Q|\gtrapprox_\delta |P_{j-1}|.
    \end{equation}
    Set $\widetilde P_{j-1}:=P_{j-1}\cap\bigcup\mathcal Q_j$, let $\widetilde{\mathbf\Theta}_{j-1}$ be the restriction of $\mathbf\Theta_{j-1}$ to $\widetilde P_{j-1}$, and denote its Furstenberg set by $\widetilde{\mathcal F}_{j-1}$. 
    
    Apply Proposition \ref{prop:induction-on-scales} to $(\widetilde P_{j-1},\widetilde{\mathbf\Theta}_{j-1})$ with $(\delta,\Delta,\overline\Delta)=(\Delta_{j-1},\Delta_j,\tau)$. Denote the configuration in Proposition \ref{prop:induction-on-scales}\ref{it-induction-i} by $(\widehat P_{j-1},\widehat{\mathbf\Theta}_{j-1})$, and denote the coarse configuration in Proposition \ref{prop:induction-on-scales}\ref{it-induction-ii} by $(P_j,\mathbf\Theta_j)$, with shading parameter $M_j$ and Furstenberg set $\mathcal F_j$. Thus
    \begin{align*}
        P_j
        &=\{z(Q):Q\in\mathcal D_{\Delta_j}(\widehat P_{j-1})\}, \\
        \bigcup_{p\in\widehat P_{j-1}\cap Q}\widehat{\mathbf\Theta}_{j-1}(p)
        &\subset\mathbf\Theta_j(z(Q)),
        \qquad Q\in\mathcal D_{\Delta_j}(\widehat P_{j-1}).
    \end{align*}
    There is also an integer $L_j\geq1$ such that $|\{\mathbf J\in\mathbf\Theta_j(q):\mathbf J\subset\mathbf I\}|=L_j$ for $q\in P_j$ and $\mathbf I\in\mathcal D_\tau(\mathbf\Theta_j(q))$.
    For $Q\in\mathcal D_{\Delta_j}(\widehat P_{j-1})$, denote the configuration in Proposition \ref{prop:induction-on-scales}\ref{it-induction-iii} by $(\check P_{j,Q},\check{\mathbf\Theta}_{j,Q})$, its shading parameter by $M_Q$, and its Furstenberg set by $\check{\mathcal F}_{j,Q}$. Then
    \begin{align*}
        \check P_{j,Q}&=S_Q(\widehat P_{j-1}\cap Q), \\
        \check{\mathbf\Theta}_{j,Q}(S_Q(p))
        &\subset\mathcal D_\tau(\widehat{\mathbf\Theta}_{j-1}(p)),
        \qquad p\in\widehat P_{j-1}\cap Q.
    \end{align*}
    This defines the recursion for $1\leq j\leq N-1$.

    Choose $q_{N-1}\in P_{N-1}$, and let $Q_{N-1}\in\mathcal D_{\Delta_{N-1}}(\widehat P_{N-2})$ satisfy $z(Q_{N-1})=q_{N-1}$. Recursively, after choosing $Q_{j+1}$, choose $q_j\in\widehat P_j\cap Q_{j+1}$ and let $Q_j\in\mathcal D_{\Delta_j}(\widehat P_{j-1})$ be the cube with $z(Q_j)=q_j$. Then
    \begin{displaymath}
        Q_1\subset\cdots\subset Q_{N-1},
        \qquad
        q_j\in\widehat P_j\cap Q_{j+1},
        \quad 1\leq j\leq N-2.
    \end{displaymath}
    For $1\leq j\leq N-1$, set $(\check P_j,\check{\mathbf\Theta}_j):=(\check P_{j,Q_j},\check{\mathbf\Theta}_{j,Q_j})$ and $\check{\mathcal F}_j:=\check{\mathcal F}_{j,Q_j}$. Proposition \ref{prop:induction-on-scales}\ref{it-induction-iv} gives
    \begin{displaymath}
        \frac{|\mathcal F_{j-1}|_{\Delta_{j-1}}}{M_{j-1}}
        \geq \frac{|\widetilde{\mathcal F}_{j-1}|_{\Delta_{j-1}}}{M_{j-1}}
        \gtrapprox_\delta
        \frac{|\mathcal F_j|_{\Delta_j}}{M_j}
        \frac{|\check{\mathcal F}_j|_\tau}{M_{Q_j}}.
    \end{displaymath}
    Iterating and using $|\mathcal F_{N-1}|_{\Delta_{N-1}}\geq M_{N-1}$ yields
    \begin{equation}\label{eq:sticky-product}
        \frac{|\mathcal F|_\delta}{M}
        \gtrapprox_\delta
        \prod_{j=1}^{N-1}\frac{|\check{\mathcal F}_j|_\tau}{M_{Q_j}}.
    \end{equation}

   Each $(\check P_j,\check{\mathbf\Theta}_j)$ is a $(\tau,s,\delta^{-O(\e')},M_{Q_j})$-spread-out configuration, and $M_{Q_j}\lesssim\tau^{-k(d-k)}$. We first record the point estimates inherited by these configurations. The Frostman condition at scale $\delta$ and the covering bound in Definition \ref{def:Ahlfors-regular} give, respectively
    \begin{displaymath}
        |P|\geq\delta^{\e'}\delta^{-t},
        \qquad
        |P\cap B(x,R)|_r\lesssim_d\delta^{-\e'}(R/r)^t
        \quad(x\in\R^d,\ \delta\leq r\leq R).
    \end{displaymath}
    Every point of $P_j$ is the lower-left endpoint of a $\Delta_j$-cube meeting $P$, so
    \begin{equation}\label{eq:sticky-coarse-point-covering}
        |P_j\cap B(x,R)|_r\lesssim_d\delta^{-\e'}(R/r)^t
        \quad(x\in\R^d,\ \Delta_j\leq r\leq R).
    \end{equation}
    Equations \eqref{eq:sticky-point-occupancy-pigeonholing} and \eqref{eq:sticky-coarse-point-covering}, together with Proposition \ref{prop:induction-on-scales}\ref{it-induction-i}, imply inductively that, for $1\leq j\leq N-1$,
    \begin{equation*}
        |P_j|N_j^{\mathrm{pts}}
        \approx_{\Delta_{j-1}}|\widehat P_{j-1}|
        \gtrapprox_\delta|P_{j-1}|,
        \qquad
        |P_j|\geq\delta^{O(\e')}\Delta_j^{-t},
        \qquad
        \delta^{O(\e')}\tau^{-t}
        \le N_j^{\mathrm{pts}}
        \le\delta^{-O(\e')}\tau^{-t}.
    \end{equation*}
    For $Q\in\mathcal D_{\Delta_j}(\widehat P_{j-1})$, Proposition \ref{prop:induction-on-scales}\ref{it-induction-i} also gives
    \begin{displaymath}
        |\widehat P_{j-1}\cap Q|\approx_{\Delta_{j-1}}N_j^{\mathrm{pts}}.
    \end{displaymath}
    Applying $S_Q$, then using \eqref{eq:sticky-coarse-point-covering} with $j-1$ and the lower bound for $N_j^{\mathrm{pts}}$, shows that $\check P_j$ is $(\tau,t,\delta^{-O(\e')})_d$-regular. Hence
    \begin{equation}\label{eq:sticky-intermediate-regularity}
    \begin{gathered}
        \check P_j\text{ is }(\tau,t,\tau^{-\e})_d\text{-regular},\\
        \check{\mathbf\Theta}_j(p)\text{ is }(\tau,s,\tau^{-\e})_{d,k}\text{-spread-out}
        \qquad(p\in \check P_j).
    \end{gathered}
    \end{equation}

    For $1\leq j\leq N-2$, define
    \begin{displaymath}
        \overline\Theta_j
        :=\{z(\mathbf I):\mathbf I\in\mathcal D_\tau(\mathbf\Theta_j(q_j))\}
        \subset\tau\Z^{k(d-k)}.
    \end{displaymath}
    Lemma \ref{lem:spread-out-discretization}\ref{it:spread-out-discretization-cubes}, Corollary \ref{cor:spread-out-coarsening}, and scale $\tau$ uniformity show that
    \begin{equation}\label{eq:sticky-intermediate-direction-spread-out}
        m_{\overline\Theta_j}
        \text{ is }(\tau,s/4,\tau^{-\e})_{d,k}\text{-spread-out}.
    \end{equation}

    If $p=S_{Q_j}(p_0)\in \check P_j$, we get from Proposition \ref{prop:induction-on-scales}\ref{it-induction-ii}(Cb) and \ref{it-induction-iii}(Zb) that
    \begin{displaymath}
        \check{\mathbf\Theta}_j(p)
        \subset\mathcal D_\tau(\widehat{\mathbf\Theta}_{j-1}(p_0))
        \subset\mathcal D_\tau(\mathbf\Theta_j(q_j)).
    \end{displaymath}
    This implies the first inequality in
    \begin{equation*}
        M_{Q_j}\leq|\overline\Theta_j|
        =|\mathcal D_\tau(\mathbf\Theta_j(q_j))|
        \lessapprox_\delta M_{Q_{j+1}},
        \qquad 1\leq j\leq N-2.
    \end{equation*}
    For the second inequality, scale $\tau$ uniformity and Proposition \ref{prop:induction-on-scales}\ref{it-induction-i} imply that $|\widehat{\mathbf\Theta}_j(q_j)|\approx_\delta L_j|\overline\Theta_j|$. Proposition \ref{prop:induction-on-scales}\ref{it:induction-zoomed-direction-multiplicity} gives $m_{Q_{j+1}}M_{Q_{j+1}}\approx_\delta|\widehat{\mathbf\Theta}_j(q_j)|$, and $m_{Q_{j+1}}\leq L_j$.

    Let $E:=\{j\in\{1,\ldots,N-2\}:M_{Q_{j+1}}\geq\tau^{-\e}M_{Q_j}\}$. Multiplying the consecutive ratios and using $M_{Q_j}\lessapprox_\delta M_{Q_{j+1}}$ and $M_{Q_j}\lesssim\tau^{-k(d-k)}$, we obtain $|E|\leq2k(d-k)\e^{-1}$ for $\delta\ll_{d,k,\e}1$.

    We claim that, for every $j\in\{1,\ldots,N-2\}\setminus E$,
    \begin{equation}\label{eq:local-claim-general}
        \frac{|\check{\mathcal F}_j|_\tau}{M_{Q_j}}
        \geq\tau^{-\min\{t,d-k\}+\eta/4+3\e}.
    \end{equation}
    Granting this, we combine \eqref{eq:sticky-product} with the trivial bound $|\check{\mathcal F}_j|_\tau\geq M_{Q_j}$ to obtain
    \begin{displaymath}
        \frac{|\mathcal F|_\delta}{M}
        \gtrapprox_\delta
        \tau^{(-\min\{t,d-k\}+\eta/4+3\e)(N-|E|-2)}.
    \end{displaymath}
    The choice of $\e$ and the bound for $|E|$ imply
    \begin{displaymath}
        \frac{N-|E|-2}{N}
        \left(-\min\{t,d-k\}+\frac{\eta}{4}+3\e\right)
        \leq-\min\{t,d-k\}+\frac{\eta}{4}+4\e <\e.
    \end{displaymath}
    Since $\tau=\delta^{1/N}$,
    \begin{displaymath}
        \frac{|\mathcal F|_\delta}{M}
        \gtrapprox_\delta
        \delta^{-\min\{t,d-k\}+\eta/4+4\e}.
    \end{displaymath}
    Since $\e<\eta/25$, \eqref{form17} follows for $\delta\ll_{d,k,s,t,\eta,\e}1$, conditional on \eqref{eq:local-claim-general}.

    Fix $j\in\{1,\ldots,N-2\}\setminus E$. For $\theta\in\overline\Theta_j$, let $\overline P_\theta:=\{p\in \check P_j:\theta\in z\check{\mathbf\Theta}_j(p)\}$. Since $j\notin E$ and $|\overline\Theta_j|\lessapprox_\delta M_{Q_{j+1}}$,
    \begin{align*}
        \sum_{\theta\in\overline\Theta_j}|\overline P_\theta|
        &=\sum_{p\in \check P_j}|z\check{\mathbf\Theta}_j(p)|
        \sim |\check P_j|M_{Q_j} \\
        &\geq\tau^\e|\check P_j|M_{Q_{j+1}}
        \gtrapprox_\delta\tau^\e|\check P_j||\overline\Theta_j|.
    \end{align*}
    Let $\Theta'$ consist of the parameters $\theta\in\overline\Theta_j$ such that $|\overline P_\theta|\geq\tau^{2\e}|\check P_j|$. Discarding the remaining parameters from the displayed incidence sum yields $|\Theta'|\gtrapprox_\delta\tau^\e|\overline\Theta_j|$ for $\delta\ll_\e1$.

    Equation \eqref{eq:sticky-intermediate-direction-spread-out} and Lemma \ref{lem:spread-out-subset} show that $m_{\Theta'}$ is $(\tau,s/4,\tau^{-3\e})_{d,k}$-spread-out. Since $\Theta'\subset\overline\Theta_j\subset\tau\Z^{k(d-k)}$, it is $\tau$-separated. The linear isometry $\theta\mapsto-\theta^*$ maps $\mathfrak{H}_{d,k}$ onto $\mathfrak{H}_{d,d-k}$ and preserves the dyadic lattice. Thus Lemma \ref{lem:spread-out-discretization}\textup{\ref{it:spread-out-discretization-cubes}} implies that $\{-\theta^*:\theta\in\Theta'\}$ is a $\tau$-separated $(\tau,s/4,\tau^{-3\e})_{d,d-k}$-spread-out set. By \eqref{eq:sticky-intermediate-regularity}, the coordinate permutation $(x,y)\mapsto(y,x)$ sends $\check P_j$ to a $(\tau,t,\tau^{-\e})_d$-regular set. Apply Theorem \ref{thm:AhlforsProjection}, in dimensions $(d,d-k)$ and with parameter $3\e$, to this set and $\{-\theta^*:\theta\in\Theta'\}$. For every $\theta\in\Theta'$, the subset $\overline P_\theta$ satisfies the cardinality hypothesis of the theorem because
    \begin{displaymath}
        |\overline P_\theta|_\tau
        =|\overline P_\theta|
        \geq\tau^{2\e}|\check P_j|
        \geq\tau^{3\e}|\check P_j|_\tau,
        \qquad \theta\in\Theta'.
    \end{displaymath}
    Since $\pi_{-\theta^*}(y,x)=y-\theta^*x=\rho_\theta(x,y)$, the theorem yields $\Theta''\subset\Theta'$ with $|\Theta''|\geq(1-\tau^{3\e})|\Theta'|\geq\tfrac12|\Theta'|$ such that
    \begin{displaymath}
        |\rho_\theta(\overline P_\theta)|_\tau
        \geq\tau^{-\min\{t,d-k\}+\eta/4},
        \qquad \theta\in\Theta''.
    \end{displaymath}

    Equation \eqref{eq:affine-incidence-counting}, applied to $\Theta''$ and the sets $\overline P_\theta$, implies
    \begin{align*}
        |\check{\mathcal F}_j|_\tau
        &\geq\sum_{\theta\in\Theta''}|\rho_\theta(\overline P_\theta)|_\tau 
\geq|\Theta''|\tau^{-\min\{t,d-k\}+\eta/4} \\
        &\gtrapprox_\delta\tau^\e|\overline\Theta_j|
        \tau^{-\min\{t,d-k\}+\eta/4} \geq\tau^{3\e}M_{Q_j}
        \tau^{-\min\{t,d-k\}+\eta/4},
    \end{align*}
    where the last inequality follows from $\delta\ll_\e1$ and $M_{Q_j}\leq|\overline\Theta_j|$. This is \eqref{eq:local-claim-general}.
\end{proof}

\section{Examples and applications}
\label{sec:examples-applications}

\subsection{Uniformly perfect measures on \texorpdfstring{$\R$}{R} are spread-out}
\label{s:uniformly-perfect-are-spread-out}

Let us first recall (from the introduction) the notion of uniformly perfect measures:
\begin{definition}[$(D,\beta)$-uniformly perfect measure]\label{def:uniformlyPerfect1} Let $D > 1$ and $\beta \in [0,1)$. A locally finite Borel measure $\mu$ on a metric space $(X,\rho)$ is called \emph{$(D,\beta)$-uniformly perfect} if $\diam(\spt \mu) > 0$, and
\begin{displaymath} \mu(B(x,r)) \leq \beta \cdot \mu(B(x,Dr)) \end{displaymath}
for all closed balls $B(x,r) \subset X$ such that $\spt \mu \not\subset B(x,Dr)$. 
\end{definition}

\begin{proposition}\label{prop:uniform-is-spread-out}
    For every $D > 1$ and $\beta \in [0,1)$, there exists $\e = \e(D,\beta) > 0$ such that the following holds. Let $\mu$ be a $(D,\beta)$-uniformly perfect measure on $\mathbf{G}_{2,1}\cong\R$ such that 
    \begin{displaymath} 
        0<\mu(B(0,1)),\quad
        \mu(B(0,5)) \leq D,
        \quad\text{and}\quad
        \spt \mu \not\subset B(0,5).
    \end{displaymath}
    Then, for every $\delta\in 2^{-\N}$, $\mu|_{B(0,1)}$ is $(\delta,\e,D)_{1}$-spread-out.
\end{proposition} 
 
 \begin{proof} 
 Put $c:=1/4$. Choose $N_0\in\N$ with $\beta^{N_0}\le c$, and then fix $0<\e=\e(D,\beta)\ll_{D,\beta}1$ so that $8D^{N_0}\e\le 1$ and each interval $[\e r,r]$ contains fewer than $\e^{-1}$ dyadic scales. Iterating uniform perfectness, we get
 \begin{equation}\label{form13} \mu(B(x,8\e r)) \leq c\mu(B(x,r)), \qquad B(x,r) \subset B(0,5). \end{equation} 
Indeed, the balls $B(x,D^i8\e r)$, $0\le i\le N_0$, are contained in $B(x,r)\subset B(0,5)$, so the non-containment condition in Definition \ref{def:uniformlyPerfect1} follows from $\spt\mu\not\subset B(0,5)$.

Fix $\delta\in2^{-\N}$. Let $E \subset B(0,1)\cap\spt\mu$ be a Borel set such that every point $x \in E$ has $\overline{\mathfrak{s}}$ many $\e$-singular scales in $2^{-\N}\cap[\delta,1]$, denoted $\{\overline{r}_{j}(x)\}_{j = 1}^{\overline{\mathfrak{s}}}$. If $E=\varnothing$ or $\overline{\mathfrak{s}}=0$, the desired estimate follows from $E\subset B(0,5)$ and $\mu(B(0,5))\le D$. Since each block $[\e r,r]$ contains fewer than $\e^{-1}$ dyadic scales, we may choose, after discarding extras, a fixed integer $\mathfrak{s}$ and subsequences $\{r_{j}(x)\}_{j = 1}^{\mathfrak{s}} \subset \{\overline{r}_{j}(x)\}_{j = 1}^{\overline{\mathfrak{s}}}$ such that 
\begin{equation}\label{eq:singular-scale-subsequence} \mathfrak{s} \geq \e \overline{\mathfrak{s}} \quad \text{and} \quad r_{j + 1}(x) \leq \e r_{j}(x) \text{ for all } 1 \leq j \leq \mathfrak{s} - 1\text{ and } x \in E. \end{equation}
For every selected scale, choose a point $a_j(x)\in\R$ witnessing singularity:
\begin{displaymath}
    E\cap B(x,r_j(x))\subset B(a_j(x),\e r_j(x)).
\end{displaymath}
The scales $r_{j}(x)$ may depend arbitrarily on $x$, and this poses a technical difficulty. To overcome it, we define a "smoother" version of the scales by the following formula:
\begin{displaymath} \rho_{j}(x) := \max \{r_{j}(y) : y \in E \text{ and } |x - y| < \tfrac12 r_{j}(y)\}, \qquad 1 \leq j \leq \mathfrak{s}. \end{displaymath}
(The $\max$ exists because $r_{j}(y)$ only takes values in $2^{-\N} \cap [\delta,1]$.) A key point is that whenever $|x - y| < \tfrac12 r_{j}(y)$, then $x \in E \cap B(y,r_{j}(y))$. Since the affine hyperplanes in $\R$ are points, the singularity of $r_j(y)$ gives
\begin{displaymath} 
    E \cap B\bigl(x,\tfrac{1}{2}r_{j}(y)\bigr) \subset E \cap B(y,r_{j}(y)) \subset B\bigl(a_j(y),\e r_{j}(y)\bigr) \subset B\bigl(x,2\e r_{j}(y)\bigr). 
\end{displaymath}
Applying this with the point $y$ which realizes $\rho_{j}(x)$, we conclude
\begin{equation}\label{form18} 
    E \cap B(x,\tfrac{1}{2}\rho_{j}(x)) \subset B(x,2\e \rho_{j}(x)), \qquad 1 \leq j \leq \mathfrak{s}. 
\end{equation}
For notational convenience, define $\rho_{0}(x) := \e^{-1}$ for all $x \in E$. We will next recursively construct families $\mathcal{B}_{j}$ of closed balls, $0 \leq j \leq \mathfrak{s}$, with the following properties:
\begin{itemize}
\item[(a)] $E \subset \bigcup_{B \in \mathcal{B}_{j}} \tfrac{1}{2}B$.
\item[(b)] $\sum_{B \in \mathcal{B}_{j + 1}} \mu(B) \leq \tfrac{1}{2} \sum_{B \in \mathcal{B}_{j}} \mu(B)$.
\item[(c)] Each ball in $\mathcal{B}_{j}$ has the form $B(x,4\e \rho_{j}(x))$ for some $x \in E$.
\end{itemize}
 To initialize the construction, set $\mathcal{B}_{0} := \{B(x_{0},4)\}$ for $x_{0} \in E$ arbitrary, so (c) is valid. Property (a) follows from $E \subset B(0,1)$. Property (b) is vacuous when only $\mathcal{B}_{0}$ has been constructed.

 We then assume that $\mathcal{B}_{j}$ has been constructed for some $0 \leq j \leq \mathfrak{s} - 1$, and we construct $\mathcal{B}_{j + 1}$. Fix $\mathbf{B} \in \mathcal{B}_{j}$. The plan is to construct a collection $\mathcal{B}_{j + 1}(\mathbf{B})$, and then set $\mathcal{B}_{j + 1} := \bigcup_{\mathbf{B} \in \mathcal{B}_{j}} \mathcal{B}_{j + 1}(\mathbf{B})$. To construct $\mathcal{B}_{j + 1}(\mathbf{B})$, let $\overline{\mathcal{B}}_{j + 1}(\mathbf{B})$ be the family of all closed balls of radius $\tfrac{1}{2}\rho_{j + 1}(x)$ centered at points $x\in E \cap \tfrac{1}{2}\mathbf{B}$. There exists a finite subfamily, still denoted $\overline{\mathcal{B}}_{j + 1}(\mathbf{B})$, with the properties
 \begin{equation}\label{form12} E \cap \tfrac{1}{2}\mathbf{B} \subset \bigcup_{B \in \overline{\mathcal{B}}_{j + 1}(\mathbf{B})} B \quad \text{and} \quad \sum_{B \in \overline{\mathcal{B}}_{j + 1}(\mathbf{B})} \mathbf{1}_{B} \leq 2. \end{equation}
Recall from \eqref{form18} that, writing $B := B(x,\tfrac{1}{2}\rho_{j + 1}(x))$, it holds 
\begin{displaymath} E \cap B \subset 4\e B = B(x,2\e \rho_{j + 1}(x)). \end{displaymath}
Consequently, writing $\mathcal{B}_{j + 1}(\mathbf{B}) := \{B(x,4\e \rho_{j + 1}(x)) : B(x,\tfrac{1}{2}\rho_{j + 1}(x)) \in \overline{\mathcal{B}}_{j + 1}(\mathbf{B})\}$, it holds
\begin{displaymath} E \cap \tfrac{1}{2}\mathbf{B} \subset \bigcup_{B \in \mathcal{B}_{j + 1}(\mathbf{B})} \tfrac{1}{2}B, \end{displaymath} 
and therefore 
\begin{displaymath} E \subset \bigcup_{\mathbf{B} \in \mathcal{B}_{j}} (E \cap \tfrac{1}{2}\mathbf{B}) \subset \bigcup_{\mathbf{B} \in \mathcal{B}_{j}} \bigcup_{B \in \mathcal{B}_{j + 1}(\mathbf{B})} \tfrac{1}{2}B \stackrel{\mathrm{def.}}{=} \bigcup_{B \in \mathcal{B}_{j + 1}} \tfrac{1}{2}B. \end{displaymath}
This verifies property (a), and (c) is immediate from the definition of $\mathcal{B}_{j + 1}(\mathbf{B})$. To verify (b), we need to check the following:
\begin{lemma}\label{claim1} The balls in $\mathcal{B}_{j + 1}(\mathbf{B})$ and $\overline{\mathcal{B}}_{j + 1}(\mathbf{B})$ are contained in $\mathbf{B}$. 
\end{lemma}

\begin{proof} Write $\mathbf{B}=B(x_{0},4\e\rho_j(x_{0}))$ for some $x_{0}\in E$, and let $x\in E\cap\tfrac12\mathbf{B}$ be the center of a ball chosen for $\overline{\mathcal{B}}_{j+1}(\mathbf{B})$.

If $j=0$, then $\mathbf{B}=B(x_{0},4)$ and $\rho_1(x)\le1$. Since $x\in B(x_{0},2)$ and $4\e\le\tfrac12$, both $B(x,\tfrac12\rho_1(x))$ and $B(x,4\e\rho_1(x))$ are contained in $B(x_{0},\tfrac52)\subset\mathbf{B}$.

Assume that $j\ge1$, and write $\rho_{j+1}(x)=r_{j+1}(y)$ for some $y\in E$. Then
\begin{displaymath}
    |x_{0}-y|<2\e\rho_j(x_{0})+\tfrac12r_{j+1}(y).
\end{displaymath}
If $2\e\rho_j(x_{0})\le r_{j+1}(y)$, then
\begin{displaymath}
    |x_{0}-y|<\tfrac32r_{j+1}(y)
    \le\tfrac32\e r_j(y)<\tfrac12r_j(y).
\end{displaymath}
Thus $y$ is admissible in the definition of $\rho_j(x_{0})$, and hence $\rho_{j+1}(x)\le\e\rho_j(x_{0})$. In the complementary case, $\rho_{j+1}(x)\le2\e\rho_j(x_{0})$. Therefore, in either case,
\begin{displaymath}
    \rho_{j+1}(x)\le2\e\rho_j(x_{0}).
\end{displaymath}
Since $|x-x_{0}|\le2\e\rho_j(x_{0})$, it follows that
\begin{displaymath}
    B(x,\tfrac12\rho_{j+1}(x))
    \subset B(x_{0},3\e\rho_j(x_{0}))
    \subset\mathbf{B},
\end{displaymath}
and, using $\e\le\tfrac14$,
\begin{displaymath}
    |x-x_{0}|+4\e\rho_{j+1}(x)
    \le(2\e+8\e^2)\rho_j(x_{0})
    \le4\e\rho_j(x_{0}).
\end{displaymath}
This also gives $B(x,4\e\rho_{j+1}(x))\subset\mathbf{B}$.
\end{proof} 
The ball in $\mathcal{B}_0$ is contained in $B(0,5)$, so Lemma \ref{claim1} shows inductively that every ball in $\mathcal{B}_j$, and every ball in $\overline{\mathcal{B}}_{j+1}(\mathbf B)$ for $\mathbf B\in\mathcal B_j$, is contained in $B(0,5)$. It now follows from \eqref{form13} and \eqref{form12} that
 \begin{displaymath} \sum_{B \in \mathcal{B}_{j + 1}(\mathbf{B})} \mu(B) = \sum_{B \in \overline{\mathcal{B}}_{j + 1}(\mathbf{B})} \mu(8\e B) \leq c \sum_{B \in \overline{\mathcal{B}}_{j + 1}(\mathbf{B})} \mu(B) \leq 2c\mu(\mathbf{B}) = \tfrac{1}{2}\mu(\mathbf{B}), \end{displaymath} 
Summing over $\mathbf{B} \in \mathcal{B}_{j}$ proves (b), and completes the construction of the families $\mathcal{B}_{j}$, $0 \leq j \leq \mathfrak{s}$.

Finally, we may deduce from (a)-(b) that
\begin{displaymath} \mu(E) \leq \sum_{B \in \mathcal{B}_{\mathfrak{s}}} \mu(B) \leq \tfrac{1}{2} \sum_{B \in \mathcal{B}_{\mathfrak{s} - 1}} \mu(B) \leq \ldots \leq 2^{-\mathfrak{s}} \mu(B(0,5)) \leq D2^{-\mathfrak{s}}. \end{displaymath}
Recalling from \eqref{eq:singular-scale-subsequence} that $\mathfrak{s} \geq \e \overline{\mathfrak{s}}$, in particular $\mu(E) \leq D2^{-\e \overline{\mathfrak{s}}}$, and we have shown that $\mu|_{B(0,1)}$ is $(\delta,\e,D)_{1}$-spread-out.
\end{proof}


\subsection{Applications to discretized sum-product estimates}
\label{s:discretized-sum-product}

The discretized sum-product problem has received considerable attention since the works of Katz and Tao \cite{KT01}, who introduced the problem, and Bourgain \cite{Bourgain03}, who proved the first nontrivial result:

\begin{theorem}[Bourgain]\label{thm:Bourgain-SP}
For every $s\in (0,1)$ there are $\e,\kappa>0$, depending only on $s$, such that the following holds for $\delta\ll_s 1$.

Let $A\subset [1,2]$ be $(\delta,s,\delta^{-\varepsilon})$-Frostman. Then
\begin{equation} \label{eq:discretized-SP}
    \max\{|A+A|_{\delta}, |A\cdot A|_{\delta}\} \ge \delta^{-(s+\kappa)}.
\end{equation}
\end{theorem}

It has been an important problem to understand the dependence of $\kappa$ on $s$; in Bourgain's proof it was tiny and left uncomputed. Quantitative bounds were provided in \cite{MR4283564, RaMa23, 23arXiv230808819R}. As a corollary of their resolution of the Furstenberg set problem, Ren and Wang \cite{23arXiv230808819R} showed that one can take $\kappa=(1/4)s$ for $s\in (0,2/3)$ (in the case $s\in (2/3,1)$ the sharp upper bound $\kappa=(1-s)/2$ was already known). 

The problem has also attracted attention in the special case in which $A$ has additional structure, and in particular in the Ahlfors-regular case. O'Regan \cite[Theorem~1.7]{ORegan25} proved that, for an $s$-Ahlfors regular set $A\subset[1,2]$ with $s\in(0,1/2]$, one has $\max\{|A+A|_\delta,|A\cdot A|_\delta\}\ge\delta^{-4s/3+\eta}$ for every $\eta>0$ and all sufficiently small $\delta$.  As a corollary of Theorem \ref{thm:AhlforsFurstenberg}, we obtain an exponent $3/2$ when the dimension of $A$ is at most $1/2$.

\begin{theorem} \label{thm:discretized-SP}
Let $s\in(0,1/2]$ and $C\ge1$, and suppose that $A\subset[1,2]$ is the support of an $s$-Ahlfors regular probability measure with regularity constant $C$. For every $\eta>0$ and all sufficiently small dyadic $\delta$, depending only on $s,C,\eta$, we have
\begin{equation} \label{eq:discretized-SP-strong}
    |A+A|_{\delta}\cdot |A\cdot A|_{\delta} \ge \delta^{\eta}\cdot |A|_{\delta}^3.
\end{equation}
In particular,
\begin{displaymath}
    \max\{|A+A|_{\delta}, |A\cdot A|_{\delta}\} \ge \delta^{\eta} \cdot |A|_{\delta}^{3/2}.
\end{displaymath}
\end{theorem}

\begin{proof}
We use Elekes' argument \cite{MR1472816} in dual coordinates. All implicit constants in this proof may depend on $s,C$.

Let
\begin{displaymath}
    P:=z\mathcal D_\delta\bigl(\{(b/4,(4-ab)/16):a,b\in A\}\bigr).
\end{displaymath}
The map $(a,b)\mapsto(b/4,(4-ab)/16)$ is bi-Lipschitz on $[1,2]^2$. Hence $P\subset[0,1]^2$ is $(\delta,2s,O(1))$-regular and $|P|\sim\delta^{-2s}$.

For every $p\in P$, choose $a,b\in A$ such that $p= z(\mathcal D_\delta(b/4,(4-ab)/16))$. The Ahlfors regular measures on $-(a+A)/4$ satisfy the hypotheses of Proposition \ref{prop:uniform-is-spread-out}, with constants depending only on $s,C$. Lemma \ref{lem:uniform-spread-out-discretization} therefore yields a family
\begin{displaymath}
    \mathbf\Theta(p)\subset\mathcal D_\delta(-(a+A)/4)
\end{displaymath}
which is $(\delta,\beta,\lessapprox_\delta1)_{2,1}$-spread-out for some $\beta>0$ depending only on $s,C$, and has cardinality $\gtrapprox_\delta\delta^{-s}$. Since $|A|_\delta\sim\delta^{-s}$, Lemma \ref{lem:spread-out-subset} allows us to restrict these families to a common cardinality $M\gtrapprox_\delta\delta^{-s}$, with the same spread-out bound.

If $\theta\in z\mathbf\Theta(p)$, choose $c\in A$ such that $|\theta+(a+c)/4|\le\delta$. The affine coordinates of $V_\theta+p$ differ by $O(\delta)$ from
\begin{displaymath}
    \left(-\frac{a+c}{4},\frac{4-ab}{16}+\frac{(a+c)b}{16}\right)
    =\left(-\frac{a+c}{4},\frac{4+bc}{16}\right).
\end{displaymath}
It follows that
\begin{displaymath}
    |\mathcal F(P,\mathbf\Theta)|_\delta
    \lesssim |A+A|_\delta\,|A\cdot A|_\delta.
\end{displaymath}
Theorem \ref{thm:AhlforsFurstenberg}, applied with $t=2s\le1$ and $\eta/2$ in place of $\eta$, implies
\begin{displaymath}
    |A+A|_\delta\,|A\cdot A|_\delta
    \gtrsim\delta^{\eta/2-2s}M
    \gtrapprox_\delta\delta^{\eta/2}|A|_\delta^3.
\end{displaymath}
For sufficiently small $\delta$, this proves \eqref{eq:discretized-SP-strong}.
\end{proof}

\begin{remark}
It is perhaps noteworthy that this is the first time in which a bound in the discretized sum-product problem has exceeded the best known bound in the discrete problem. Solymosi \cite{Solymosi09} proved that for any finite $A\subset\R$,
\[
    \max\{|A+A|,|A\cdot A|\}\gtrsim |A|^{4/3-o_{|A|\to\infty}(1)}.
\]
This has been improved by small amounts, with the current best exponent, due to Cushman \cite{Cushman25SumProduct}, equaling $4/3+10/4407$.
\end{remark}

\begin{remark}
    Another consequence of \eqref{eq:discretized-SP-strong} is
    \[
        |A+A|_\delta\lesssim|A|_\delta
        \quad\Longrightarrow\quad
        |A\cdot A|_\delta\gtrsim\delta^\eta|A|_\delta^2,
    \]
    and, with sums and products interchanged,
    \[
        |A\cdot A|_\delta\lesssim|A|_\delta
        \quad\Longrightarrow\quad
        |A+A|_\delta\gtrsim\delta^\eta|A|_\delta^2.
    \]
    This also shows that \eqref{eq:discretized-SP-strong} is sharp when $A$ is the $\delta$-neighborhood of a suitable arithmetic or geometric progression. In the Ahlfors-regular case, the ``few products implies maximal sums'' direction was already established by O'Regan \cite{ORegan25}. The other direction seems to be new even in the Ahlfors-regular case, but Demeter and O'Regan proved a general ``few sums implies maximal quotients'' result in \cite[Theorem 1.12]{DemeterORegan25} under only an $s$-Frostman assumption, $s\le 1/2$.
\end{remark}

A related sum-product problem concerns the size of $A+A\cdot A$ (see the work of Roche-Newton, Ruzsa, Shen and Shkredov \cite{RocheNewtonRuzsaShenShkredov19} for recent progress in the finite case of this problem). For Ahlfors regular sets, Theorem \ref{thm:projection-ADR} gives the following dimension bound.
\begin{corollary}
Let $A\subset[1,2]$ be a nonempty $s$-Ahlfors regular set with $s\in(0,1)$. Then
\begin{displaymath}
    \hdim(A+A\cdot A)\ge\min\{2s,1\}.
\end{displaymath}
\end{corollary}
\begin{proof}
The set $A\times A$ has Hausdorff and packing dimensions $2s$. Apply Theorem \ref{thm:projection-ADR} to $A\times A$ and the Ahlfors regular measure on the directions $\operatorname{span}(1,a)$, $a\in A$. For almost every $a$ with respect to the Ahlfors regular measure on $A$, we obtain $\hdim(A+aA)=\min\{2s,1\}$. Since $A+aA\subset A+A\cdot A$, the conclusion follows.
\end{proof}

\subsection{Triple sums for subsets of the parabola}
\label{s:triple-sums-parabola}

In this section, we write $\mathbb{P} := \{(x,x^{2}) : x \in [-1,1]\}$. Because the addition map $(x,y)\mapsto x+y$ is bi-Lipschitz on disjoint arcs of $\mathbb P$ (see Lemma \ref{lemma:cordoba} below), understanding sums $A+B$ or additive energies $\mathcal E(A,B)$ for $A,B\subset\mathbb{P}$ is straightforward. The situation is much more subtle for triple sums.

For finite subsets of $\mathbb P$, third and higher additive energies have been studied in \cite{BourgainDemeter15,Mudgal23DiameterFree} with a recent breakthrough achieved in \cite{CushmanDemeterWu26}; the corresponding fractal problem was considered in \cite{Orponen23Parabola,MR4956377}, see also \cite{2025arXiv251115899D,MR4869897} for closely related work. The first author \cite[Corollary~1.3]{Orponen23Parabola} proved that every closed $A\subset\mathbb P$ with $0<\hdim(A)<1$ satisfies $\hdim(A+A+A)\ge 2\hdim(A)+\varepsilon(\hdim(A))$ for some $\varepsilon(\hdim(A)) > 0$. In \cite{MR4956377} the result was quantified with $\varepsilon(\hdim A) = \min\{\hdim A/2,1-\hdim A\}$. As a particular case of their results, Algom and the first author \cite{AlgomOrponen26Flattening} established $L^2$-flattening of Ahlfors-regular sets on the parabola, which yields additive growth through iterated self-convolutions.

It follows from \cite[Theorem 1.1]{CushmanDemeterWu26} and Cauchy--Schwarz that, for every $\varepsilon>0$ and every finite $\Lambda\subset\mathbb P$, we have $|\Lambda+\Lambda+\Lambda|\gtrsim_{\varepsilon}|\Lambda|^{3-\varepsilon}$. A natural fractal version of this problem is to ask whether $\hdim(A+A+A)\ge 3\hdim(A)$ when $\hdim(A)$ is not too large. As an application of the Furstenberg set estimate in Theorem \ref{thm:Furstenberg-continuous}, we answer this question in the affirmative when $A$ is Ahlfors regular.

\begin{theorem}\label{thm:tripleSums} Let $s \in (0,1]$, $t \in [0,1]$. Let $A \subset \mathbb{P}$ be a non-empty $s$-Ahlfors regular set, and let $E \subset \R^{2}$ be a Borel set with $\hdim E=\pdim E=t$. Then $\hdim (A + E) \ge s+t$.
    In particular, if $s\le 1/2$, then $\hdim (A + A + A) =3s$.
\end{theorem}

The following is very standard; we include the short argument for completeness.
\begin{lemma}\label{lemma:cordoba} Assume that $I_{1},I_{2} \subset [-1,1]$ are intervals such that $\max I_{1} \leq \min I_{2} - c$ for some $c > 0$. Write $\mathbb{P}_{j} := \{(x,x^{2}) : x \in I_{j}\}$ for $j \in \{1,2\}$. Then $(z_{1},z_{2}) \mapsto z_{1} + z_{2}$ is $O(c^{-1})$-bi-Lipschitz on $\mathbb{P}_{1} \times \mathbb{P}_{2}$. 
\end{lemma}

\begin{proof}
For $x_j,y_j\in I_j$, $j=1,2$, we have
\begin{displaymath}
    \begin{pmatrix}
        x_1+x_2-y_1-y_2\\
        x_1^2+x_2^2-y_1^2-y_2^2
    \end{pmatrix}
    =
    \begin{pmatrix}
        1&1\\
        x_1+y_1&x_2+y_2
    \end{pmatrix}
    \begin{pmatrix}x_1-y_1\\x_2-y_2\end{pmatrix}.
\end{displaymath}
The matrix entries are bounded by $2$, and its determinant is at least $2c$. Its inverse therefore has norm $O(c^{-1})$. Since $x\mapsto(x,x^2)$ is bi-Lipschitz on $[-1,1]$, the claimed bound follows.
\end{proof} 

\begin{proof}[Proof of Theorem \ref{thm:tripleSums}]
We follow the reduction from parabola sums to Furstenberg sets used by the first author, Puliatti, and Py\"or\"al\"a \cite[\S1.1]{MR4956377}.

Set $F_{0}:=A+E$. For every $z\in E$, the intersection $F_{0}\cap(z+\mathbb P)$ contains the nonempty Ahlfors $s$-regular set $z+A$.

Consider the locally bi-Lipschitz invertible map $\Psi(x,y):=(x,x^{2}-y)$. For $z\in\R^{2}$, the set $\Psi(z+\mathbb P)$ lies on a line $\ell_z\in\A(2,1)$; see \cite[Proposition 2.6(b)]{MR4956377}. We have
\begin{displaymath}
    \Psi(F_0)\cap\ell_z\supset\Psi(F_0\cap(z+\mathbb P))\supset\Psi(z+A),
    \qquad z\in E.
\end{displaymath}
The map $z\mapsto\ell_z$ is locally bi-Lipschitz, so the family $\{\ell_z:z\in E\}$ has Hausdorff and packing dimension $t$. After point-line duality, Theorem \ref{thm:Furstenberg-continuous} implies that $\hdim(A+E)\ge s+t$.

For the final assertion, assume that $s\le1/2$. We can find subsets $A_1,A_2\subset A$ at positive distance from each other which are still $s$-Ahlfors regular. Lemma \ref{lemma:cordoba} shows that addition is $O(1)$-bi-Lipschitz on $A_1\times A_2$, with the constant depending on the separation of the sets. Thus 
\begin{displaymath}
    \hdim(A_1+A_2)=\pdim(A_1+A_2)=2s.
\end{displaymath}
 The first assertion, applied with $E=A_1+A_2$ and $t=2s$ (this is where we use $s\le 1/2$), yields $\hdim(A+A+A)\ge 3s$. Since addition on $A^3$ is Lipschitz and $\hdim(A^3)=3s$, the reverse inequality also holds.
\end{proof}


\def\cprime{$'$}

\end{document}